\documentclass[10pt]{article}

\usepackage{amsmath, bm, bbm}
\usepackage{amssymb,amsthm}
\usepackage[numbers]{natbib}
\usepackage{bbm}

\newtheorem{theorem}{Theorem}

\newtheorem{lemma}[theorem]{Lemma}
\newtheorem{proposition}[theorem]{Proposition}

\newcommand{\R}{\mathbb{R}}
\newcommand{\C}{\mathbb{C}}
\newcommand{\pr}{\mathbb{P}}
\newcommand{\Z}{{\mathbb{Z}}}
\newcommand{\E}{{\mathbb{E}}}
\newcommand{\N}{{\mathbb{N}}}
\newcommand{\ind}{\mathbbm{1}}

\newcommand{\vw}{v\otimes w}
\newcommand{\wv}{w\otimes v}
\newcommand{\ww}{w\otimes w}
\newcommand{\vv}{v\otimes v}
\newcommand{\fr}{\frac12}
\newcommand{\xv}{\mathbf{x}}
\newcommand{\yv}{\mathbf{y}}
\newcommand{\tf}{\text{T}(\Omega)}

\newcommand{\iv}{\bm{I}}
\newcommand{\A}{\mathbb{A}}

\newcommand{\bp}{\begin{pmatrix}}
\newcommand{\ep}{\end{pmatrix}}

\newcommand{\bea}{\begin{eqnarray}}
\newcommand{\eea}{\end{eqnarray}}
\newcommand{\beast}{\begin{eqnarray*}}
\newcommand{\eeast}{\end{eqnarray*}}

\newcommand{\rvline}{\hspace*{-\arraycolsep}\vline\hspace*{-\arraycolsep}}

\usepackage{amsmath} 
\usepackage{amssymb}
\usepackage{hyperref} 
\usepackage{authblk}
\usepackage{graphicx}
\graphicspath{ {./images/} }
\title{Representations of Hecke algebras and Markov dualities for interacting particle systems}
\author[1,2]{Alexander Povolotsky}
\author[2,1]{Pavel Pyatov}
\author[3]{ Roger Tribe}
\author[3]{ Bruce Westbury}
\author[3]{Oleg Zaboronski}
\affil[1]{Bogoliubov Laboratory of Theoretical Physics, Joint Institute for Nuclear Research, 141980 Dubna, Moscow region, Russia}
\affil[2]{National Research University, Higher School of Economics, 20 Myasnitskaya street, Moscow 101000, Russia}
\affil[3]{Department of Mathematics, University of Warwick, Coventry, CV4 7AL, UK}

\begin{document}

\maketitle
\begin{abstract}
Many continuous reaction-diffusion models on $\Z$ (annihilating or coalescing random walks, exclusion processes, voter models) admit a rich set of Markov duality functions which determine the single time distribution. 
A common feature of these models is that their generators are given by sums of two-site  idempotent operators.
In this paper, we classify all continuous time Markov processes on $\{0,1\}^{\Z}$ whose generators have this 
property, although
to simplify the calculations we only consider models with equal left and right jumping rates. The
classification leads to six familiar models and three exceptional models. 
The generators of all but the exceptional models turn out to belong to an infinite dimensional 
Hecke algebra, and the duality functions appear as spanning vectors for 
small-dimensional irreducible representations of this Hecke algebra. A second classification 
explores generators built from two site operators satisfying the Hecke algebra relations. 
The duality functions are intertwiners between configuration and co-ordinate 
representations of Hecke algebras, which results in a novel co-ordinate
representations of the Hecke algebra. The standard Baxterisation
procedure leads to new solutions of the Young-Baxter equation
corresponding to particle systems which do not preserve the number of particles.

\end{abstract}


%
\newpage
\tableofcontents
\newpage
 \section{Introduction}
\subsection{Markov duality}
Let  $\{X_t\}_{t\in \mathbb{R}_+}$ and $\{Y_t\}_{t\in \mathbb{R}_+}$ be continuous time    Markov  processes with state spaces  $\mathcal{X}$ and $\mathcal{Y}$   and $H(\cdot,\cdot):\mathcal{X}\times\mathcal{Y}\to\mathbb{R}$ be a measurable function. We say that $X_t$ and $Y_t$ are dual with respect to  $H(\cdot,\cdot)$ if for any $x\in \mathcal{X}$ and $y\in\mathcal{Y}$   the  identity 
	\begin{equation}
\mathbb{E}^X_xH(X_t,y)=\mathbb{E}^Y_yH(x,Y_t), 	\label{eq: daulity}
	\end{equation}  
holds	for any  $t\in \mathbb{R}_+$, where the expectation $\mathbb{E}^X_x$ ($\mathbb{E}^Y_y$) is with respect to the process $X_t$ ($Y_t$) conditioned on the initial state $X_0=x$ ($Y_0=y$) respectively.
 Given the generators of the processes $L_X$ and $L_Y$ that govern the  evolution of expectations of suitable test  functions, 
 \begin{equation}
\frac{d}{dt}\mathbb{E}^{X}_xF(X_t)= L_X	 \mathbb{E}^{X}_xF(X_t), \quad \frac{d}{dt}\mathbb{E}^{Y}_yG(Y_t)= L_Y	 \mathbb{E}^{Y}_yG(Y_t), \label{eq: Kolmogorov}
\end{equation}  
where $F(\cdot):\mathcal{X}\to\mathbb{R}$  and $G(\cdot):\mathcal{Y}\to\mathbb{R}$, eq. (\ref{eq: daulity}) implies
 \begin{equation}
L_X	 H(\cdot,y)(x)=L_Y	 H(x,\cdot)(y),\label{eq: daulity generators}
\end{equation}  
where 	the operator $L_X$ ($L_Y$) in lhs (rhs) acts on $ H(x,y)$ as a function of the first (second) argument respectively with the other argument fixed.
	
Formulas    (\ref{eq: daulity},\ref{eq: daulity generators}) allow one to rewrite the Kolmogorov equation for the expectation of $H(x,y)$ considered as a function of the first argument as a system of ordinary differential equations  indexed by the values of the second one. 
\begin{equation} 
	\frac{d}{dt}\mathbb{E}^{X}_xH(X_t,y)=L_Y\mathbb{E}^{X}_xH(X_t,\cdot)(y), \quad y\in \mathcal{Y}.
	\label{eq: duality ODE}
\end{equation}
Markov duality is particularly useful in situations when the original process $X$ is too complicated to analyse directly,  whereas the operator $L_Y$ is tractable (the equation (\ref{eq: duality ODE}) is useful even if a dual process $Y$ does not exist). In other words, 
it is the exact solvability of (\ref{eq: duality ODE}) that we really care about. 
Furthermore, in exceptional cases one can find sufficiently many duality functions
so that their expectations determine
the whole law of the original process or at least its fixed time distribution. 
In these cases we say that the set of Markov  dualities is complete.  

In relation to the subject of this article, interacting particle systems, one can draw a parallel between   the existence of a complete set of Markov dualities and the exact solvability of quantum particle models or quantum spin chains coined in the term ``quantum  integrability'' \cite{faddeev1996algebraic}. Loosely speaking, a system is integrable if its  evolution operator (the 
Hamiltonian for quantum models or the generator for Markov processes) belongs to a sufficiently  `large'
number of commuting operators called integrals of motion. Algebraically, this corresponds to a large  
centre of a representation of some operator algebra containing the evolution operator.  
Typically, such  representation  possesses a  high degree of symmetry.  For example, in the simplest case of a finite XXX Heisenberg  spin-$1/2$  chain the Hamiltonian is an element of the representation of the group algebra of the symmetric group on the tensor product of two-dimensional spaces.  The statement known as Schur-Weyl duality  implies that this representation commutes with the representation of $SL(2)$ Lie group on the same space.  This fact allows one to identify  the   invariant subspaces  used to construct the  eigen-basis of the Hamiltonian. However, there is a long way from the eigen-basis to physical observables, which  often makes  the information about the correlation functions of interest difficult to extract.

The Markov duality is often used as  an alternative route which  avoids the diagonalization problem and goes
directly to the observables instead. In addition, it has a natural algebraic interpretation. Namely, interpreting the generator $L_X$ as an element of the representation of some algebra,  it is natural to treat the dual operator $L_Y$ as  another representation of the same algebra with the  duality function playing
the role of the intertwiner between the two, see e.g. \cite{giardina2009duality}. 
The possibility to intertwine  the  representations is crucial for the mentioned  solvability of the problem.
A typical scenario realised for many models considered in this paper is as follows: the original model has generator
$L_X$ belonging to a representation of some algebra $\A$ of dimension $O(2^N)$, where $N\rightarrow \infty$ is the system size. 
Yet, for each $n\in \N$ there is a duality function intertwining this `large' representation to a `small' representation of $\A$
of dimension $O(N^n)$
containing the dual generator $L_Y$. 
Then it follows from (\ref{eq: duality ODE}) that the expectation of the corresponding
duality function solves a closed system of $O(N^n)$ ODE's, a small system in comparison with $2^N$ differential-difference
equations comprising the Kolmogorov equation for the original system. Probabilistically, this scenario corresponds
to a Markov duality between the original system with $O(N)$ particles and the dual system with $n$ particles, for each
$n\in \N$. The crucial element of the described construction is the existence of `small' irreducible representations of 
the algebra $\A$ containing the generator. Therefore, the natural starting point for constructing duality
functions is the search for $\A$-invariants subspaces within the original space of dimension $O(2^N)$.

Of course, exact solvability is an exceptional phenomenon, implying that the corresponding
algebra $\A$ has a special structure. One of the aims of the present paper is to understand
the origins of the exact solvability for a number of `integrable' interacting particle systems including
exclusion processes as well as systems of interacting random walks with interactions which do not preserve
the particle number (coalescence and annihilation in the presence  of branching or immigration). 
For each of these cases we will observe that the corresponding algebra $\A$ is a quotient of Hecke
algebra. In this instance the duality functions emerge as spanning vectors of `small' representations of
Hecke algebra.   
\subsection{Literature review}
Let us briefly review the previous developments in this direction. Early appearances of Markov duality in the literature go back to the middle of the last century, when it was applied   to studies of the Brownian motion \cite{levy25processus},  queuing theory \cite{lindley1952theory}, birth and death processes \cite{karlin1957classification}, etc. Since then the duality was shown to be  useful in plenty applications and  several  efforts of developing an  approach to systematic  search for Markov dualities were undertaken (see \cite{jansen2014notion} for review and references therein).

In the field of our interest,  interacting particle systems, the first application  of duality  was exploited  by  Spitzer to study the stationary measures of the symmetric simple exclusion process ([SEP]) and independent random walks \cite{spitzer1970interaction}. Later the duality  was applied to study ergodic properties of several models like [SEP], 
voter model, contact process etc \cite{liggett2012interacting}.
It is a common scenario  that  the dualities, mainly  constructed  in ad hoc manner for every particular case, allowed a reformulation of the problem of calculation of $n-$point correlation functions  in terms of $k$-particle problems in the dual system with $k\leq n$.  In particular, it was implemented in  \cite{schutz1994non,sandow1994uq,schutz1997duality}, where the duality in partially symmetric and asymmetric simple exclusion processes  on the segment with reflecting ends  followed from the invariance  of the   Hamiltonians of the quantum chains associated with the  generator of the process with respect to the  action  of the representation of    $SU(2)$   and quantum group $U_q(gl_2)$, respectively. 

An attempt   of a systematic search for dualities for the interacting particle systems going beyond a case by case consideration was undertaken in \cite{sudbury1995quantum,lloyd1997quantum}, where the  quantum mechanical algebraic language was shown efficient in  obtaining duality functions of  product form for a large family of models. However, the crucial step was the observation in \cite{schutz1994non,sandow1994uq,schutz1997duality} of  the connection  between generator symmetries and dualities.   It was used  in \cite{giardina2009duality} as a starting point to propose a new  systematic scheme for constructing   the duality functions  in interacting particle systems  within  purely algebraic framework. In a nutshell, starting from a trivial duality with the  time reversed  process, new nontrivial dualities can be generated by the action of the operators  from the representation of  a symmetry group of the generator.   This idea was then efficiently used to obtain new dualities in  interacting particle systems with one \cite{carinci2016generalized,carinci2016asymmetric,carinci2019orthogonal}  as well as  many species   \cite{belitsky2015self,kuan2016stochastic,belitsky2018self,kuan2018multi,kuan2018algebraic} of particles.
The  development \cite{kuantwo} of this idea was based on the  use of  the Schur-Weyl dauality to construct the Markov duality. The former connects  the  representation of an algebra, to which a Markov generator belongs, with a representation of its symmetry group.

Among numerous formulae for interacting particle systems based on dualities, many of them can be expressed  
in terms of  one-point or two-point correlation functions. The reason, as was mentioned above, is that the duality often allows one to recast the problem of calculating $n$-point correlation functions in the infinite particle system in terms of a dual system with no more than $n$  particles. Of course to go beyond $n\leq 2$ case, the many particle problem in dual system should be solvable in some sense.  If one wants to describe exactly the  finite time evolution, a kind of integrability should stand behind the   generator of the dual process, which potentially would make a complete set of dualities analytically accessible. There are a few recent advances in this direction.  First we mention the results of \cite{imamura2011current,borodin2014duality}, where the duality was used to obtain the integral formulas for  the so-called $q$-moments of distances traveled by tagged particles in the two models,  ASEP and q-TASEP, with specific initial conditions, using the self-duality of the former, found first in \cite{schutz1997duality}, and the dulaity of the latter with q-Boson totally asymmetric zero range process.
These results were later extended to several other models, like   q-Hahn TASEP \cite{corwin2015q}, q-Hahn ASEP \cite{barraquand2016q}, q-Hahn PushTASEP \cite{corwin2021q}, dynamic ASEP \cite{borodin2020dynamic}, and stochastic six vertex model \cite{lin2019markov}.  In a parallel development, see \cite{tz1}, \cite{tz2}, \cite{tz3} for details, 
duality was used to obtain exact solutions for a number of the so called reaction-diffusion particle systems including one-dimensional 
annihilating-coalescing, coalescing-branching random walks, and annihilating random walks with pair immigration. 
All of these models were shown to have a dual process that can be taken to be a system
of annihilating random walks 
with $n < \infty$ initial particles. This dual process turns out to have the structure
reminiscent of free fermions, so that  the $n$-particle evolution problem is solved in terms of 
Pfaffians and the fixed time
distribution of particles for each of the model problems was shown to be Pfaffian 
for all deterministic initial conditions (and a class of random initial conditions). Interestingly,
the Pfaffian structures survive the continuous limit and describe
the distribution of particles on the real line for the corresponding interacting Brownian motions.   
 \subsection{Aims of the present paper and its organisation}
Our principal aim is to investigate the common structures responsible for exact solvability of such
apparently unrelated particle systems as exclusion processes, annihilation random walks with pairwise
immigration and coalescing-branching random walks. Each of these models has a complete
set of Markov duality functions, which was discovered by independent analysis of each particle system. 

We notice that Markov generators for each of the interacting particle systems listed above is defined on the space
of test functions on the configuration space $\{0,1\}^\Z$ and has the following form,
\bea\label{intrcom1}
L=\sum_{i \in \Z}(\sigma_i-I),
\eea
where $I$ is the identity operator and $\sigma_i$ is the $2$-site operator acting on functions of particle occupation numbers
at sites $i$ and $i+1$. Moreover, each of the two-site operators is idempotent, that is
\bea\label{intrcom2}
\sigma_i^2=\sigma_i \quad \mbox{for $i \in \Z$.}
\eea
So in order to understand the commonality between the models of interest we decided to classify $all$ continuous
time Markov chains on the configuration space $\{0,1\}^\Z$ with the generator satisfying (\ref{intrcom1}, \ref{intrcom2}).

Keeping in mind the importance of duality for the models which motivated our investigation, the first classification theorem is proved 
by attempting to construct duality functions as basic vectors for `small' representations of the algebra $\A$ generated
over $\R$ by the two-site generators $\{\sigma_i\}_{i \in \Z}$, which we refer to as the {\bf generator algebra}.
The result is a list of nine inequivalent interacting particle systems, six familiar models either of exclusion or
reaction-diffusion type and three unfamiliar 'exceptional' models. Also it turns out that for all the six familiar models 
the generator algebras are certain quotients of the infinite-dimensional Hecke algebra. 
For all the familiar cases, the duality
functions are built using a similar method which, as we show separately, is  a consequence
of the braid relation satisfied by the generators of a Hecke algebra. The method itself is one 
of the secondary results of our investigation which can be described as a construction of representations
of Hecke algebras starting from the eigenvectors of the two-site generators $\sigma_i$. 
Moreover, the generator algebras for many of the reaction-diffusion systems on the list turn out to be
isomorphic and the corresponding duality functions have an identical structure, which 
we believe should be responsible for the appearance of Pfaffians - a
question which we hope to investigate in the future. 

The proof of the classification theorem consists of listing of all stochastic idempotent $4\times 4$ matrices. 
Presently, we were only able to carry such a search out under the assumption of reflection symmetry. In other
words, we only looked at particle systems with equal right and left hopping rates, leaving a
classification in the absence of such a symmetry for the future. In the mean time, we notice that quotients of the
Hecke algebra appear as the generator algebras for most of our models, even for two of the three exceptional models
for special values of their parameters. Inspired by this observation we attempt another classification theorem by listing
all Markov chains whose two-site generators obey the relations of Hecke algebra, but not assuming the left-right
symmetry. The classification is based on solving the braid relation in terms of stochastic idempotent $4\times 4$
matrices. Unfortunately, we were unable to solve the problem in full generality and were forced to use a simplifying
Ansatz by forbidding certain particle reactions. Still, the resulting list contains eleven models all of which admit
a full set of duality functions supporting our suspicion about the link between exact solvability and the 
representation theory of Hecke algebra.

At this point it is important to  mention an alternative approach to construction of Markov  dualities 
for the one-component and two-component asymmetric exclusion process, 
which  also  based on the Hecke structure of its generator algebra \cite{chen2020integrable}. The main
idea  is to use the so called q-Knizhnik-Zamolodchikov equation to construct the intertwining relation between two representations of the Hecke algebra, one in the tensor product of finite-dimensional spaces and another in the space of symmetric
polynomials. The construction of polynomial representations and the extraction of  
Markov duality functions from these representations turns out to be a highly non-trivial problem
which uses complicated tools of the theory of Macdonald  polynomials  and quantum integrable systems. In contrast our construction is completely explicit and requires only an elementary linear algebra toolbox.  
   
\textbf{Organisation of the paper.}   
In Section \ref{s1.1} we motivate the algebraic approach to the investigation
of interacting particle systems by building duality functions as bases of representations of the generator
algebra for two classical cases: annihilating random walks ([ARW]) and the asymmetric simple exclusion process ([ASEP]) 
on $\Z$. We notice that for each model, the two-site generators satisfy the relations of a Hecke algebra (reduced to
Temperley-Lieb relations in the case [ASEP]). Our approach consists of building
representations of the infinite-dimensional Hecke algebra by extending the representation of the algebra
generated by a single two-site generator. In Subsection \ref{s1.1.4} we show that such an extension 
is always possible due to the braid relation. In Section \ref{s2} we state our two main classification theorems.
In Section \ref{s3} we discuss the solvability of all the models appearing in the classification theorems. Section \ref{s4}
is dedicated to the discussion of the algebraic structures appearing in the paper: in Subsection 
\ref{s4.1} we show how to 
use duality functions to construct
irreducible representations of the quotient of a Hecke algebra corresponding to annihilating random walks on 
$\Z_N$; in Subsection \ref{s4.2} we discuss possible links between the structure of irreducible
representations of generator algebras and exact solvability;
in Subsection \ref{s4.3} we use the interpretation of duality functions as intertwiners to construct  co-ordinate
representations of Hecke algebra in the space of functions of several integer variables (the resulting Hecke
generators are expressed in terms of discrete Laplacians); in Subsection \ref{s4.4} we use the standard Baxterisation
construction to find solutions of Yang-Baxter equation which do not preserve the number of particles.
The proof of the main classification theorems can be found in Section \ref{s5}.
\section{Two familiar models}  \label{s1.1}
We start by re-examining two well understood particle systems, with well known 
Markov dualities, using the algebraic approach adopted throughout this paper. 
The models studied are continuous time Markov processes with state space 
$\Omega = \{0,1\}^\Z$. Elements of $\Omega$ 
are denoted by $\eta=(\eta(x))_{x \in \Z}$, where $\eta(x)=1$  or $\eta(x)=0$
indicates the presence of a particle or a hole at site $x \in \Z$. We write
$T(\Omega)$ for the linear space of cylinder functions $f: \Omega \to \R$, that is functions that 
depend on only finitely many co-ordinates.

The particle systems we study involve interactions only between nearest neighbour sites. Their infinitesimal 
generators can be written as
 \begin{equation} \label{genone}
\mathcal{L} f(\eta) = \sum_{n \in \Z} \, 
\sum_{\xi, \xi' \in \{0,1\}^2} R_{\xi, \xi'} 
\left( f(\eta^{(n,n+1)}_{\xi'})-f(\eta) \right) 
\ind_{\xi}(\eta(n), \eta(n+1)) 
\end{equation}
where $\eta^{(n,n+1)}_{\xi'}$ is the configuration $\eta$ with the pair $(\eta(n), \eta(n+1))$ replaced by $\xi'$
and where $R_{\xi, \xi'}$ is the rate that a pair of sites with value $\xi$ jumps to the value $\xi'$. 
We wish to re-write the infinitesimal generators of our Markov processes in a convenient 
form, and adopt some tensor notation as explained in the next Section, familiar from the study of quantum spin
chains, see e.g. \cite{korepin1997quantum} for a review.  
\subsection{Tensor notation} \label{s1.1.1}
Let $V$ denote a single copy of $\R^2$.
Let $(V_{n})_{n \in \Z}$ be a collection of copies $V_n\simeq \R^2$. 
Define
\bea
\label{vacvec}
v=\left(
\begin{array}{c}
1\\
1
\end{array}
\right) \in \R^2.
\eea
The infinite tensor product 
$T= \bigotimes_{n \in \Z} V_n$ is the linear space spanned by infinite strings $\otimes_{n\times \Z} v_n$
where $v_n \in V_n$ and $v_n=v$ for all but finitely many values of index $n \in \Z$,
modulo the usual equivalence relations:
\beast
 \ldots \otimes \alpha a_m\otimes\ldots \otimes b_n\otimes \ldots &\sim&
 \ldots \otimes a_m\otimes\ldots \otimes \alpha b_n\otimes \ldots, 
 \qquad \forall m,n \in \Z, \; \alpha
  \in \R,\\
 \ldots \otimes (a_n+b_n) \otimes\ldots&\sim&  \ldots \otimes a_n \otimes\ldots+
 \ldots \otimes b_n \otimes\ldots, \qquad\forall n\in \Z.
\eeast
The infinite-dimensional vector space $T$ is isomorphic to the cylinder functions $T(\Omega)$.
To construct such 
an isomorphism explicitly, notice first that the space of functions on $\{0,1\}$ is 
isomorphic to $V$:
\beast
f = f(1) \, \ind_1+f(0) \, \ind_0 \quad \leftrightarrow \quad
\bp
f(1)\\
f(0)
\ep
\in V.
\eeast
In particular $ \ind_1 \leftrightarrow e^{(1)}:=  \bp 1\\ 0 \ep$ and 
$ \ind_0 \leftrightarrow e^{(0)}:=  \bp 0\\ 1 \ep$. 
If $f\equiv 1$, its image under the above isomorphism is the vector $v$
defined in (\ref{vacvec}), which we now call the vacuum vector. 

The space of test functions $T(\Omega)$ is spanned by products of indicator functions
$\prod_{x \in [n,m]} \ind_{\alpha_x} (\eta(x))$. 
The claimed isomorphism between $T(\Omega)$ and $T$ is then defined via its action on
these spanning elements
\bea\label{isomorph}
\prod_{x =n}^{m} \ind_{\alpha_x} (\eta(x)) \quad \leftrightarrow \quad 
\ldots \otimes \left(\bigotimes_{x=n}^{m} e^{(\alpha_x)}_x \right)  \otimes  \ldots
\qquad \mbox{for $n\leq m$ and $\alpha_x \in \{0,1\}$}
\eea
where 
$\ldots \otimes \left(\bigotimes_{x=n}^{m} e^{(\alpha_x)}_x \right)  \otimes  \ldots $ denotes the vector with
$e^{(\alpha_x)}_x$ in position $x$, for all $n \leq x \leq m$, and 
$v$ at all other positions.  Such vectors span $T$. 
As an illustrative example of the isomorphism, the vector $\ldots \otimes e^{(1)}_0 \otimes e^{(0)}_1 \otimes e^{(1)}_2 \otimes \ldots$ is mapped
to the cylinder function $f(\eta) = \ind(\eta(0)=1,\eta(1)=0, \eta(2) =1)$. 

Using the isomorphism
(\ref{isomorph}) we can rewrite the generator (\ref{genone}) as a linear operator  $L$ on the vector space $T$ in the form 
$ L = \sum_{n \in \Z}  q_n $, 
where $q_n$ is the part of the generator corresponding to jumps at sites $(n,n+1)$. Namely we write
a $4 \times 4$ Q-matrix
\[ 
q = \bp 
\bullet & q_{(11),(10)} & q_{(11),(01)} & q_{(11),(00)} \\
 q_{(10),(11)} & \bullet & q_{(10),(01)} & q_{(10),(00)} \\
 q_{(01),(11)} & q_{(01),(10)} &\bullet & q_{(01),(00)} \\
 q_{(00),(11)} & q_{(00),(10)} & q_{(00),(01)} & \bullet 
\ep
\]
where we have written $q$ as the matrix of an operator on $V \otimes V$ in the basis
$(e^{(1)} \otimes e^{(1)}, e^{(1)} \otimes e^{(0)}, 
e^{(0)} \otimes e^{(1)}, e^{(0)} \otimes e^{(0)})$. The diagonal elements
are chosen in such a way that all row sums are equal to zero. 
Then 
\begin{equation} \label{genqform}
L = \sum_{n \in \Z}  q_n 
\end{equation}
where $q_n$ is the operator
on $T= \bigotimes_{n \in \Z} V_n$ which leaves entries $(v_k: k <n)$ and $(v_k: k > n+1)$ unchanged
and acts as $q$ on the pair $V_n \otimes V_{n+1}$. 
\subsection{Annihilating simple random walks} \label{s1.1.2}
Particles independently perform simple random walks on $\Z$, jumping right with rate $r$ and left with rate $l$, and 
particles instantaneously annihilate upon collision.  We scale time so that $r+l =1$.
The infinitesimal generator $L$ acting on functions in $T$ can be written using the form (\ref{genqform}) above, 
with
%
\[
q = \bp
-1&0&0&1\\
0&-r&r&0\\
0&l&-l&0\\
0&0&0&0\ep
\]
For our results it is convenient to write
\begin{equation} \label{sigmaone}
\sigma  = q + I = \bp
0&0&0&1\\
0&l&r&0\\
0&l&r&0\\
0&0&0&1\ep
\end{equation} 
where $I$ is the unit element of $\text{End}(V\otimes V)$. Then the generator acting on $T$ becomes
\begin{equation} \label{generator}
L = \sum_{n \in \Z} (\sigma_n - I).
\end{equation}
This trivial change is convenient because the relations satisfied by  $\{\sigma_n\}$ are in a familiar algebraic form. For example the matrix $\sigma$ satisfies $\sigma^2 = \sigma$, as is immediate from (\ref{sigmaone}).
A little harder to check is 
\begin{equation}  \label{braidreln}
\sigma_n \sigma_{n+1} \sigma_n -  Q \sigma_n=\sigma_{n+1} \sigma_{n} \sigma_{n+1} - Q \, \sigma_{n+1}, \quad
\mbox{for $n \in \Z$,}
\end{equation}
for the value $Q=rl$.
This is an identity in $\mbox{End}(V \otimes V \otimes V)$, and can be checked slowly by hand
using $8 \times 8$ matrices, or quickly by Matlab. 
This relation is a deformation of the braid relation which, together with the commutativity relation
\begin{equation} \label{commutereln}
\mbox{$\sigma_n \sigma_m = \sigma_m \sigma_n \quad$ when $|n-m|>1$,}
\end{equation}
and the quadratic relation 
\begin{equation} \label{quadraticreln}
\sigma_n^2=\sigma_n, \quad \mbox{for $n \in \Z$,}
\end{equation}
are the defining relations for a Hecke algebra with parameter $Q$
(the Appendix gives some details of Hecke algebras).
Thus, for annihilating simple random walks, the algebra $\A$ generated by 
$\{\sigma_n\}_{n \in \Z}$ over $\R$ in $\otimes_{i \in \Z} V_i$
is a quotient of a Hecke algebra with parameter $Q=rl$. In Section \ref{s4.1}
we conjecture the exact quotient for the model on $N$ sites $\{1,2,\ldots,N\}$. 

We now study some representations of the generator algebra $\A$, which  
becomes a surprisingly straightforward task due to the structure of 
eigenvectors of $\sigma$. Since $\sigma^2=\sigma$ 
the eigenvalues of $\sigma$ are either zero or one. The row sums of
$\sigma$ are all $1$, and the therefore one of the eigenvectors with eigenvalue $1$ is
\[
\bp 1\\1\\1\\1\ep =v\otimes v,
\]
recalling that $v=(1,1)^T$. 
As $\mbox{rank}(\sigma)=2$, we expect that there might be a linearly independent eigenvector with eigenvalue $1$. It is straightforward to 
check that such an eigenvector for $\sigma$ exists and can be written in a factorised form $w \otimes w$, where 
\[
w=\bp -1\\1\ep.
\]
Another computation shows that
\begin{equation} \label{opaction}
\sigma v\otimes w=l v\otimes v+  rw\otimes w, \qquad
\sigma w\otimes v=r v\otimes v+l w\otimes w.
\end{equation} 
This allows us to extend a representation 
of $\sigma_n$ on $V_n \otimes V_{n+1}$ to a representation of $\A$ on $\otimes_{n \in \Z} V_n$ as follows. 
Concretely, (\ref{opaction}) implies that  $\A$ acts on certain subspaces of $\otimes_{n \in \Z}V_n$
spanned by the tensor product of runs of $w$'s and $v$'s. For example, consider
\bea
f_x=\left(\otimes_{i \leq x} w_i \right) \otimes \left(\otimes_{j >x}v_j \right) \in \bigotimes_{n \in \Z}V_n.  
\eea
(We temporarily ignore that $f_x$ is not eventually constant and so not an element of $T$.)
Using (\ref{opaction}), together with $\sigma v \otimes v = v \otimes v $ and 
$\sigma w \otimes w = w \otimes w$, we see that
\[
\sigma_x f_x = r f_{x-1}+l f_{x+1}
\quad \mbox{and} \quad \sigma_y f_x = f_x \quad \mbox{for $y \neq x$.}
\]
Thus the algebra $\A$ is represented on the subspace $\mbox{Span}_\R(f_x, x \in \Z)$. Moreover
\bea\label{dislap}
Lf_x=rf_{x-1}+lf_{x+1}-f_x :=\Delta^{r,l} f_x,
\eea
defining $\Delta^{r,l}$ a non-symmetric discrete Laplacian (we write $\Delta$ in the symmetric case $r=l=\frac12$).

More generally, let $x_1\leq x_2\leq x_3\leq \ldots \leq x_{2n}$ for $n \geq 1$ and
\begin{equation} \label{fdualityfn}
f^{(2n)}_{x_1, x_2, \ldots, x_{2n}}=\ldots \left(\otimes_{x_1+1\leq j_1 \leq x_2} w_{j_1}\right)\ldots  
\left(\otimes_{x_3+1\leq j_2 \leq x_4} w_{j_2}\right)\ldots \left(\otimes_{x_{2n-1}+1\leq j_{n} \leq x_{2n}} w_{j_n}\right) \ldots
\end{equation}
where each $\ldots$ represent tensoring with the vacuum vector $v$ at all other positions. 
We also let $f^{(0)} = \otimes_{i \in \Z} v_i$ (corresponding to the constant function with value $1$). 
These vectors do lie in $T$. The action 
 (\ref{opaction}) again implies that $\sigma_x f^{(2n)}_{x_1, x_2, \ldots, x_{2n}}$ can be expressed either in terms
 of the same type of function $f^{(2n)}$, or it may produce  a term where a run of $w$'s disappears or one where
 two runs of $w$'s merge, both of which can be expressed as $f^{(2n-2)}$. For example
 \[
\sigma_3 f^{(4)}_{1346} = r f^{(4)}_{1246} + l f^{(4)}_{1446} = r f^{(4)}_{1246} + l f^{(2)}_{16}.
 \]
 Furthermore
\bea\label{dual1}
Lf^{(2n)}_{x_1, x_2, \ldots, x_{2n}}=\sum_{k=1}^{2n} \Delta^{r,l}_{k} f^{(2n)}_{x_1, x_2, \ldots, x_{2n}},
x_1<x_2<\ldots <x_{2n},\\\label{dual2}
f^{(2n)}_{x_1, \ldots, x_i=x_{i+1},\ldots, x_{2n}}=f^{(2n-2)}_{x_1, \ldots, x_{i-1},x_{i+2},\ldots, x_{2n}},
~i=1,\ldots,2n-1,
\eea
where $\Delta^{r,l}_k$ is the discrete Laplacian applied to the $k$-th argument.
 
From the algebraic point of view, we have constructed a lower triangular representation of $\A$ in $\otimes_{i \in \Z} V_i$:
\[
T_{2n}  :=  \mbox{Span}_{\R}(f^{(2n)}_{x_1, x_2 \ldots x_{2n}}: x_1<x_2<\ldots <x_{2n}), 
\qquad \A T_{2n} \subset T_{2n} \oplus T_{2n-2}.
\]
From the probabilistic point of view, (\ref{dual1}, \ref{dual2}) mean that $f^{(2n)}$, regarded as a function on
$\Z^{2n}\times \Omega$, is a Markov duality function, and establishes a duality with
annihilating random walks in reverse time started with $2n$ particles. The duality function $f^{(2n)}$ is mapped under
the isomorphism (\ref{isomorph}) to the well known product spin function
\[
f^{(2n)}_{x_1,x_2,\ldots,x_{2n}} \quad \leftrightarrow \quad \prod_{1 \leq i \leq n} (-1)^{\eta(x_{2i-1},x_{2i}]}
\quad \mbox{where} \quad \eta(a,b] = \sum_{x=a+1}^b \eta(x).
\]
See \cite{tz1} for the use of these duality functions, where the 
lower-triangular system of differential equations for  $ \Phi^{(n)}(x_1,\ldots,x_{2n}) :=
 \E[ f^{(2n)}_{x_1,x_2,\ldots,x_{2n}} (\eta_t)]$ arising from the generator action in 
(\ref{dual1}, \ref{dual2}) is solved, for deterministic initial conditions, 
using Pfaffians, and the resulting formulae are shown to characterise the distribution of $\eta_t$,
for any fixed $t>0$, 
as a Pfaffian point process on $\Z$. 

We call a duality function in the form (\ref{fdualityfn}) an {\it alternating interval} duality function. 
The key facts that allowed it to work were (i) the existence of a second factorised eigenvector
$w \otimes w$ and (ii) the fact that in the action (\ref{opaction}) of $\sigma$, the right hand side of $\sigma v\otimes w$ (respectively of $\sigma w\otimes v$) does not contain a term proportional to  $w\otimes v$ (respectively $v\otimes w$).
We show in Section \ref{s1.1.4} below that the second fact is sometimes a consequence of the braid relation in a Hecke algebra.
%
\subsection{Asymmetric simple exclusion process} \label{s1.1.3}
%
In this process, particles jump to the right at rate $r$, or left at rate $l$, but the jump is suppressed unless
the jump is onto an unoccupied site. We suppose $r+l=1$, and also that $r>l$.
The generator $L$ on $T$ is still in the form (\ref{generator}) with
\bea \label{genasep}
\sigma=\bp
1&0&0&0\\
0&l&r&0\\
0&l&r&0\\
0&0&0&1
\ep.
\eea
Again it is possible to check that the operators $\{\sigma_n\}$ satisfy the three relations
 (\ref{braidreln}), (\ref{commutereln}) and (\ref{quadraticreln}) of a Hecke algebra.
 Indeed the relations of the Temperly-Lieb algebra are satisfied (which imply the Hecke relations) namely
 \begin{equation} \label{TL}
 \sigma_n \sigma_{n+1} \sigma_n -  rl \, \sigma_n \; = 0 = \; \sigma_{n+1} \sigma_{n} \sigma_{n+1} - rl \, \sigma_{n+1}, \quad
\mbox{for $n \in \Z$.}
 \end{equation}
 
We can again use the eigenvalues of $\sigma$ to understand the known Markov duality functions.
It is easy to check that $\sigma w\otimes w=w\otimes w$ for any $w \in V$.
In other words, any tensor square of a non-zero two-dimensional vector is an eigenvector of 
$\sigma$ with eigenvalue $1$. This is to be contrasted with the case of annilating random walks,
where the requirement that $w \otimes w$ is an eigenvector with eigenvalue one essentially fixes a single
$w$ that is independent of $v$. However we aim to 
choose $w$ from the requirement that $\sigma$ has a good action on $v\otimes w$
in the sense that
\bea\label{asep_fe}
\sigma v \otimes w=\alpha v\otimes v+\beta w\otimes w+\gamma v \otimes w. 
\eea
There are no non-zero solutions with $w_2=0$. Therefore, 
we can search for the answer in the form $w=(q,1)^T$, $q \neq 1$.
The only solution to (\ref{asep_fe}) for such an Ansatz is
\bea
q=\frac{l}{r},~\alpha=l,~\beta=r, \gamma=0.
 \eea 
In contrast to the annihilating random walk example, $\sigma w\otimes v$ cannot then be expressed in terms
of $v\otimes v$,  $w\otimes w$ and $w\otimes v$ only, so the representations of the algebra $\A$ generated
by $\{\sigma_n\}$ has to be constructed differently.

Fortunately, there is an additional algebraic 
structure, which comes to the rescue. For any $a=(a_1,a_2)^T, b=(b_1, b_2)^T \in V$, we use the linear operation
of pointwise multiplication (the Hadamard product), which corresponds to the multiplication of functions on $\Omega$:
\bea
a\cdot b:=\left(\begin{array}{c}
a_1b_1\\
a_2b_2
\end{array}
\right) \in V
\eea
and where powers are defined by $a^k = a \cdot a \cdots a$. 
Similarly, for $\otimes_{n \in \Z} a_n,  \otimes_{n \in \Z} b_n \in \otimes_{n\in Z} V$,
we define $\otimes_{n \in \Z} a_n \cdot \otimes_{n \in \Z} b_n:=\otimes_{n \in \Z} 
(a_n\cdot b_n)$. The key facts we need  concerning
$\sigma$, $v$ and $w$, alongside
$\sigma v \otimes v=v \otimes v$ and $\sigma w \otimes w=w\otimes w$, are
\begin{eqnarray}
\label{asep_action3}
\sigma w^n \otimes w^{n+1}&=&lw^n \otimes w^{n}+rw^{n+1} \otimes w^{n+1}, \quad n\geq 0,\\
\label{asep_action4}
0&=&rw^{n+2}-w^{n+1}+lw^n, \quad n\geq 0,
\end{eqnarray}
the latter arising from the relation $rq^2-q+l=(rq-l)(q-1) =0$.
Now we define
\begin{eqnarray}
h_x & = & \ldots \otimes v_{x-2}\otimes v_{x-1}\otimes w_{x} \otimes w_{x+1}\otimes \ldots
\quad \mbox{for $x \in \Z$}\\
h^{(n)}_{x_1,\ldots,x_n} & = & h_{x_1} \cdot \ldots \cdot h_{x_n} \quad \mbox{for $x_1\leq x_2
\leq \ldots 
\leq x_n$.}
\end{eqnarray}
We check below that the action in (\ref{asep_action3}) implies
that these are duality functions, and we call $h^{(n)}_{x_1,\ldots,x_n} $ a {\it staircase} duality function. The corresponding functions on the configuration space are
$q^{\eta[x,\infty)} = q^{\sum_{y \geq x}\eta(y)}$
and products $\prod_i q^{\eta[x_i,\infty)}$, which are the duality
functions developed in Schutz \cite{schutz1997duality}.
Alas, these functions do not depend on finitely many co-ordinates, and 
the vector  $h_x$  does not lie in our tensor space $T= \bigotimes_{n \in \Z} V_n$.
To bypass this difficulty, we follow \cite{schutz1997duality}
by restricting to configurations $\eta$ that have a rightmost particle. 
Just for this Section we adopt a different definition of the infinite tensor product.
Let $\Omega'$ be a reduced configuration space containing only 
configurations with a rightmost particle: $\eta\in \Omega'$ if there is $j_0 \in \Z:$ $\eta=((\eta_j)_{j\leq j_0},0,0,\ldots)$. 
Next, let $T(\Omega')$ be the space of functions on $\Omega'$ depending on semi-infinite strings of
arguments, that is $f \in T(\Omega')$ if there is $i_0\in \Z$ such that $f$ depends on $(\eta_i)_{i\geq i_0}$ only. The action of the generator is well-defined on $T(\Omega')$: for any $\eta \in \Omega'$ and $f \in 
T(\Omega')$
\[
Lf(\eta)=\sum_{k=i_0-1}^{j_0}(q_k f)(\eta).
\]
To model $T(\Omega')$ using vector notation, we adapt the definition of the infinite tensor product to 
let $T'$ 
be the linear space spanned by infinite strings $\otimes_{n\times \Z} \, v_n$,
where $v_n \in V_n$, for which there exists $i_0$ so that
$v_n=v$ for all $n < i_0$ (modulo the usual equivalence relations).
%
%
As before, the linear spaces $T(\Omega')$ and $T'$ are isomorphic.
Using (\ref{asep_action3}), 
we can check that the algebra $\A$ generated by $\{\sigma_n\}$ is represented in each of the spaces
\bea\label{asepis}
T_n' =\text{Span}_{\R} \left(h^{(n)}_{x_1,\ldots,x_n}, \; x_1\leq x_2
\leq \ldots 
\leq x_n \right)\subset T', \quad \mbox{for 
$n=0,1,\ldots.$}
\eea
Indeed it follows from (\ref{asep_action3}) and $\sigma w^n \otimes w^n = w^n \otimes w^n$
that for $x_1<x_2<\ldots <x_n$,
\begin{eqnarray}
\sigma_{x_i} h^{(n)}_{x_1,\ldots,x_n}  & = & 
h_{x_1} \cdot \ldots 
\left(l h_{x_i+1}+r h_{x_i-1} \right)
\ldots \cdot h_{x_n}
\quad \mbox{for $1\leq i \leq n,$} \\
\sigma_{k}  h^{(n)}_{x_1,\ldots,x_n}  & = &  h^{(n)}_{x_1,\ldots,x_n} \quad
\mbox{for $k \not\in \{x_1,x_2,\ldots , x_n\}.$}
\end{eqnarray}
It remains to check that $\A$ maps $h_{x_1,\ldots,x_n}$ to an element of
$T_n'$ if $x_i$'s are allowed to coincide.  Such a check is elementary using the relation
\bea\label{wnm}
\sigma(w^n \otimes w^{n+m})=\frac{q^2-q^m}{q^2-1}w^n\otimes w^n+
\frac{q^m-1}{q^2-1} w^{n+1}\otimes w^{n+1},~m,n\geq 0
\eea
which is proved using (\ref{asep_action3}) and the identity 
$w^m=\frac{q-q^m}{q-1}v+\frac{q^n-1}{q-1}w$. 

From the probabilistic point of view, we must examine the action of the entire generator:
\bea\label{dual3}
L h^{(n)}_{x_1,\ldots,x_n}  = \sum_{k=1}^{n} \Delta^{r,l}_{k} h^{(n)}_{x_1,\ldots,x_n} ,
\quad x_1<x_2<\ldots <x_{2n},\\ 
\label{dual4}
r h^{(n)}\mid_{x_i=x,x_{i+1}=x} + l h^{(n)}
\mid_{x_{i}=x+1,x_{i+1}=x+1} -
h^{(n)}_{ x_i=x,x_{i+1}=x+1}=0, &1\leq i<n.
\eea
The operator on the right hand side of (\ref{dual3}) is the generator of [ASEP] with $n$
particles in the co-ordinate representation. 
Notice that the knowledge of all $\Phi^{(n)}_t (x_1,\ldots,x_n) := \E( h^{(n)}_{x_1,\ldots,x_n})$ 
for $n \geq 0$ determines the law of [ASEP] at time $t$, showing
the system of duality functions constructed above is complete. The solution of the system
for $\Phi^{(n)}_t (x_1,\ldots,x_n)$ in
Borodin, Corwin and Sasamoto  \cite{borodin2014duality}
leads eventually to Tracy-Widom distribution fluctuations for [ASEP].

\subsection{The Hecke relations and duality} \label{s1.1.4}
%
Another commonly arising type of duality function is of the form 
\begin{equation} \label{pmdf}
\ldots \otimes w_{x_1} \otimes \ldots \otimes w_{x_n} \otimes \ldots  
\end{equation}
where the vector $w$ is in positions $x_1,x_2, \ldots, x_n$ and each $\ldots$ represent tensoring with the vacuum vector $v$ at all other positions. When $w=e^{(1)}$ this corresponds to the {\it product moment} duality function:
\begin{equation} \label{dualitytype3}
g^{(n)}_{x_1, x_2, \ldots, x_{n}}=\ldots \otimes e^{(1)}_{x_1} \otimes \ldots \otimes e^{(1)}_{x_n} \otimes \ldots  
\quad \leftrightarrow \quad \prod_{1 \leq i \leq n} \eta(x_i),
\end{equation} 
and when $w = (\beta,1)^T$ for some $\beta \neq 1$ this corresponds to the {\it exponential product moment} duality function:
\begin{equation} \label{dualitytype4}
g^{(n)}_{\beta; x_1, x_2, \ldots, x_{n}}=\ldots \otimes w_{x_1} \otimes \ldots \otimes w_{x_n} \otimes \ldots  
\quad \leftrightarrow 
\quad \beta^{\sum_{i=1}^n \eta(x_i)}.
\end{equation} 
We call the number $n$ 
of non vacuum entries the order of the product moment. 
These are well known duality functions for the symmetric exclusion process and certain voter models (see Section \ref{s3}). The key relation that needs to hold for these to be duality functions
is that $\sigma \vw$ and $\sigma \wv$ can be expressed in terms of $\vv,\vw,\wv$, that is without using $\ww$. This ensures that $\sigma_y$ can only reduce the number of sites in 
a product moment. This leads, as in the two example above, to invariant subspaces for the generator algebra, which 
here consist of
product moments of at most a given order. 

Our aim in this paper is to examine whether there are underlying algebraic facts that 
lead to the emergence of the duality functions we have seen above.
The braid relation (\ref{braidreln}) emerges as a surprise in the two models described above.
The following proposition shows that, when present together with two factorised form eigenvectors,
the Hecke relations imply that the generator will act in a suitable way to construct duality functions.

\begin{proposition}
Suppose that $\{\sigma_n\}_{n \in \Z}$, where $\sigma_n$ is a linear operator $\sigma \in \text{End}(V\otimes V)$ 
acting on entries $V_n \times V_{n+1}$ in the space $\otimes_{n \in Z} V$, satisfy the deformed 
braid and 
quadratic Hecke relations (\ref{braidreln}) and (\ref{quadraticreln}).  
Suppose also that $\dim \ker (\sigma-I)=2$ and that there are two independent
factorised form eigenvectors $\sigma v\otimes v=v \otimes v$ and $\sigma w\otimes w=w\otimes w$.

Then the action of $\sigma$ on $v\otimes w$ and $w \otimes v$ must take one of the following two forms:
\bea\label{prop_litar1}
\left\{
\begin{array}{c}
\sigma v\otimes w= \alpha \vv+\beta \ww+\gamma \vw, \\
\sigma w\otimes v= \tilde{\alpha} \ww+\tilde{\beta} \vv+\tilde{\gamma} \wv,
\end{array}
\right.
\eea
or
\bea\label{prop_litar2}
\left\{
\begin{array}{c}
\sigma \vw=\gamma \vw+\delta \wv, \\
\sigma \wv=\tilde{\gamma} \wv+\tilde{\delta} \vw,
\end{array}
\right.
\eea
for some $\alpha, \beta, \gamma, \delta, \tilde{\alpha}, \tilde{\beta}, \tilde{\gamma}, \tilde{\delta}
\in \R$.
\end{proposition}
The form (\ref{prop_litar1}) is consistent with alternating interval duality functions, and the form
 (\ref{prop_litar1}) is consistent with product moment duality functions. 

\begin{proof} For this proof, the argument is easier using the Q-matrix form, that is
$q = \sigma -I$.  A short computation shows that the Hecke relations (\ref{braidreln}) and
(\ref{quadraticreln}) become
\beast
q_i q_{i+1} q_{i}- (Q+1) q_i & = & q_{i+1} q_{i} q_{i+1} - (Q+1) q_{i+1}, \\
q_i^2 & = &-q_i,  \quad i \in \Z.
\eeast
Expressed in terms of $\{q_i\}$ our assumptions mean that $q v\otimes v=q w\otimes w =0$ and
$\dim \ker (q)=2$. 

The conclusion will follow after analysing the action of $L_1:=q \otimes Id$ and $L_2=Id \otimes q$
on the vectors $v\otimes w \otimes w, v\otimes v \otimes w \in V_1\otimes V_2 \otimes V_3$.
By definition $L_1:=L\otimes Id$ and $L_2=Id\otimes L$ act non-trivially on $V_1\otimes V_2$ and $V_2\otimes V_3$
respectively.  Applying the deformed braid relation one finds
\bea\label{aoact}
\left\{
\begin{array}{ccc}
L_1 L_2 L_1 v\otimes w\otimes w &=& (Q+1) \, L_1 v\otimes w\otimes w\\
L_2L_1L_2 v\otimes v\otimes w&=& (Q+1) \, L_2 v\otimes v\otimes w
\end{array}
\right.
.
\eea
As $\{ v\otimes v,w\otimes w,v\otimes w,w\otimes v\}$ is a basis of $V\otimes V$,
there are $\alpha, \beta, \gamma, \delta$ such that
\bea\label{aobasis}
L v\otimes w=\alpha v\otimes v+\beta w\otimes w +\gamma v\otimes w+\delta w\otimes v
\eea
Substituting (\ref{aobasis}) into (\ref{aoact}) and observing that 
$a\otimes v+b\otimes w=0$ implies
$a=b=0$, one finds
 \bea\label{aoact2}
\left\{
\begin{array}{c}
\alpha \delta q \left(v\otimes w +w\otimes v\right)=0, \\
\beta \delta q \left(v\otimes w+w\otimes v\right)=0.
\end{array}
\right.
\eea
Notice that the vectors $v\otimes v$, $w\otimes w$ and $v\otimes w+w\otimes v$ are linearly
independent. As $qv\otimes v=q w\otimes w=0$ and $\text{null}(q)=2$, we conclude that $\delta=0$
or $\alpha=\beta=0$. Repeating the argument for $q w\otimes v$ we arrive at the following 
possibilities
\bea\label{bigar}
\left\{
\begin{array}{c}
qv\otimes w= \alpha \vv+\beta \ww+\gamma \vw \\
OR\\
q\vw=\gamma \vw+\delta \wv,
\end{array}
\right.
\mbox{and} \quad
\left\{
\begin{array}{c}
qw\otimes v= \tilde{\alpha} \ww+\tilde{\beta} \vv+\tilde{\gamma} \wv \\
OR\\
q\wv=\tilde{\gamma} \wv+\tilde{\delta} \vw.
\end{array}
\right.
\eea
The four possible actions described here can be reduced to just the two stated in the proposition. 
Indeed, assume that
\bea\label{alt1}
q\vw=\gamma \vw+\delta \wv,
\eea
but
\bea\label{alt2}
qw\otimes v= \tilde{\alpha} \ww+\tilde{\beta} \vv+\tilde{\gamma} \wv.
\eea
Applying $q$ to (\ref{alt1}) and using the relation $q^2=-q$ and (\ref{alt2}), one finds
\beast
-\gamma \vw-\delta \wv=\gamma (\gamma \vw+\delta \wv)+\delta (\tilde{\alpha}\ww+\tilde{\beta}\vv
+\tilde{\gamma}\wv).
\eeast
It then follows from linear independence that either $\delta=0$ or $\tilde{\alpha}=\tilde{\beta}=0$.
In the former case (\ref{alt1}), (\ref{alt2}) become a particular case of (\ref{prop_litar1}),
in the latter of (\ref{prop_litar2}). 
%
\end{proof}
Let us notice that the proposition is applicable to [ARW] as well as all other reaction-diffusion systems
considered below: for all these cases the non-degeneracy condition $\dim \ker(\sigma-I)=2$
is satisfied and the corresponding pair of factorised null vectors of $\sigma-I$ exists. For 
the reaction-diffusion systems the alternative (\ref{prop_litar1}) is realised, which 
allows one to construct a full system of the alternating interval duality functions, see
Section \ref{s1.1.2} for the example of annihilating random walks. The  alternative
(\ref{prop_litar2}) corresponds to duality functions of the product moment type. It is well known
that products
of particle indicators constitute a full system of duality functions for the symmetric
simple exclusion process ([SEP]). However notice that the non-degeneracy condition is not
satisfied for [SEP] as $\dim \ker(\sigma-I)=1$. 

The above proposition motivates a systematic search for operators $\sigma$ for which the family 
$\{\sigma_n\}$
satisfy the Hecke algebra relations (\ref{braidreln}),(\ref{quadraticreln}), and which have the constraints that allow 
$\sum_n (\sigma_n-I)$ to act as the generator of a particle system on $\Z$: that the off diagonal entries of 
$\sigma$ are non-negative and that row sums equal $1$.  As an example, there is exactly one family of such operators which 
have rank $1$, namely:
\begin{equation} \label{rmtemp}
\sigma = \bp 
t^2 & t(1-t) & t(1-t) & (1-t)^2 \\ 
t^2 & t(1-t) & t(1-t) & (1-t)^2 \\ 
t^2 & t(1-t) & t(1-t) & (1-t)^2 \\ 
t^2 & t(1-t) & t(1-t) & (1-t)^2 \\ 
\ep \qquad \mbox{for $t \in [0,1]$}. 
\end{equation}
which satisfies the braid relation (\ref{braidreln}) with $Q=0$. 
To prove this is the only rank $1$ solution we may set $\sigma = (1,1,1,1)^T(a_1,a_2,a_3,a_4)$ since we know
that its image must contain the eigenvector $(1,1,1,1)^T$. We require $a_1+a_2+a_3+a_4=1$ and
$a_i \geq 0$ for $i=1,..,4$. This already ensures that $\sigma^2=\sigma$. 
Substituting this form into the deformed braid relation (\ref{braidreln}) we find, using a symbolic algebra package,
that it can only hold for the family above. More details are given at the start of Section \ref{s5}. 

A full search for all examples is still unfinished.
The results in Section \ref{s2} give a partial classification, based on adding certain 
extra requirements, of the set of rank $2$ and rank $3$ examples. We find
exactly 11 families, among them many familiar models and just a few new ones,
such as (\ref{rmtemp}). The particle system corresponding to (\ref{rmtemp})
we call a {\em reshuffle model}, as pairs of neighbouring sites
update their values at rate one, becoming $11$, $10$, $01$, $00$ with fixed probabilities 
independent of their previous values. We study this model in Section \ref{3.6} where 
we show that its invariant measure is a determinantal point process on $\Z$ related to the 
{\it descent process} introduced by Borodin, Diaconis and Fulman in \cite{borodin2010adding}.
%
%
\section{Classifications of generators} \label{s2}
\subsection{Generators built out of projectors} \label{s2.1}
We will make two reductions in our first classification.

\textit{Reflective symmetry.} Let  $\rho \in \text{End}(V \otimes V)$ be defined by
$\rho (a\otimes b)=b\otimes a$, so that in the standard basis $\rho$ is given by  
\bea \label{symmetry}
\rho =   \bp  1&0&0&0\\0&0&1&0\\0&1&0&0\\0&0&0&1 \ep.
\eea 
If a generator of a Markov process is of the form (\ref{generator}) and satisfies $ \rho \sigma = \sigma  \rho$
then the reflected Markov chain $(\tilde{\eta}_t)$, defined
by $\tilde{\eta}_t(x)= \eta_t(-x)$ for $x \in \Z$, has the same law as $(\eta_t)$.

\textit{Particle-hole interchange.}
The map $\eta_x \to 1- \eta_x$ for $x \in \Z$ interchanges occupied and empty sites.
For a Markov process of the form (\ref{generator}) the induced Markov chain after this map
is of the same form but with  $\sigma$ replaced by the conjugation $\tilde{\sigma} = \tau \sigma \tau$ where 
$\tau \in \text{End}(V \otimes V)$ is given in the standard basis by
\bea \label{holeparticle}
\tau = \bp 0&0&0&1\\0&0&1&0\\0&1&0&0\\1&0&0&0 \ep.
\eea 
We note that $\rho^2 = I$, $\tau^2 =I$ and $\rho \tau = \tau \rho$.

 \begin{theorem} \label{T1}
Suppose $\sigma \in \text{End}(V \otimes V)$ is idempotent,
 that is $\sigma^2=\sigma$, and satisfies the stochastic constraints, that is its matrix in the standard 
 basis satisfies
 \begin{equation} \label{stochasticity}
\mbox{$\sigma_{ij} \geq 0$ for $1 \leq i \neq j \leq 4$, and $\sum_{j=1}^4 \sigma_{ij}=1$ for $1\leq i \leq 4$.}
\end{equation}
 Assume also the reflective symmetry condition $ \rho \sigma =\sigma \rho$.
Then any such non-trivial $\sigma \neq I$ must, after possibly a particle-hole 
conjugation $\tau \sigma \tau$, lie in one of 
 following nine families (written in the standard basis) 
\bea
   \begin{array}{c}
 \mbox{Symmetric exclusion process}\\ \mbox{[SEP]} \end{array} \quad & 
 \sigma = \bp 1 &0&0& 0 \\0&\frac12&\frac12&0\\
 0&\frac12&\frac12&0\\ 0 &0&0&1 \ep &
 \quad \label{sep} \\
  \begin{array}{c}
 \mbox{Biased voter model}\\ \mbox{[BVM]} \end{array}
  \quad & 
\sigma = \bp 1&0&0&0\\ \theta&0&0&1-\theta\\  \theta&0&0&1-\theta\\0&0&0&1 \ep &
 \quad   \mbox{for $\theta \in [0,1]$,}\label{bvm} \\
    \begin{array}{c}
  \mbox{Symmetric anti-voter model}\\\mbox{[SAVM]} \end{array}  \quad & 
 \sigma = \bp  0&\frac12&\frac12&0\\0&1&0&0\\0&0&1&0\\0&\frac12&\frac12&0 \ep &
 \quad  \label{savm} \\
  \begin{array}{c}
\mbox{Annihilating-coalesing} \\ \mbox{symmetric random walks}\\ \mbox{[ACSRW]} \end{array} \quad & 
 \sigma = \bp 0&\frac{1-\theta}{2}&\frac{1-\theta}{2}&\theta \\ 0&\frac12&\frac12&0\\ 
 0&\frac12&\frac12&0\\0&0&0&1 \ep &
 \quad \mbox{for $\theta \in [0,1]$,} \label{acsrw} \\
   \begin{array}{c}
\mbox{Coalesing symmetric random} \\ \mbox{walks with branching}\\ \mbox{[CSRWB]} \end{array} \quad & 
\sigma =  \bp 1-2 \theta&\theta&\theta&0\\1-2 \theta&\theta&\theta&0\\
1-2 \theta&\theta&\theta&0\\0&0&0&1 \ep &
 \quad \mbox{for $\theta \in [0,\frac12]$,} \label{csrwb} \\
   \begin{array}{c}
\mbox{Annihilating symmetric random} \\ \mbox{walks with pair immigration}\\ \mbox{[ASRWPI]} \end{array} \quad & 
 \sigma = \bp \frac12 - \theta &0&0& \frac12 + \theta \\0&\frac12&\frac12&0\\
 0&\frac12&\frac12&0\\ \frac12 - \theta &0&0& \frac12 + \theta  \ep &
 \quad \mbox{for $\theta \in [0,\frac12]$,} \label{asrwpi} \\
 \begin{array}{c}
 \mbox{Stationary coalescence}\\ \mbox{annihilation model} \\
  \mbox{[SCAM]} \end{array} \quad & 
 \sigma = \bp 0&\theta&\theta&1-2 \theta\\0&1&0&0\\0&0&1&0\\0&0&0&1 \ep &
 \quad \mbox{for $\theta \in [0,\frac12]$,} \label{scam} \\
 \begin{array}{c}
 \mbox{Dimer model}\\ \mbox{[DM]} \end{array}\quad & 
 \sigma = \bp 1-\theta&0&0&\theta\\0&1&0&0\\0&0&1&0\\1-\theta&0&0&\theta \ep &
 \quad \mbox{for $\theta \in [0,1]$,} \label{dm} \\
 \begin{array}{c}
 \mbox{Reshuffle model}\\ \mbox{[RM]} \end{array} \quad & 
 \sigma = \bp \theta_1&\theta_2&\theta_2&\theta_3\\  \theta_1&\theta_2&\theta_2&\theta_3\\
 \theta_1&\theta_2&\theta_2&\theta_3\\  \theta_1&\theta_2&\theta_2&\theta_3 \ep &
  \begin{array}{l} \mbox{for $\theta_i \geq 0$ and with }\\
 \mbox{$\theta_1 + 2 \theta_2 + \theta_3 = 1$.} \end{array} \label{rm} 
 \eea
  \end{theorem}
 There are certain obvious redundancies in this listing: for example under particle-hole conjugation the
models [BVM],  [DM] and [RM] change their parameter values. 
 There are certain less obvious overlaps in this listing: for example 
 the model [BVM] at $\theta = 1$ equals [CSRWB] at $\theta = 0$,
and the model [SCAM] at  $\theta = 0$ equals [DM] at $\theta = 1$.
 
 We use the above operators to form an 
 algebra $\A$ generated by $\{\sigma_n\}_{n \in \Z}$ over $\R$ in $\otimes_{i \in \Z} V_i$
 as in Section \ref{s1.1.2}. We can now check for which of the models in the classification
 the three conditions for a Hecke algebra are satisfied, which here reduces to checking the deformed braid relation
 (\ref{braidreln}).
 \begin{lemma} \label{temp3}
 The algebras corresponding to the models [SEP],  [SAVM] and [ACSRW], defined in (\ref{sep}), (\ref{savm}), and
 (\ref{acsrw}),  satisfy the deformed braid relation 
(\ref{braidreln}) with parameter $Q=1/4$. 

 The algebra corresponding to [BVM] (\ref{bvm}) 
 satisfies the deformed braid relation with parameter $Q=\theta(1-\theta)$;  
 that corresponding to [CSRWB] (\ref{csrwb}) with parameter $Q=\theta(1-\theta)$;  
 that corresponding to [ASRWPI] (\ref{asrwpi}) 
 with parameter $Q=\theta^2$. 

The remaining three models [SCAM], [DM], [RM] do not generically satisfy the deformed 
braid relation, but at 
some exceptional values do so: namely 
[DM] (\ref{dm}) at $\theta = 1/2$ with parameter $Q=1/4$ and 
[RM] (\ref{rm})  when $\theta_1=t^2, \theta_2=t(1-t), \theta_3 = (1-t)^2$ for $t \in [0,1]$ with parameter $Q=0$.
 \end{lemma} 
The verification of Lemma \ref{temp3} is a computation best carried out using a symbolic linear algebra package.
%

The proof of the classification Theorem \ref{T1} is elementary linear algebra, meshed with the
stochasticity conditions, but is somewhat lengthy. The approach, motivated by  
the treatment of annihilating random walks and the asymmetric exclusion process in the introduction, 
is to look for eigenvectors of $\sigma$ that are in the factorised form 
$w \otimes w$ for some $w \in V$. The aim is to find two such independent vectors, recalling that
$v \otimes v = (1,1,1,1)^T$ is always one such vector.  
The first lemma below identifies exactly when this is possible and the second lemma records 
an exact form for the second factorised eigenvector, when it can be found, and the 
corresponding action of $\sigma$ on $v \otimes w$ and $w \otimes v$ (which will
be useful in describing duality functions).
Luckily, the search described in these two lemmas also characterises all possibilities for $\sigma$, and 
Theorem \ref{T1} follows directly. The lemmas are proved in Section \ref{s5}. 
 \begin{lemma} \label{temp1}
 Let $\sigma \in End(V\otimes V)$ be a stochastic matrix such that 
 $ \sigma^2=\sigma$  and $ \rho \sigma=\sigma \rho$.
 Then either 
 
(i) there exists $w \in V$ such that $v=(1,1)^T$ and $w$ are linearly independent
and $\sigma w\otimes w=w\otimes w$, or 

(ii) $\sigma$ is one of the following three models: 
the symmetric anti-voter  model [SAVM], 
the symmetric case $ \theta = \fr$ of the biased voter model [BVM], or the reshuffle model [RM]
described in (\ref{savm}, \ref{bvm}, \ref{rm}).
\end{lemma}

 \begin{lemma} \label{temp2}
 Let $\sigma \in End(V\otimes V)$ be a stochastic matrix such that $ \sigma^2=\sigma$ and $\rho \sigma=\sigma \rho$.
 Assume also that there is an element $w \in V$, independent of $v = (1,1)^T$, such that $ \sigma w\otimes w=w\otimes w$.
 Then, after possibly a particle-hole conjugation, there are exactly seven such non-trivial 
 matrices $\sigma \neq I$ as listed below, where we indicate one possible choice of the vector $w$, and the 
 action of $\sigma$ on $v \otimes w$ and $w \otimes v$.
 \begin{enumerate}
 \item[] [SEP]  Symmetric exclusion process:
 \beast 
 w=\bp 1\\\omega\ep \; \mbox{with} \;\; \omega \neq 1, \mbox{ or }
 w=\bp 0\\1\ep, \qquad
\sigma v \otimes w=\sigma w\otimes v= \frac{1}{2}( v\otimes w+ w\otimes v).
\eeast
 \item[]  [BVM] The non-symmetric case $\theta \neq \fr$ of the biased voter model:
\beast 
w=\bp 1-\theta \\ \theta \ep,  \qquad
\sigma v\otimes w = \sigma w\otimes v=\theta(1-\theta) v\otimes v+ w\otimes w.
\eeast
 \item[]  [ACSRW] Symmetric annihilating-coalescing random walks:
\beast \label{L2_3_ac}
w=\bp -\theta\\1\ep,  \qquad
\sigma v\otimes w = \sigma w\otimes v=\fr v\otimes v+\fr w\otimes w.
\eeast
 \item[] [CSRWB] Symmetric coalescing random walks with branching: 
\beast
w=\bp 0\\1\ep,  \qquad
\sigma v\otimes w=\sigma w\otimes v=\theta v\otimes v+(1-\theta) w\otimes w.
\eeast
 \item[]  [ASRWPI] Symmetric annihilating random walks with pairwise immigration:
\beast
w=\bp -1\\1\ep,  \qquad
\sigma v\otimes w=\sigma w\otimes v=\theta v\otimes v+\theta w\otimes w.
\eeast
\item[] [SCAM] Stationary annihilation coalescence model:
\beast
w=\bp 2\theta-1\\1\ep,  \qquad
\sigma v\otimes w=\frac{1}{2}\left(v\otimes v+ w\otimes w\right)+\frac{1}{2}( v\otimes w- w\otimes v),\nonumber\\
\sigma w\otimes v=\frac{1}{2}\left(v\otimes v+ w\otimes w\right)+\frac{1}{2}( w\otimes v- v\otimes w).
\eeast
 \item[] [DM] Dimer model::
\beast
w=\bp 1\\-1\ep,  \qquad
\sigma v\otimes w=\frac{1-2\theta}{2}\left(v\otimes v+ w\otimes w\right)+\frac{1}{2}( v\otimes w- w\otimes v),\nonumber\\
\sigma w\otimes v=\frac{1-2\theta}{2}\left(v\otimes v+ w\otimes w\right)+\frac{1}{2}( w\otimes v- v\otimes w). 
\eeast
 \end{enumerate} 
 \end{lemma}
\subsection{Generators built using the braid relation}    \label{s2.2}
One drawback of Theorem \ref{T1} is that we imposed the symmetry condition
$ \rho \sigma =\sigma \rho$. More work is needed to complete a classification without
this assumption, and we know we will find non-symmetric versions of the models in Theorem
\ref{T1} but do not know if other models will emerge. In this Section we give a partial answer 
by establishing a second classification without the symmetry assumption. instead we ask that 
the deformed braid relation (\ref{braidreln}) holds (and also a further constraint detailed below
that makes the classification tractable). 
The outcome of the search is that we do indeed get non-symmetric versions of all the models, and 
two further exceptional models, which we call Erosion Models and are easy to analyse (see \ref{s3.7}).
%

We were not able to classify all the examples of $\sigma$ satisfying the 
deformed braid relation (\ref{braidreln}) and the quadratic relation
(\ref{quadraticreln}). Instead we have analyzed two Ansatzes for $\sigma$ which do give rise to physically 
interesting solutions. We choose $\sigma$ in one of the following forms:
\begin{equation}
\label{ansatze}
\mbox{(a)}\quad \sigma = 
\bp
* & * & * & *\\
* &* &* &* \\
* &* &* &* \\
* &0 &0 &* 
\ep
\qquad \mbox{ (b)}\quad \sigma =
\bp
* & * & * & *\\
* &* &* &0\\
* &* &* &0 \\
* & * & * & *
\ep
\end{equation}
were symbol "$*$" denotes components of $\sigma$ which are not restricted by the Ansatz. 
Substituting the Ansatzes (\ref{ansatze}) into the Hecke relations we notice that certain components of these 
matrix polynomial equations assume a factorised form. The factorisation allows us to find more and more
constraints, and to eventually classify all models fitting the Ansatzes.
\begin{theorem}
\label{T2}	
Suppose $\sigma \in \text{End}(V \otimes V)$ satisfies the stochastic constraints (\ref{stochasticity})
and also one of the two Anzatses (\ref{ansatze}).
Suppose that the family $\{\sigma_n\}$ satisfy the deformed braid relation (\ref{braidreln}) and the quadratic relation
(\ref{quadraticreln}).
Then, after possibly a conjugation to $\rho \sigma \rho$ or $\tau \sigma \tau$ or $\rho \tau \sigma \tau \rho$, 
any non-trivial $\sigma \neq I$  must lie in one of the 
following eleven families (written in the standard basis):
 \bea
 \begin{array}{c}
 \mbox{Asymmetric exclusion process}\\ \mbox{[ASEP]} \end{array} \quad & 
 \sigma = \bp 1 &0&0& 0 \\0&l&r&0\\
 0&l&r&0\\ 0 &0&0&1 \ep &
 \quad     \begin{array}{c}  \mbox{for $r,l \in [0,1]$,} \\
  \mbox{$r+l=1$,} \end{array}   \label{asep} \\ 
   \begin{array}{c}
   \mbox{Biased voter model}\\ \mbox{[BVM]} \end{array}
  \quad & 
\sigma = \bp 1&0&0&0\\ \theta&0&0&1-\theta\\  \theta&0&0&1-\theta\\0&0&0&1 \ep &
 \quad   \mbox{for $\theta \in [0,1]$,}\label{bvm2} \\
 \begin{array}{c}
 \mbox{Asymmetric voter model}\\ \mbox{[AVM]} \end{array}
  \quad & 
\sigma = \bp 1&0&0&0\\ r&0&0&l\\ l&0&0& r\\0&0&0&1 \ep &
 \quad     \begin{array}{c}  \mbox{for $r,l \in [0,1]$,} \\
  \mbox{$r+l=1$,} \end{array} \label{avm} \\
    \begin{array}{c}
     \mbox{Asymmetric anti-voter model}\\\mbox{[AAVM]} \end{array}  \quad & 
 \sigma = \bp  0&l &r &0\\0&1&0&0\\0&0&1&0\\0&r&l&0 \ep &
 \quad    \begin{array}{c}  \mbox{for $r,l \in [0,1]$,} \\
  \mbox{$r+l=1$,} \end{array} \label{aavm}  \\
  \begin{array}{c} 
\mbox{Annihilating-coalesing} \\ \mbox{random walks}\\ \mbox{[ACRW]} \end{array} \quad & 
 \sigma = \bp 0&l(1-\theta)&r(1-\theta)&\theta \\ 0&l&r&0\\ 
 0&l&r&0\\0&0&0&1 \ep &
 \quad \begin{array}{c}  \mbox{for $r,l,\theta \in [0,1]$,} \\
  \mbox{$r+l=1$,} \end{array} \label{acrw} \\
 \begin{array}{c}
\mbox{Coalesing symmetric random} \\ \mbox{walks with branching}\\ \mbox{[CSRWB]} \end{array} \quad & 
\sigma =  \bp 1-2 \theta&\theta&\theta&0\\1-2 \theta&\theta&\theta&0\\
1-2 \theta&\theta&\theta&0\\0&0&0&1 \ep &
 \quad \mbox{for $\theta \in [0,\frac12]$,} \label{csrwb2} \\
     \begin{array}{c}
\mbox{Annihilating symmetric random} \\ \mbox{walks with pair immigration}\\ \mbox{[ASRWPI]} \end{array} \quad & 
 \sigma = \bp \frac12 - \theta &0&0& \frac12 + \theta \\0&\frac12&\frac12&0\\
 0&\frac12&\frac12&0\\ \frac12 - \theta &0&0& \frac12 + \theta  \ep &
 \quad \mbox{for $\theta \in [0,\frac12]$,} \label{asrwpi2} \\
 \begin{array}{c}
 \mbox{Dimer model}\\ \mbox{[DM]} \end{array} \quad & 
 \sigma = \bp \frac12 &0&0&\frac12 \\0&1&0&0\\0&0&1&0\\ \frac12&0&0&\frac12 \ep &
 \quad 
  \label{dm2} \\
  \begin{array}{c}
 \mbox{Reshuffle model}\\ \mbox{[RM]} \end{array} \quad & 
 \sigma = \bp 1&0&0&0 \\  1&0&0&0 \\
1&0&0&0 \\  1&0&0&0 \ep &
  \label{rm2} \\
   \begin{array}{c}
 \mbox{Erosion model 1}\\ \mbox{[EM1]} \end{array}\quad & 
 \sigma = \bp 1&0&0&0\\0&1&0&0\\0&0&0&1\\0&0&0&1 \ep &
\label{em1} \\
 \begin{array}{c}
 \mbox{Erosion model 2}\\ \mbox{[EM2]} \end{array}\quad & 
 \sigma = \bp 0&1&0&0\\0&1&0&0\\0&0&0&1\\0&0&0&1 \ep &
\label{em2} 
 \eea
\end{theorem}
\section{The eleven models} \label{s3}
We examine the models that arise in the classification theorems and, where possible, link
the algebraic information gained in the classification to duality functions. 
\subsection{Exclusion processes [SEP] and [ASEP]}  \label{s3.1}
The asymmetric exclusion process [ASEP] is discussed in 
Section \ref{s1.1.3}, where for initial conditions with a rightmost particle the staircase
duality functions form a complete set. 

For the symmetric exclusion process [SEP] 
%
%
the action of $\sigma$ on the basis $(v \otimes v, w \otimes w, v \otimes w, w \otimes v)$ recorded
in Lemma \ref{temp2} shows that duality functions of product moment type (\ref{pmdf})
exist. Indeed
the choice $w=(1,0)^T$ gives exactly the product moment $g^{(n)}_{x_1,\ldots,x_n}$ 
which corresponds to the test function $\prod_{k=1}^n \eta(x_k)$. These are the 
standard dualities functions (see Liggett \cite{liggett1985interacting} Chapter VIII) 
where the dual process can be taken to be a symmetric exclusion process with 
finitely many particles.

Using the other vectors listed in Lemma \ref{temp2}, namely $w = (1,\omega)^T$ or
$w = (0,1)^T$, in the product moment duality function (\ref{pmdf})
corresponds to certain special
linear combinations of product moments, and the algebra just shows that 
the generator leads to expressions that can be re-expressed
in terms of the same special linear combinations.  For example when
$w = (0,1)^T$ this shows that the probability of empty regions solves a closed set
of equations (for any fixed cardinality of region). 
%
\subsection{Voter models [BVM], [AVM] and [AAVM]} \label{s3.2}
In the asymmetric voter model model [AVM], each site independently triggers a disagreeing neighbouring site to the right 
to take its value at rate $r$, and a disagreeing neighbouring site to the left to take its value at rate $l$. 
This model has a generator of the form (\ref{generator}) with $\sigma$ given by (\ref{avm}).
This model is usually studied using the mapping to domain walls: the bijective mapping that 
sends $\eta$ to the set markers in the dual lattice $\frac12 + \Z$ that mark the boundaries 
between runs of $0$'s and runs of $1$.
The positions of these markers perform instantaneously annihilating random walks on $\frac12 + \Z$, where
the domain walls jump left with rate $l$ and right with rate $r$, a model that is known to lead to a Pfaffian point process at fixed times (see Section \ref{s4.3} below). 

Analysing the generator of  [AVM]  in the framework we are using, we find there are no second tensor form eigenvector
$w \otimes w$ of $\sigma$. However there are product moment type dualities functions. Indeed 
the product moments $g^{(n)}_{x_1,\ldots,x_n}$ as in (\ref{dualitytype3}) are duality functions, and the dual action
is that of instantaneously coalescing asymmetric 
random walks (left jumps at rate $r$ and right jumps at rate $l$).  Product
exponential moments $g^{(n)}_{\beta;x_1,\ldots,x_n}$ as in (\ref{dualitytype4}) are also duality functions for any $\beta$,
in that $\sigma$  has a good action on $w \otimes w$, but the dual action is more complicated. 

In the biased voter model [BVM], 
each neighbouring pair of disagreeing sites, that is with values $10$ or $01$, jump (independently) to $11$ at rate 
$\theta$ or $00$ at rate $1-\theta$.  This model has a generator of the form (\ref{generator}) with $\sigma$ given by (\ref{bvm}). The mapping to domain walls still yields a Markov model,  
with a domain wall with $1$s to the left and $0$s to the right jumps to the right at rate
$\theta$ and to the left at rate $1-\theta$, while a domain wall with $0$s to the left and $1$s to the right jumps to the right at rate $1-\theta$ and to the left at rate $\theta$. In the non-symmetric case $\theta \neq \frac12$, there is a 
second eigenvector $w \otimes w$ as found  in Lemma \ref{temp2}. To simplify the dual action we scale this vector and take, 
when $\theta \neq 1$, the vector 
$w = (1, \theta/(1-\theta))^T$, which leads to an alternating interval duality function
$ f^{(2n)}_{x_1,x_2,\ldots,x_{2n}}$ defined in (\ref{fdualityfn}). The dual action has
annihilating particles with sites $x_{2n}$ moving right at rate $\theta$ and left at rate $1-\theta$, while
sites $x_{2n+1}$ moving right at rate $1-\theta$ and left at rate $\theta$. 
The case $\theta = 1$ can be treated by particle-hole conjugation, which interchanges $\theta$ and $1-\theta$. 

The final voter model we met, the asymmetric anti-voter model [AAVM], is mapped to the asymmetric model  [AVM] by the bijection $\eta \to \tilde{\eta}$ where $\tilde{\eta}(x) = \eta(x)$ for $x \in 2\Z$ and  
$\tilde{\eta}(x) = 1-\eta(x)$ for $x \in 2\Z+1$, and the domain wall map and the dualities functions 
still operate after this bijection.
\subsection{Annihilating and coalescing systems [ACRW], [CSRWB], [ASRWPI]}  \label{s3.3}
Particles perform independent asymmetric simple random walks on $\Z$, jumping right with rate $r$ and left with rate $l$ (and 
$r+l=1$), and upon any collision they instantaneously annihilate with probability $\theta$ or 
 instantaneously coalesce with probability $1-\theta$. This produces the process [ACRW] 
with generator of the form (\ref{generator}) with $\sigma$ given by (\ref{acrw}).
Using the eigenvector $w = (-\theta,1)^T$, the alternating interval
vector 
$ f^{(2n)}_{x_1,x_2,\ldots,x_{2n}} $ defined in (\ref{fdualityfn}) corresponds to 
\[
f^{(2n)}_{x_1,x_2,\ldots,x_{2n}}
\quad \leftrightarrow \quad \prod_{1 \leq i \leq n} (-\theta)^{\eta(x_{2i-1},x_{2i}]}
\quad \mbox{where} \quad \eta(a,b] = \sum_{x=a+1}^b \eta(x).
\]
The action of $\sigma$ on the basis $(v \otimes v, w \otimes w, v \otimes w, w \otimes v)$ recorded
in Lemma \ref{temp2} extends to the asymmetric case, with 
$\sigma(w \otimes v) = r v \otimes v + l w \otimes w$ and 
$\sigma(v \otimes w) = l v \otimes v + r w \otimes w$.
This implies that
\[
Lf^{(2n)}_{x_1, x_2, \ldots, x_{2n}} = \sum_{k=1}^{2n} \Delta^{r,l}_{k} f^{(2n)}_{x_1, x_2, \ldots, x_{2n}},
\quad \mbox{for $x_1<x_2<\ldots <x_{2n}$,}
\]
where $\Delta^{r,l}_{k}$ is the non-symmetric discrete asymmetric Laplacian 
acting on the variable $x_k$. 
Thus 
$f^{(2n)}$ is a duality function where the dual is a system of $2n$ annihilating simple 
random walks. This duality function was used in \cite{tz1} to show the one-dimensional marginal $\eta_t$, for any $t \geq 0$,  
is Pfaffian point processes on $\Z$ under any deterministic initial condition. There are no
product moment or exponential product moment duality functions. 

Next we examine coalescing symmetric random walks with branching [CSRWB]. 
Particles perform independent symmetric simple random walks on $\Z$, with total jump rate
 $2 \alpha$, they instantaneously coalesce upon collision, and in addition each particle independently
 branches, producing a particle onto each neighbouring empty site at rate $\beta$. Scaling time 
 so that $2 \alpha+\beta =1$, then taking $\theta = \alpha = (1 - \beta)/2$, 
 we arrive at the generator of the form (\ref{generator}) with $\sigma$ given by (\ref{csrwb}).
The only other tensor form eigenvector $w \otimes w$ with a good action is 
$w = (0,1)^T$, for which the alternating interval
vector 
$ f^{(2n)}_{x_1,x_2,\ldots,x_{2n}} $ defined in (\ref{fdualityfn}) corresponds to 
\[
f^{(2n)}_{x_1,x_2,\ldots,x_{2n}}
\quad \leftrightarrow \quad \prod_{1 \leq i \leq n} \ind (\eta(x_{2i-1},x_{2i}]=0)
\quad \mbox{where} \quad \eta(a,b] = \sum_{x=a+1}^b \eta(x),
\]
the well known empty interval duality function.
The action of $\sigma$ on the basis $(v \otimes v, w \otimes w, v \otimes w, w \otimes v)$ recorded
in Lemma \ref{temp2} shows that
\[
Lf^{(2n)}_{x_1, x_2, \ldots, x_{2n}} = \sum_{k=1}^{2n} \Delta^{1-\theta,\theta}_{2k-1} 
f^{(2n)}_{x_1, x_2, \ldots, x_{2n}} +  \sum_{k=1}^{2n}  \Delta^{\theta,1-\theta}_{2k}  f^{(2n)}_{x_1, x_2, \ldots, x_{2n}}, \quad
\quad x_1<x_2<\ldots <x_{2n}
\]
The vector $f^{(2n)}$ is a duality function where the dual is a system of $2n$ annihilating random walks,
where even and odd positioned particles have different asymmetric rates. 
This duality function was used in \cite{tz2} to show the one-dimensional marginal $\eta_t$, for any $t \geq 0$,  
is Pfaffian point processes on $\Z$ under any deterministic initial condition. Again there are no 
product moment or exponential product moment duality functions.

Finally we examine annihilating symmetric random walks with pair immigration [ASRWPI].
Particles perform independent symmetric simple random walks on $\Z$, with total jump rate
 $2 \alpha$, they instantaneously annihilate upon collision, and in addition there is immigration of pairs 
 of particles at each pair of sites $\{x, x+1\}$ at rate $\beta$ independently for all $x \in \Z$. Scaling time 
 so that $\alpha+\beta =\frac12$, then taking $\theta = \alpha = \frac12 - \beta$, 
 we arrive at the generator of the form (\ref{generator}) with $\sigma$ given by (\ref{asrwpi}).
 
 The only other tensor form eigenvector $w \otimes w$ with a good action 
using $w = (-1,1)^T$, for which the alternating interval
vector 
$ f^{(2n)}_{x_1,x_2,\ldots,x_{2n}} $ defined in (\ref{fdualityfn}) corresponds to the duality function
\[
f^{(2n)}_{x_1,x_2,\ldots,x_{2n}}
\quad \leftrightarrow \quad \prod_{1 \leq i \leq n} (-1)^{\eta(x_{2i-1},x_{2i}]}
\quad \mbox{where} \quad \eta(a,b] = \sum_{x=a+1}^b \eta(x).
\]
The action of $\sigma$ on the basis $(v \otimes v, w \otimes w, v \otimes w, w \otimes v)$ recorded
in Lemma \ref{temp2} shows that
\[
Lf^{(2n)}_{x_1, x_2, \ldots, x_{2n}} = \theta \sum_{k=1}^{2n}  \Delta_{k} f^{(2n)}_{x_1, x_2, \ldots, x_{2n}}
- (1-2 \theta) \sum_{k=1}^{2n}  f^{(2n)}_{x_1, x_2, \ldots, x_{2n}},
\quad x_1<x_2<\ldots <x_{2n}.
\]
This duality function was used in \cite{tz2} to show the one-dimensional marginal $\eta_t$, for any $t \geq 0$,  
is Pfaffian point processes on $\Z$ under any deterministic initial condition. 
Only the case $\theta =0$, corresponding to immigration of pairs with no underlying motion, has 
product moment dualities, where the product exponential moment function $g^{(n)}_{\beta;x_1,\ldots,x_n}$ as in (\ref{dualitytype4}) is a duality function for any $\beta$. 
\subsection{Stationary coalescence-annihilation model [SCAM]} \label{s3.4}
Every neighbouring pair of particles, at sites $\{x,x+1\}$ and independently for all $x$, reacts at rate $1$;
the reaction is either a coalescence, producing $10$ or $01$ each with probability $\theta$, or
an annihilation, producing $00$ with probability $1-2 \theta$. This process has
generator of the form (\ref{generator}) with $\sigma$ given by (\ref{scam}).

The rather simple dynamics produce decreasing sample paths. Configurations
without a pair of neighbouring occupied sites are the fixed point, which form
the extreme points for the set of invariant measures.  No pair of
distinct initial conditions can be successfully coupled, and each
solution converges to a distinct (random) mixture of fixed points. 
Nevertheless, without a duality function, exact calculations for
the limiting distribution would be difficult. We use the duality to analyse the limiting
state for the process started from its maximal initial condition where all sites are occupied,
which turns out to be a renewal process on $\Z$.

The action of $\sigma$ on the basis $(v \otimes v, w \otimes w, v \otimes w, w \otimes v)$ 
recorded in Lemma \ref{temp2} does not mesh with either an alternating interval duality function, or a 
staircase duality function.  
Moreover there are no product moment, or exponential product moment, dualities. 
However there are other choices of $w$ that lead to a useful action, and we choose
the eigenvector $e^{(1)} \otimes e^{(1)}$, with eigenvalue $0$, on 
which the generator acts nicely:
%
\bea 
&& \sigma v\otimes e^{(1)}= v\otimes e^{(1)} - (1-\theta)e^{(1)}\otimes e^{(1)}, \qquad
\sigma e^{(1)}\otimes v= e^{(1)}\otimes v - (1-\theta)e^{(1)}\otimes e^{(1)}. \label{scamga}
\eea
Therefore the action of the generator preserves the order of vectors $e^{(1)}$ and $v$
in the tensor product. We may take the alternating 
interval vector defined in (\ref{fdualityfn}) with $w = e^{(1)} = (1,0)^T$, corresponding to 
\[
f^{(2n)}_{x_1,x_2,\ldots,x_{2n}}
\quad \leftrightarrow \quad \prod_{1 \leq i \leq n} 
\ind(\eta(y)=1 \; \mbox{for $x_{2i-1} <y \leq x_{2i}$}),
\]
that is the indicator that the intervals $(x_{2i-1}, x_{2i}]$ are fully occupied. We take parameters
$x_1 < x_2 < \ldots < x_{2n}$ and $n \geq 1$.

Using (\ref{scamga}) it is straightforward to derive the following set of equations for the 
expectations of the duality functions $\Phi^{(2n)}_t(x_1,\ldots, x_{2n}) 
:= \E [  \prod_{1 \leq i \leq n} 
\ind(\eta(y)=1 \; \mbox{for $x_{2i-1} <y \leq x_{2i}$})] $
under the initial condition $\eta_0 \equiv 1$: for every $n \in \N$,
\begin{equation}
\label{scameq}
\partial_t \Phi^{(2n)}_t  = - \sum_{k=1}^n D_k \Phi^{(2n)}_t, \quad \mbox{for $t \geq 0$, 
$x_1< \ldots < x_{2n}$},
\end{equation}
where 
\begin{eqnarray*}
D_k \Phi (\ldots,x_{2k-1},x_{2k},\ldots) & = & (1-\theta) \Phi(\ldots,x_{2k-1}-1,x_{2k},\ldots) 
+ (1-\theta)\Phi(\ldots,x_{2k-1},x_{2k}+1,\ldots) \\
&& + (x_{2k}-x_{2k-1}-1) 
\Phi(\ldots,x_{2k-1},x_{2k},\ldots),
\end{eqnarray*}
and where we use the boundary conditions
\begin{equation} \label{BC}
\Phi^{(2n)}_t (\ldots,x_{i-1}, x_{i},x_{i+1},x_{i+2} \ldots) = \Phi^{(2n-2)}_t (\ldots,x_{i-1},x_{i+2} \ldots)
\quad \mbox{when $x_{i}=x_{i+1}$,}
\end{equation}
and the convention $\Phi^{(0)}_t \equiv 1$ to define the right hand side of (\ref{scameq}).

The system of equations (\ref{scameq}), with the boundary conditions (\ref{BC}) and 
initial condition $\Phi^{(2n)}_0 \equiv 1$ corresponding
to all sites occupied, can be solved exactly. Uniqueness holds for bounded solutions, as can be 
shown  inductively in $n \in 2 \N$, by an energy calculation: if 
$\mathcal{E}_t := \sum_{x \in V_{2n}} (\Phi^{(2n)}_t)^2(x) \exp(-\sum_i |x_i|)$, 
summing over the open cell $V_{2n} = \{x: x_1 < \ldots < x_{2n}\}$, then $(d/dt)  \mathcal{E}_t \leq C \mathcal{E}_t$.
The main building block for the solution is $\Phi^{(2)}_t(x,y)$, the probability that the interval
$(x,y]$ is fully occupied at time $t$. Exploiting the translation invariance we find that
the (unique bounded) solution is, setting $\gamma=2(1-\theta)$, 
\begin{equation} \label{2ptrho}
\Phi^{(2)}_t(x,y)=e^{-(y-x-1)t}e^{-\gamma (1-e^{-t})}, \quad x \leq y, \; t \geq 0.
\end{equation}
In particular, the one-particle density is
\bea\label{scamdens}
\rho_t :=
\Phi_t^{(2)}(x-1,x)=e^{-\gamma (1-e^{-t})}.
\eea 
We look for the solution to (\ref{scameq}) in the form 
\bea\label{scamnrun}
\Phi_t^{(2n)}(x_1,y_1,\ldots, x_n,y_n)=\prod_{k=1}^n \Phi_t^{(2)}(x_k,y_k)\prod_{k=1}^{n-1} F_t(x_{k+1}-y_k).
\eea
The insertions of the function  $F_t$ in (\ref{scamnrun}) describe
spatial correlations between the runs of occupied sites. Substituting 
the Ansatz (\ref{scamnrun}) into (\ref{scameq}), one finds that it can be
satisfied provided the function $(F_t(x): t \geq 0, x \in \N)$ solves the following problem
\bea
\left\{
\begin{array}{ll}
\partial_t F_t(x)=\gamma e^{-t} \left(F_t(x)-F_t(x-1) \right),& t>0, \, x>0,\\
F_t(0)=e^{-t}e^{\gamma \left(1-e^{-t}\right)},& t>0,\\
F_0(x)=1,& x \geq 0.
\end{array}
\right.
\eea
The solution is 
\bea\label{scamspcor}
F_t(x)=e^{\tau(t)}\left(E_x(-\tau(t))+\frac{1}{\gamma}
\frac{(-\tau(t))^{x+1}}{(x+1)!} \right) \quad \mbox{for $t \geq 0$ and $x\geq 0$,} 
\eea
where $\tau(t)=\gamma (1-e^{-t})$ for $t \geq 0$, and $E_n(z) = \sum_{k=0}^n z^k/k!$ is the 
exponential polynomial of degree $n$. Notice that, as $x\rightarrow \infty$,
$F_t(x)$ approaches $1$ exponentially fast, which describes the de-correlation
of runs in the limit of large separations. 

By setting $x_k=y_k-1$ in (\ref{scamnrun}) we get a complete description of
the law of SCAM at a fixed time $t>0$ in terms of the correlation functions:
for $t\geq 0, x_1<x_2<\ldots <x_n$ and  $n\geq 1$
\bea 
\rho^{(n)}_t (x_1, x_2,\ldots, x_n) & := & \pr\left[\eta_t(x_k)=1, k=1,2,\ldots, n\right]
\nonumber\\
& = &
(\rho_t)^n \prod_{k=2}^n F_t(x_k-x_{k-1}-1). \label{scamcfs}
\eea
Examining (\ref{scamcfs}), we conclude that SCAM started from the fully
occupied lattice at $t=0$ is a {\em stationary renewal process} at any fixed $t>0$.
One immediate consequence of (\ref{scamcfs}) and Bayes' formula is that
$\{\eta_t(k)\}_{k<x}$ and $\{\eta_t(k)\}_{k>x}$ are independent conditionally
on the event $\{\eta_t(x)=1\}$. 
A stationary renewal processes are characterised by its renewal measure ({\em Janossi
density} in the terminology of point processes),
\begin{equation} \label{janossi}
J_t (n)= \pr [\eta_t(k)=0, 1\leq k \leq n-1,
\eta_t(n)=1| \eta_t(0)=1], \quad \mbox{for $n \geq 1$, $t>0$.} 
\end{equation}
The above conditional independence
property of a renewal processes allows one to derive a linear (renewal) equation for the 
renewal measure
\bea
\rho_t^{(2)}(0,n) & = & \pr [\eta_t(1) = 1, \eta_t(k)=0, 1\leq k \leq n-1,
\eta_t(n)=1] \nonumber \\
&& \hspace{.2in} + 
\sum_{k=1}^{n-1} 
\pr[ \eta_t(0)=1, \eta_t(j)=0, 1\leq j \leq k, \eta_t(k+1) =1, \eta_t(n)=1] \nonumber \\
& = & \rho_t J_t (n)+ \sum_{k=1}^{n-1}
J_t(k) \; \rho_t^{(2)}(k+1,n) \quad \mbox{for $n \geq 1$}.
\eea
Translation invariance implies $\rho_t^{(2)}(k+1,n) = \rho_t^{(2)}(0,n-k)$. 
Introducing the generating functions
$ \hat{J}_t(z) = \sum_{n=1}^\infty J_t(n) z^n$ and $\hat{\rho}^{(2)}_t (z) = \sum_{n=1}^\infty \rho^{(2)}_t(0,n) z^n$
it is easy to solve the equation, namely
\bea\label{scamtransj}
\hat{J}_t(z)=\frac{\hat{\rho}^{(2)}_t(z)}{\rho_t + \hat{\rho}^{(2)}_t(z)}.
\eea
Moreover, an explicit expression for the transform of the two-point density
can be found using (\ref{scamcfs})
\bea\label{scamtransrho}
\hat{\rho}^{(2)}_t (z)=\frac{z\rho_t}{1-z}e^{-\tau(t)z}
+ \frac{\rho_t}{\gamma}(e^{-\tau(t)z}-1).
\eea
The relations (\ref{scamtransj}, \ref{scamtransrho}) fully characterise
the renewal measure. 

As an application of the above exact formulae, we examine the large $n$
asymptotic of the renewal measure $J_\infty(n)$ in the stationary state, that is
sending $t \to \infty$.  In this limit, $\tau(t) \rightarrow \gamma$ and
$ \rho_t \to e^{-\gamma}$. We contrast the behaviour under pure coalescence and 
under pure annihilation. 
%
When $\theta = \frac12$ all reactions are coalesences.
In this case $\gamma=1$. From (\ref{scamtransj}) 
one finds $\hat{J}_\infty(z)=\left(1- (1-z)e^z\right)$ 
and hence
\bea\label{scamfd}
J_\infty(n) = \frac{(n-1)}{n!}, \quad \mbox{for $n\geq 1$.}
\eea
Notice that $J_\infty(1)=0$, reflecting the fact that in the stationary
state the probability of a configuration with two ones in a row is zero. 

When $\theta = 0$ all reactions are annihilations.
In this case $\gamma=2$ and the formula for the generating function
of Janossi densities takes the form
\bea
\hat{J}_\infty(z)=\frac{z-\tanh(z)}{1-z \tanh(z)},~ z \in \C.
\eea
Notice that the function $\hat{J}_\infty$ is odd and so has zero
Taylor coefficients of even degree, reflecting the fact that
the probability of runs of zeros of odd length is zero
as particles can only vanish in pairs. The large $n$ asymptotics of $J_\infty(n)$ is determined 
by the singularities of $\hat{J}_\infty$ nearest to zero. These are the smallest roots
of the equation $z\tanh(z)=1$ which are equal to $\pm r$, where $ r \approx 1.19968$. 
Applying Cauchy's theorem one finds that
\bea
J_\infty(n) = (1-(-1)^n) \, \frac{r-r^{-1}}{r^2}r^{-n}(1+O(\Delta^{-n})),
\eea
where $\Delta>1$ is determined by the gap between $r$ and the set
of absolute values of complex roots.  Notice also that the renewal measure
decays exponentially with $n$, whereas for pure coalescence the decay
(\ref{scamfd}) is super-exponential. The reason is that 
for pure coalesence particles disappear only in singletons while reactions occur only when there 
is a neighbouring pair of occupied sites; so to produce a long row of empty sites, the coalescences 
must occur successively from left to right or successively from right to left, both
unlikely scenarios. 

\subsection{The dimer model [DM]} \label{s3.5}
At neighbouring sites $\{x,x+1\}$, and independently for all $x$, an occupied pair $11$
annihilates at rate $\theta$, and an unoccupied pair $00$ is replaced by an occupied pair 
at rate $1-\theta$. This process has
generator of the form (\ref{generator}) with $\sigma$ given by (\ref{dm}).

The action of $\sigma$ on the basis $(v \otimes v, w \otimes w, v \otimes w, w \otimes v)$ 
recorded in Lemma \ref{temp2} does not mesh with either an alternating interval duality function, or a 
staircase duality function. When $\theta \neq \frac12$ there are no product moment or exponential product 
moment dualities (we will see that the case $\theta = \frac12$ is special below). 

This model, however, is equivalent to an inhomogeneous exclusion model, as follows.
The dimer model has the following two obvious trapped states, where no jumps are possible,
\[
\eta_a(k) : = \frac{\left(1-(-1)^k\right)}{2},~k\in \Z, \quad \mbox{and} \quad
 \eta_b(k) : =\frac{\left(1+(-1)^k\right)}{2},~k\in \Z.
\]
Let us define quasi-particles as deviations from one of the above states. 
We fix $\eta_a$ as the reference state and then the configuration $\xi$ of quasi-particles
for a given configuration $\eta$ of particles is defined by 
\[
\xi(k)=\left\{ \begin{array}{ll}
\eta(k) & \text{when $k$ is even,}\\
1- \eta(k) & \text{when $k$ is odd,}
\end{array}
\right.  ~~k \in \Z.
\]
The map $\eta \to \xi$ is a bijection on $ \{0,1\}^\Z$, so that the process of quasi-particles
$(\xi_t)$ remains Markov. Its generator is 
\bea
L = \sum_{k\in \Z} \left(q^{(0)}_{2k}+q^{(1)}_{2k+1} \right),
\eea
where
\bea\label{tmgens}
q^{(0)}=\bp
0&0&0&0\\
0&-\theta&\theta&0\\
0&1-\theta&-(1-\theta)&0\\
0&0&0&0
\ep, \qquad
q^{(1)}=\bp
0&0&0&0\\
0&-(1-\theta)&1-\theta&0\\
0&\theta&-\theta&0\\
0&0&0&0
\ep.
\eea
We conclude that the image of the dimer model is equivalent to an exclusion model. It
coincides with standard SEP for $\theta = \frac12$, but is a spatially inhomogeneous, and non-symmetric in the 
notation from Liggett \cite{liggett1985interacting}, exclusion model otherwise. 
Therefore, when $\theta=\fr$ the model has a spanning set of duality functions while
for general asymmetric exclusion processes dualities are not known. 

We can, however, use the map to an exclusion model to find invariant measures.
Firstly, let us consider the degenerate case $\theta=0$. In this case the only allowed transitions
are $00 \rightarrow 11$. The process has monotone increasing paths and converges to a trap,
that is any state without any pair of successive zeroes.
The particle-hole conjugation maps the dimer model parameter $\theta$ to $1-\theta$, so
we now consider the range of the parameter $\theta$ to $(0,\fr]$. In \cite{liggett2012interacting},
Chapter V111.2, Theorem 2.1, Liggett  gives
a condition ensuring a family of product Bernouilli invariant measures, which is satisfied
for our quasi-particle process $(\xi_t)$. In our algebraic notation, such product Bernouilli measures
can be found by looking for a vector $\mu$ satisfying  
\[
\left(\sum_{k \in \Z} q^*_{k}\right) \mu =0, \qquad \mbox{where} 
\quad
q^*=\bp
-\theta&0&0&1-\theta\\
0&0&0&0\\
0&0&0&0\\
\theta&0&0&-(1-\theta)
\ep
\]
and which is stochastic in each co-ordinate.  Let $ \mu_0=(p,1-p)$ for $p \in [0,1]$. 
Solving for a stochastic vector $\mu_1 = (\Phi(p), 1-\Phi(p))$ so that
$q^* \mu_0 \otimes \mu_1=0$, one finds 
\[
\Phi(p)=\frac{(1-\theta)(1-p)}{\theta p + (1-\theta) (1-p)}.
\]
Notice that $\Phi \circ \Phi=I$, and this ensures that 
$q^* \mu_1 \otimes \mu_0=0$. Hence we may take 
$\mu = \cdots \otimes \mu_0 \otimes \mu_1 \otimes \mu_0 \otimes \cdots$
for any $p$ as an invariant measure, corresponding to product Bernouilli 
with parameters alternating between $p$ and $\Phi(p)$. The fixed point
where $\Phi(p) = p$ is solved by $p = \sqrt{1-\theta}/(\sqrt{\theta} + \sqrt{1-\theta})$ giving
a translation invariant  measure. 
%
\subsection{The reshuffle model [RM]} \label{3.6}
Each pair of neighbouring sites $\{x,x+1\}$, independently for all $x$, reacts 
at rate $1$ to produce $00$ with probability $\alpha_1$, $10$ with probability $\alpha_2$, $01$ 
with probability $\alpha_3$, and $11$ with probability $\alpha_4$, leading to a process with
generator of the form (\ref{generator}) with the rank one matrix
\[
\sigma = \bp
\alpha_1&\alpha_2&\alpha_3&\alpha_4\\
\alpha_1&\alpha_2&\alpha_3&\alpha_4\\
\alpha_1&\alpha_2&\alpha_3&\alpha_4\\
\alpha_1&\alpha_2&\alpha_3&\alpha_4
\ep,
\quad 
\sum_{k=1}^4\alpha_k=1, \quad \alpha_i\geq 0, \; 1\leq i \leq 4.
\]
The symmetric case (\ref{rm}) that emerged in the first classification theorem is a special case.

The classification Theorem \ref{T1} did not reveal a second tensor form eigenfunction.  
However the range of $\sigma$ is spanned by $v \otimes v$ and this implies that product moments, or 
exponential product moments for any $\beta$, are duality functions. We now make use of these to 
identify the unique invariant measure, which can be described rather completely. 
\begin{figure}[h] 
\includegraphics[width=12cm]{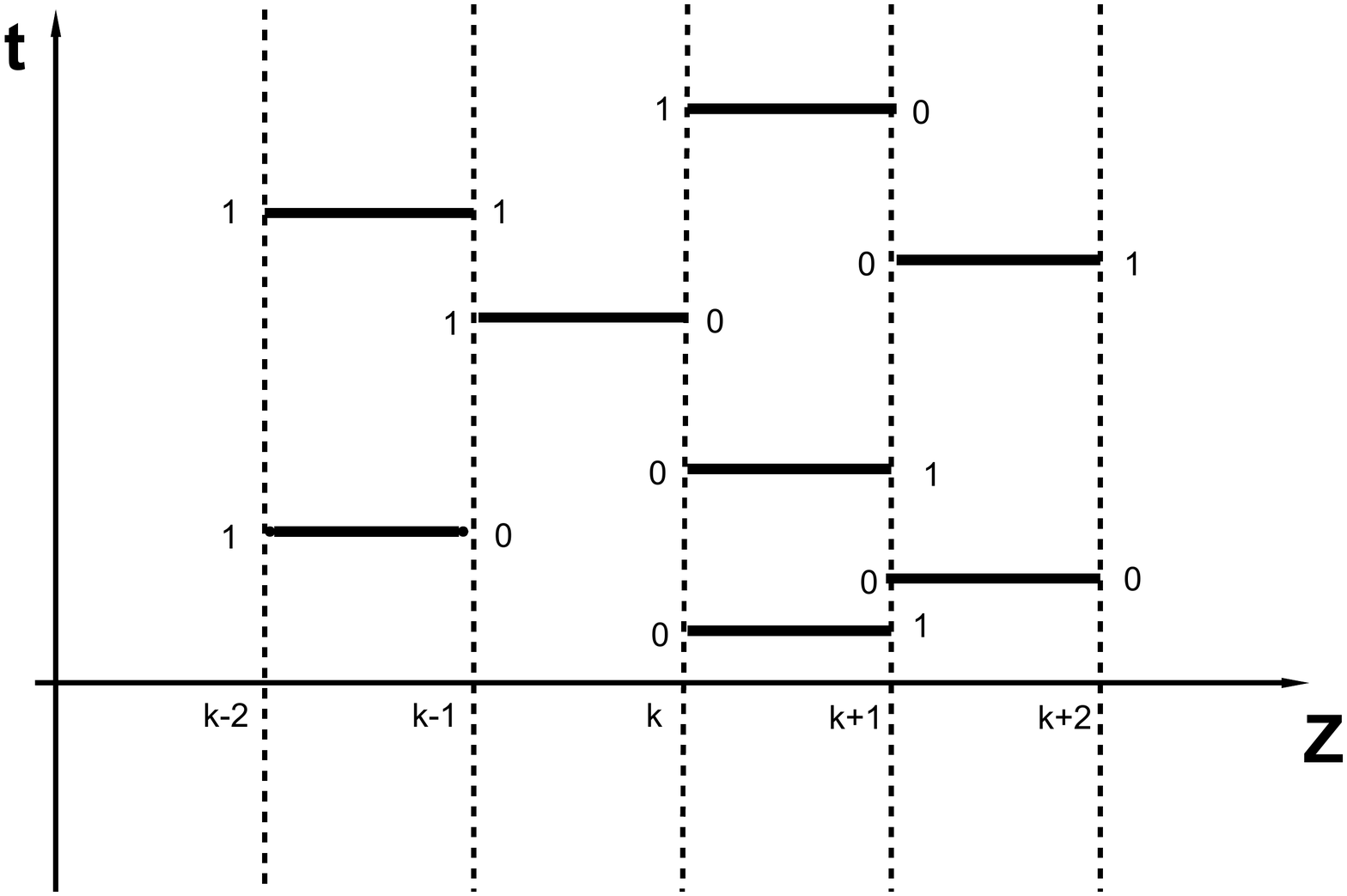}
\caption{Graphical construction of RRM: times of 
horizontal bars connecting vertical lines through $k$ and $k+1$ are a Poisson point process
with unit intensity. Poisson processes corresponding to different pairs of lines are independent.
Each horizontal bar is labelled with $(11),(10), (01),(00)$ with probabilities $\alpha_1,\alpha_2,
\alpha_3, \alpha_4$ independently of other lines. The current state at site $k$ at time
$t$ is given by the label at the last bar adjacent to $k$ dropped before time $t$, or the initial condition
if there is no such bar. 
\label{rrmgc}
}
\end{figure}

The evolution of the model can be visualised using the graphical construction
shown in Fig. \ref{rrmgc}. This construction gives a successful coupling of any two initial conditions, which 
implies that all solutions are attracted to a single invariant measure $(\xi_{\infty}(x): x \in \Z)$. 
A further
immediate consequence of the construction is that for
a fixed $t>0$
the random vectors $(\eta_t(k))_{k<x}$ and $(\eta_t(k))_{k>x}$ are independent
for any $x \in \Z$ provided the initial random vectors 
$(\eta_0(k))_{k<x}$ and $(\eta_0(k))_{k>x}$ are independent. This 
property is known as one-dependency. 
Starting with a deterministic initial condition we deduce that the invariant 
measure is one dependent.  
Moroever, a one-dependent point process on $\Z$ is also a determinantal point process, which follows 
from a general result proved in \cite{borodin2010adding}. By the Theorem 7.1 of that paper, 
the corresponding determinantal kernel is expressed in terms of the probability of runs of occupied sites,
which we now proceed to compute.  

To use the product moment function $g^{(n)}_{x_1,x_2,\ldots,x_n}$ as in (\ref{dualitytype3}) 
we need the action of $\sigma$ in the basis $e^{(1)} \otimes e^{(1)}$, $e^{(1)}\otimes v$,
$v\otimes e^{(1)}$, $v \otimes v$. The calculation is simplified since 
$\text{Image}(\sigma)$ is spanned by $v\otimes v$.  A direct computation of the relevant constants gives:
\bea\label{rrmga}
 \sigma v\otimes v = v\otimes v, & \qquad &
\sigma e^{(1)} \otimes e^{(1)}=\alpha_1 v \otimes v,\\
\sigma v\otimes e^{(1)}=(\alpha_1+\alpha_3)v\otimes v, & \qquad & 
\sigma e^{(1)}\otimes v=(\alpha_1+\alpha_2)v\otimes v.\nonumber
\eea
Let 
$$P_t(n)=\E\left[\prod_{k=1}^n\eta_t(k)\right]$$
be the probability of finding $n \geq 1$ particles at sites $1,2,\ldots,n$ at time $t$, and set $P_t(0) = 1$ for all $t$.
Using (\ref{rrmga}) one finds the following evolution equation for $P_t$
\beast
\frac{\partial P_t}{\partial t}(n)=2\rho P_t(n-1)+\alpha_1 \sum_{k=1}^{n-1}\E
\left[ \prod_{p=1}^{k-1}\eta_t(p) \prod_{q=k+2}^{n}\eta_t(q)\right]-(n+1)P_t(n), \quad n\geq 1
\eeast
where an empty product is defined to be $1$, and where we introduce the 
following parameter
\bea
\rho: =\frac{2\alpha_1+\alpha_2+\alpha_3}{2} \in [0,1].
\eea
Note that under particle-hole conjugation, $\rho \mapsto 1-\rho$. Therefore, without loss, we now take
 $\rho \in [0,1/2]$. 

Due to the one-dependence property established above, the random variables 
$ \prod_{p=1}^{k-1}\eta_t(p)$ and $\prod_{q=k+2}^{n}\eta_t(q)$ are independent, which
leads to the following non-linear equation for the probabilities of runs
\begin{equation} 
\frac{\partial P_t}{\partial t}(n)  =  2\rho P_t(n-1)\!+\!\alpha_1 \sum_{k=1}^{n-1}
P_t(k-1) P_t(n-k-1) - (n+1) P_t(n), \quad n\geq 1, \; t>0. \label{rrmrp-}
\end{equation}
The differential-difference equation (\ref{rrmrp-}) can be converted to a partial differential
equation using the generating function for the run probabilities defined as 
\bea\label{rrmgf}
G_t(z)=\sum_{n=0}^\infty P_t(n) z^n, \quad \mbox{for $|z| < 1$}.
\eea
In the invariant measure $G_t(z) = G_{\infty}(z)$ and (\ref{rrmrp-}) becomes 
\bea\label{rrmme}
(zG_\infty(z))' = 1+2\rho zG_\infty(z)+\alpha_1 (zG_\infty(z))^2,  \quad |z|<1, \quad G_\infty(0)=1.
\eea
It is useful to consider the cases $\alpha_1=0$ and $\alpha_1>0$ separately.
When $\alpha_1=0$ the equation (\ref{rrmme}) is linear and has solution
\bea\label{rrmgf0}
G_\infty(z)=\frac{1}{2\rho z}\left(e^{2\rho z}-1 \right)=\sum_{n=0}^\infty \frac{(2\rho z)^n}{(n+1)!}.
\eea
Therefore,
\bea\label{rrmrp}
P_\infty(n)=\frac{(2\rho)^n}{(n+1)!}, \quad \mbox{for $n\geq 0$.}
\eea
In particular, the particle density in the steady state is $P_\infty(1)=\rho$. Together with
the one dependence property, (\ref{rrmrp}) determines all the correlation functions of the invariant 
distribution $(\eta_{\infty}(x): x \in \Z)$. 
Indeed the determinantal structure gives, for $x_1<x_2< \ldots < x_m$ and $m \geq 1$,
\begin{equation} \label{dpp}
\pr \left[\eta_\infty(x_k)=1, k = 1, \ldots, m \right] =
\det \left[ K(x_k-x_j): j,k = 1, \ldots, m\right],
\end{equation}
where 
the translation invariant kernel $K: \Z \rightarrow \R$
is determined by the following inversion formula (Corollary 4.3 of \cite{borodin2010adding})
\bea
\sum_{n \in \Z} K(n)z^n & = & -\frac{1}{zG_\infty(z)} \label{inversionf} \\
& = & -\frac{2\rho}{e^{2\rho z}-1}=
-2\rho \sum_{n=-1}^\infty \frac{B_{n+1}}{(n+1)!}\left(2\rho z \right)^n, \nonumber
\eea
where $(B_n)_{n\geq 0}$ is the set of Bernoulli numbers defined via
the generator function 
$\frac{x}{e^x-1}=\sum_{n=0}^\infty \frac{B_n}{n!}x^n$.
We conclude that $K(n)=2\rho K_{2 \rho}^{(BDF)}(n), n \in \Z$, where $K_{a}^{(BDF)}$
is the translationally invariant kernel for the {\it descent process} introduced by Borodin, Diaconis and Fulman in \cite{borodin2010adding}:
\bea\label{rrmbdf}
K_a^{(BDF)}(n)=\left\{\begin{array}{cc}
0,& n\leq -2,\\
-\frac{B_{n+1} a^n}{(n+1)!}&n\geq -1.
\end{array}
\right.
\eea
Note that the parameter $a$ does not affect the law of the descent process since the determinants
in (\ref{dpp}) do not depend on $a \neq 0$ (there is a similarity transformation
by a diagonal matrix with diagonal entries $a^k$). We conclude that, when $\alpha_1=0$, 
the kernel for the reshuffle model can be taken to be the multiple, by $2 \rho \in [0,1]$, of the kernel for the
descent process. The point process is therefore a thinning of the  descent process. We recall a direct construction
of the descent process:  let $(U_k)_{k \in \Z}$ be the sequence of i.i.d. uniform random variables; the
descent process can be defined by 
\bea\label{rrmbdf1}
\eta^{(BDF)}_k=\ind\left(U_{k}>U_{k+1}\right), \quad k \in \Z.
\eea 
Correspondingly, the stationary state of the random reshuffle model
can be generated using the usual thinning formula,
\bea\label{rrmtp}
\eta_\infty(k)=Z_k \ind\left(U_{k}>U_{k+1}\right), \quad k \in \Z,
\eea
where $(Z_k)_{k \in \Z}$ is the sequence of i.i.d. Bernoulli($2 \rho$) random
variables independent of $(U_k)_{k \in \Z}$, 
%
%

When $\alpha_1>0$ the equation (\ref{rrmme}) is non-linear, but it is still solvable
using  elementary methods: 
let 
\bea\label{rrmroots}
r_{\pm} = -\rho \pm \sqrt{\rho^2-\alpha_1} \in \C,
\eea
be the roots of the polynomial $h^2+2\rho h+\alpha_1$; the solution is, for small $z$, 
\bea
zG_{\infty}(z) = 
\frac{e^{r_{-}z}-e^{r_{+}z}}{r_{-}e^{r_{-}z}-r_{+}e^{r_{+}z}}  \quad \mbox{when $r_{-} \neq r_+$,} 
\eea
and $zG_{\infty}(z) = \frac{x}{1+ rx}$ when $r_{-} = r_+ = r$.  
As it is easy to see, the generating function of run probabilities is real for real values 
of the argument and $G(0)=1$. Notice also that for $\alpha_1>0$ the probability of a run
of $1$'s decays exponentially with run length, rather than as $e^{-n\log n}$ in the $\alpha_1=0$
case, see (\ref{rrmrp}).

An application of the inversion formula (\ref{inversionf}) yields the
corresponding determinantal kernel. For $\alpha_1\in [0,\rho^2]$ the explicit answer is
\bea\label{rrmag0}
K(n)=(\rho-\sqrt{\rho^2-\alpha_1})\ind(n=0)+2\sqrt{\rho^2-\alpha_1}K^{(BDF)}_{2 \rho}(n),~n\in \Z.
\eea
To see how to construct the process (\ref{rrmag0}) from the descent process define
\beast
p:=\rho-\sqrt{\rho^2-\alpha_1},\quad \gamma:=\frac{2\sqrt{\rho^2-\alpha_1}}{1-p}.
\eeast 
As it is easy to check, $p \in [0,\fr]$, $\gamma \in [0,1]$. In terms of these parameters,
\bea\label{rrmag01}
K(n)=p\ind(n=0)+(1-p)\gamma K^{(BDF)}_{2 \rho}(n),~n\in \Z,
\eea
which implies that the stationary state of the random reshuffle model for 
$\alpha_1 \in [0,\rho^2]$ is a determinantal point process built as the union of a Bernoulli($p$) process on
$\Z$ with an independent $\gamma$-thinning of the descent process. 
More work is needed to characterise the stationary state for $\alpha_1>\rho^2$ when the roots (\ref{rrmroots})
become complex. 

\subsection{Erosion models: [EM1] and [EM2]} \label{s3.7}
In both models, the only changes that occur are flips $1 \to 0$, so the models are monotone decreasing.

The model [EM2] (\ref{em2}) is easiest to describe, as it is in fact a disguised formulation of a non-interacting model.
Indeed a comparison of the generators shows that in this model each site $x \in \Z$ evolves independently, 
flipping from $1 \to 0$ at rate $1$. We say no more. 

The other erosion model [EM1] (\ref{em1}) has no second factorized eigenvector. Although there is a product exponential moment duality
function (using $w = (-1,1)^T$ yielding $\prod_{i \leq n} (-1)^{\eta(x_i)}$), solving 
the closed equations for $\E[\prod_{i \leq n} (-1)^{\eta_t(x_i)}]$ is not that useful since the system has a very simple
probabilistic description. We can break the initial condition
into runs of successive $1$s, separated by runs of successive $0$s.  Zeroes never change and the process therefore 
decouples into the independent evolution of the runs of $1$s. A finite run $111 \ldots 1$ erodes only from the left, that is only the 
left most $1$ changes to a $0$ at rate $1$, reducing the length. Finite runs therefore eventually disappear, whilst a run that is semi-infinite 
on the right will erode from the left forever, and a run that is semi-infinite on the left is a fixed configuration. Thus the only possible 
limit points are the zero configuration, and the traps $\eta(x) = \ind(x \leq k)$ for $k \in \Z$.
\section{Algebra} \label{s4}
This Section contains a discussion of the algebraic
structures appearing naturally in the context of the Markov
processes we have studied. One aim is to examine in detail the models
on a finite set of sites  $\Z_N=\{1,2,\ldots,N\}$, with the corresponding 
generators $\sum_{i=1}^{N-1} q_i$ and their corresponding finite dimensional 
generator algebras $\A_N$. In particular
the irreducible representations of the 
finite dimensional generator algebras can sometimes be fully analysed. 

In Section \ref{s4.1} we consider the generator algebra 
corresponding
to annihilating random walks on $\Z_N$. 
We compute exactly the dimensions of the irreducible representations of 
this algebra. Originally, these dimensions were calculated for small $N$ using the 
computational algebra package `Magma' \cite{magma} and turned out to be
given by binomial coefficients. It was this fact that motivated
us to build these representations concretely as invariant subspaces
of $V^{\otimes N}$ spanned by eigenvectors. 

In Section \ref{s4.2} we explore whether the size of irreducible representations can
provide information on the solvability of the model. The existence as in Section
\ref{s4.1} of `small' (that is of polynomial size in $N$) irreducible representations
suggests there may be tractable duality functions. 
%

In Section \ref{s4.3} we    
recall the well known fact that 
the duality functions can be regarded as the intertwiners between two representations of  
the generator algebra. We use this observation to construct a co-ordinate representation of the infinite dimensional Hecke algebra in terms of the discrete Laplacians, which we have not seen in the literature.

In Section \ref{s4.4} we recall the well known fact that a presentation of a Hecke algebra can be used to construct a
solution to the Young-Baxter equation called an $R$-matrix. The R-matrix 
method is a cornerstone of another approach to integrable probabilistic models based on the 
diagonalisation of the transfer matrix. We give explicit examples of  R-matrices for some of the reaction-diffusion models on our 
classification list. 

\subsection{Duality functions and representations of the
generator algebra} \label{s4.1}
The considerations below apply to any of the reaction-diffusion
models on our list: [ACRW], [CSRWB], [ARWPI]. 
This is due to the fact that the two-site generators for each
model look identical if expressed in the corresponding $(v,w)$ basis.

For concreteness we will work with annihilating random walks [ARW], 
perhaps the most classical of the models and the one we considered in 
greatest detail in Section \ref{s1.1.2} for the non-symmetric case $r\neq l$.
To avoid dealing with exceptional cases, we will not consider the totally asymmetric
models and assume that $r>0, l>0$. 

The generator of [ARW] on $\Z_N$ with free boundary conditions is mapped, as in Section 
\ref{s2.1}, to an operator on $V^{\otimes N} = \otimes_{1 \leq k \leq N}V$ given by
\bea\label{asgen}
L_N=\sum_{k=1}^{N-1} (\sigma_k - I)
\eea 
as in (\ref{generator}) where $\sigma_{k}$ is the two site generator $\sigma$ from (\ref{sigmaone})
acting on $V_k \otimes V_{k+1}$. 
Let $\A_N$ be the generator algebra of the model. By definition, $\A_N$ is the algebra over $\R$
generated by $\{\sigma_k\}_{k=1}^{N-1}$.

Simulations using Magma for small values of $N$ suggest that the irreducible representations of $L_N$ have the following
dimensions
\bea
1^{(2)},(N-1)^{(2)}, \frac{(N-1)(N-2)}{2}^{(2)},
\ldots, {N-1 \choose k}^{(2)},\ldots, {N-1 \choose N-1}^{(2)}.
\eea
The superscript denotes the multiplicity of the corresponding
representation. As a consistency check, notice that
\[
2\sum_{k=0}^{N-1}{N-1 \choose k}=2^N=\dim V^{\otimes N}.
\]
Our aim is to construct these representations explicitly using
the duality functions constructed in Section \ref{s2.1}.
Using $v = (1,1)^T$ and $w=(1,-1)^T$ 
we define the following elements of $T_N$:
\bea
f_{x_1,x_2}=v_1\otimes \ldots \otimes v_{x_1}\otimes w_{x_1+1}\otimes \ldots \otimes w_{x_2}
\otimes v_{x_2+1}\otimes \ldots \otimes v_{N},\quad 1\leq x_1\leq x_2\leq N,\\
g_{x_1,x_2}=w_1\otimes \ldots \otimes w_{x_1}\otimes v_{x_1+1}\otimes \ldots \otimes v_{x_2}
\otimes w_{x_2+1}\otimes \ldots \otimes w_{N},\quad 1\leq x_1\leq x_2\leq N.
\eea
Next define
$U^{(k)}, W^{(k)}\subset V^{\otimes N}$ for $0\leq k\leq N-1$ as follows:
\bea
U^{(2n)}&=&\mbox{Span}_{\R}\left\{ \prod_{k=1}^n 
f_{x_{2k-1},x_{2k}}, 0\leq x_1\leq x_2\leq \ldots \leq x_{2n}\leq N \! \right\}, \\
U^{(2n-1)}\!\!\!\!\!\!&=&\!\!\!\!\mbox{Span}_{\R}\!\!\left\{\! \prod_{k=1}^{n-1} \!
f_{x_{2k-1},x_{2k}} f_{x_{2n-1},N}, 0\leq x_1\leq x_2\leq \ldots \leq x_{2n-1}\leq N \!\right\},\\
W^{(2n)}&=&\mbox{Span}_{\R}\left\{ \prod_{k=1}^n 
g_{x_{2k-1},x_{2k}}, 0\leq x_1\leq x_2\leq \ldots \leq x_{2n}\leq N\!\right\},\\
W^{(2n-1)}\!\!\!\!\!\!&=&\!\!\!\!\mbox{Span}_{\R}\!\!\left\{ \!\prod_{k=1}^{n-1} \!
g_{x_{2k-1},x_{2k}} g_{x_{2n-1},N}, 0\leq x_1\leq x_2\leq \ldots \leq x_{2n-1}\leq N\!\right\},
\eea
where $0\leq 2n \leq N-1$ or $1\leq 2n-1 \leq N-1$ and the Hadamard products are used.
Let us refer to the configurations of the form $v\otimes w$ and $w\otimes v$ as {\em jumps}.
For example, the vector $v\otimes v \otimes w\otimes w \in V^{\otimes 4}$  has one jump,
the vector  $v\otimes w \otimes v\otimes w \in V^{\otimes 4}$ has three jumps. Both $U^{(k)}$
and $W^{(k)}$ are spanned by duality functions with at most $k$ jumps.

As a direct consequence of our definitions,
\[
W^{(k-1)}, U^{(k-1)}\subset U^{(k)}, \quad  W^{(k-1)}, U^{(k-1)}\subset W^{(k)},
\quad \mbox{for $1\leq k\leq N-1$.}
\]
Note that $U,W$ spaces are spanned by the duality functions. Therefore, it follows from 
(\ref{dual1}, \ref{dual2}) that
these spaces are $\A_N$-invariant:
\[
\A_N W^{(k)} \subset W^{(k)}, \quad \A_N U^{(k)}\subset U^{(k)},  \quad \mbox{for $1\leq k\leq N-1$.}
\]
So the action of  $\A_N$ is well defined on the quotient spaces
\[
P^{(k)}=U^{(k)}/\left(U^{(k-1)}+ W^{(k-1)}\right),
~Q^{(k)}=Q^{(k)}/\left(U^{(k-1)}+ W^{(k-1)}\right),
\]
namely if $a\in \A_N$, $u \in U^{(k)}$ and $[u]\in P^{(k)}$ then
$
a[u]:=[au].
$
Moreover, due to the invariance of the $U,W$ spaces, the quotient spaces are also $\A_N$-invariant:
\[
\A_N P^{(k)}\subset P^{(k)}, \quad \A_N Q^{(k)}\subset Q^{(k)}, 
\quad \mbox{for $1\leq k\leq N-1$.}
\]
The quotient spaces $P$ and $Q$ can be explicitly constructed as follows:
\bea
P^{(2n)}&=&\mbox{Span}_{\R}\left\{ [\prod_{k=1}^n 
f_{x_{2k-1},x_{2k}}], 1\leq x_1< x_2< \ldots < x_{2n}< N\!\right\}, \label{p2n}\\
P^{(2n-1)}\!\!\!\!\!\!&=&\!\!\!\!\mbox{Span}_{\R}\!\!\left\{\! [\prod_{k=1}^{n-1} \!
f_{x_{2k-1},x_{2k}} f_{x_{2n-1},N}], 1\leq x_1\!<\! x_2\!<\! \ldots \!<\! x_{2n-1}\!<\! N\!\right\},
\label{p2n-1}\\
Q^{(2n)}&=&\mbox{Span}_{\R}\left\{ [\prod_{k=1}^n 
g_{x_{2k-1},x_{2k}}], 1\leq x_1< x_2< \ldots < x_{2n}< N\!\right\},\\
Q^{(2n-1)}\!\!\!\!\!\!&=&\!\!\!\!\mbox{Span}_{\R}\!\!\left\{ \![\prod_{k=1}^{n-1} \!
g_{x_{2k-1},x_{2k}} g_{x_{2n-1},N}], 1\leq x_1\!<\! x_2\!<\! \ldots \!<\! x_{2n-1}\!<\! N\!\right\},
\eea
where $0\leq 2n \leq N-1$ or $1\leq 2n-1 \leq N-1$. In words, the elements of $P^{(k)}$
consist of equivalence classes of elements of $U^{(k)}$ of the 
form $v\otimes \ldots$ with exactly $k$ jumps . The elements of $Q^{(k)}$
consist of equivalence classes of elements of $V^{(k)}$ of the 
form $w\otimes \ldots$ with exactly $k$ jumps. 

Given such a characterisation, the dimensions of the $(P,Q)$-spaces are easy to calculate:
\[
\dim P^{(k)}=\dim Q^{(k)}={N-1 \choose k},
\quad \mbox{for $0 \leq k\leq N-1$.}
\] 
which coincide with (\ref{asgen}).

In fact, each of the invariant spaces
$P^{(n)}, Q^{(n)}$ for $0\leq n\leq N-1$ is irreducible. In particular,
$V^{\otimes N}=\oplus_{k=0}^{N-1}\left(P^{(k)}\oplus Q^{(k)}\right)$, which is
a representation of $V^{\otimes N}$ as the direct sum of irreducible
representations of $\A_N$. Namely, one has the following statement:
\begin{proposition}
The generator algebra $\A_N$ is semi-simple. The representation of $\A_N$
in $P^{(k)}, Q^{(k)}$ is irreducible for $1\leq k\leq N-1$. 
\end{proposition}
\begin{proof}
It has been already established that the generators $\{\sigma_k\}_{k=1}^{N-1}$
obey the relations 
(\ref{braidreln},\ref{commutereln},\ref{quadraticreln}) 
of (a specialization of) type-A Hecke algebra with the parameter $Q=rl$. 
(Various constants used to parameterise Hecke algebras 
are reviewed in the Appendix.) 
Therefore, $\A_N$ is a quotient
algebra of Hecke algebra. 
As $r+l=1$, $Q$ takes values in $(0,\frac{1}{4}]$. This implies  that the parameter
$q^2\in (0,1]$, where $q$  is related to $Q$ via (\ref{constQ}). The
Hecke algebra over $\R$ generated by $\{s_i\}_{i=1}^{N-1}$ with $q^2 \in (0,1]$
(see (\ref{hecke_sym}) for the relations) is semi-simple, see \cite{andersen2017semisimplicity}, Theorem $5.1$.\footnote{As far as Hecke algebras are concerned, the paper  \cite{andersen2017semisimplicity} generalises classical results of \cite{gyoja1989semisimplicity}
to arbitrary fields. Notice that there is a misprint in the cited Theorem: the correct statement
is obtained by the replacing $q$ with $q^2$.}

Therefore, the generator algebra $\A_N$ is semi-simple as a quotient of a semi-simple
algebra. 

Next, let us fix any $n: 1\leq n\leq N-1$.  
It is sufficient to prove the irreducibility of $P^{(n)}$, the result for $Q^{(n)}$ follows
from the symmetry  $v\leftrightarrow w$, $l\leftrightarrow r$. The irreducibility
will be established by brute force by checking that any
representation endomorphism of $P^{(n)}$ is proportional to the identity operator. In this
case, due to the semi-simplicity of the algebra $\A_N$, the converse to
Schur's lemma  guarantees that $P^{(n)}$ is irreducible.

Let $U \in \mbox{End}(P^{(n)})$ be a representation
isomorphism, a linear map from $P^{(n)}$ to $P^{(n)}$ which commutes with the representation of $\A_N$ in $P^{(n)}$.
In other words,
\bea\label{asrepend}
q_{x} U=Uq_{x},~1\leq x\leq N-1.
\eea
It follows from (\ref{opaction}), that
the matrices of the generators of $\A_N$ in the basis (\ref{p2n}) or (\ref{p2n-1}) of $P^{(n)}$
have the following form: 
\bea
(q_x)_{\xv, \yv}=\sum_{k=1}^{n}
\delta_{x,x_k}\prod_{p\neq k} \delta_{y_p,x_p}
\left(\ell \delta_{y_k,x+1}+r\delta_{y_k,x-1}-\delta_{y_k,x} \right),\\ 1\leq x\leq N-1, \xv,\yv\in W_n^{(N)}.
\nonumber
\eea
Here $W_n^{(N)}=\{1\leq x_1<x_2< \ldots x_n\leq N-1\}$
and the elements of $W_n^{(N)}$ are denoted by vectors, 
e.g. 
$\xv=(x_1, x_2, \ldots, x_n)$. 
Let $(U_{\xv, \yv})_{\xv, \yv \in W^{(N)}_n}$
be the matrix of the linear operator $U$ in the basis (\ref{p2n}) or (\ref{p2n-1}).
Using the fact that $[g]=0$ for any $g \in U^{(n-1)}+W^{(n-1)}$, 
we can extend the 
definition of $U$ by assuming that 
$U_{\xv, \yv}=0$  if either $\xv$ or $\yv$ is not in $W^{(N)}_n$. 
The matrix form of (\ref{asrepend}) is
\bea\label{asme}
\sum_{k=1}^n \delta_{x,x_k} \left(\ell U_{\xv+\iv_k,\yv}
+rU_{\xv-\iv_k,\yv}-U_{\xv,\yv} \right)=
\sum_{k=1}^n U_{\xv,\yv}\mid_{y_k=x} 
\left(\ell \delta_{y_k,x+1}+r\delta_{y_k,x-1}-\delta_{y_k,x} \right),
\eea
 where $\xv, \yv \in W^{(N)}_n$ and $I_k$ is an
 $n$-dimensional vector with the components
 $(\iv_k)_\ell=\delta_{k,\ell}$.
 
 Let us pick any $x \neq x_k$. Then the left hand side of (\ref{asme}) vanishes and one finds the following relation:
 \beast
 \sum_{k=1}^n U_{\xv,\yv}\mid_{y_k=x} 
\left(\ell \delta_{y_k,x+1}+r\delta_{y_k,x-1}-\delta_{y_k,x} \right)=0.
 \eeast 
 Next fix $m$ and set $y_m=x$. The only non-zero term
 in the above sum corresponds to $k=m$ as for
 $k\neq m$, $U_{\xv,\yv}\mid_{y_k=x, y_m=x}=0$ due to
 $\yv\mid_{y_k=y_m=x} \notin W^{(N)}_n$. 
 Therefore, the last displayed equality reduces to
 $U_{\xv,\yv}\mid_{y_m=x}=0$. We proved that
 \[
 U_{\xv,\yv}=0,
 \]
 provided that there is $m:~1\leq m\leq n$: $y_m
 \notin\{x_1, x_2, \ldots, x_n\}$. If such $m$ does not
 exist, then the components of $\yv$ are a permutation
 of the components of $\xv$. But the components of
 $\yv\in W^{(N)}_n$ are ordered, therefore the only such 
 permutation is the identity permutation and $\xv=\yv$.
 We conclude that $U_{\xv,\yv}=0$ for any $\xv\neq \yv$,
that is the matrix of $U$ is diagonal. Using this diagonal
 Ansatz in (\ref{asme}), one finds:
 \beast
&& \hspace{-.3in}
 \sum_{k=1}^n\delta_{x,x_k} 
 \prod_{p\neq k}\delta_{y_p,x_p}
 \left(\ell \delta_{y_k,x_k+1}+
 r\delta_{y_k,x_k-1}\right)U_{\yv,\yv} \\
 & = & U_{\xv,\xv}\sum_{k=1}^n
\delta_{x,x_k}\prod_{p\neq k} \delta_{x_p,y_p}
\left(\ell\delta_{y_k,x_k+1}+r\delta_{y_k,x_k-1}\right).
 \eeast
 Equating the coefficients
 in front of similar products of $\delta$-functions in the above 
 identity shows that
 \[
 U_{\xv+\iv_k,\xv+\iv_k}=U_{\xv,\xv}
 =U_{\xv-\iv_k,\xv-\iv_k} \quad \mbox{for $1\leq k\leq n$.}
 \]
 Iterating the above identity we conclude that the diagonal 
 matrix $U$ is constant along the diagonal.
Therefore, $U=cI$ for some constant $c$ and the proposition is proved.
 \end{proof}
{\bf Remark.} At present we do not have a full set of relations obeyed
the generators $\{q_k\}_{k=1}^{N-1}$ characterising the generator algebra
$\A_N$. Basing on the dimensions of irreducible representations calculated
above we conjecture that $\A_N$ is isomorphic to the centraliser algebra
$\mbox{End}_{U_q(gl(1|1))}(V^{\otimes N})$ analyzed in \cite{benkart2013planar}.
\subsection{On exact solvability and representation theory} \label{s4.2}
The model [ARW] on $\Z_N$ in the previous Section
is an example of  interacting particle system where 
the generator algebra has 'small' irreducible representations ("irreps"), 
growing polynomially with $N$,
where these irreps can be constructed using bases of duality functions, and where these duality functions are sufficient to determine the fixed time law. 
It is therefore natural to ask whether more generally 
the presence of 'small' irreps of the generator algebra signals the integrability
of the model. We examine three more models in this perspective, the final example
indicating that the presence of such small irreps does not seem to be sufficient for 
integrability of the model on $\Z$. We used the symbolic algebra package Magma to
calculate the dimensions of the irreps for small $N$.
%
%

Consider the random reshuffle model [RM] from (\ref{rm}) on the finite lattice $\Z_N$.
Magma experiments for small $N$ suggest that all irreps of the corresponding 
generator algebra $\A$ are one-dimensional. 
This can be easily confirmed using duality functions: 
the products of particle indicators constitute a complete set of duality 
functions. Indeed for
%
%
%
%
$b \in \{0,1\}^{\Z_N}$ let $|b|_1:=\sum_{k=1}^N b_k$ be its Hamming weight; let $m_b$
be the test function on $\eta \in \Z_N$ defined by  $m_b(\eta)=\prod_{k: b_k=1} \eta_k$;
and let $M_t(b) = \E[m_b(\eta_t)]$ be the corresponding product moment. 
A generator calculation shows that 
\begin{equation} \label{rmgencalc}
\partial_t M_t(b)=\lambda(b)M_t(b)+\sum_{c\in \Z_N: |c|_1< |b|_1}\mu(b,c)M_t(c),
\end{equation}
where $\lambda$ and $\mu$ are functions on $\Z_N$ and $\Z_N\times \Z_N$
whose explicit form we do not need.
 In other words, the product moments satisfy closed linear inhomogeneous 
ODE's with  the right hand side determined by product moments of lower weights. 
We now check that this structure meshes with the irreps of the generator algebra. For any
$b \in \Z_N$ let
\[
g_b:=a_1\otimes a_2\otimes \ldots \otimes a_N \in V^{\otimes N},
\] 
where $a_i=(1,1)^T$ if $b_i=0$ and $a_i=(1,0)^T$ if $b_i=1$. Then $g_b$ is the image
of the duality function $m_b$ under the isomorphism between test functions
and $V^{\otimes N}$ (as in Section \ref{s1.1.1} but restricted to $N$ sites). 
Define $\overline{B}_b \subset \Z_N$
to be the set of binary strings which contains contains $b$ as well as all the strings 
$c$ such that $c_k \leq b_k$ for $1 \leq k \leq N$.
Let $B_b:=\overline{B}_b \setminus \{b\}$. Define
\[
\overline{W}_b:=\mbox{Span}_\R\{m_c: c \in \overline{B}_b\}, \quad 
W_b:=\mbox{Span}_\R\{m_c: c \in B_b\}, \quad \mbox{for $b \in \Z_N$.}
\]
It follows from (\ref{rmgencalc}) that the spaces $W_b$ and $\overline{W}_b$ are invariant
with respect to the action of the generator algebra. By construction, 
$W_b\subset \overline{W}_b$.
So the generator algebra acts on the quotient space $P_b:=\overline{W}_b/W_b$ for
$b\in \Z_N$ and $\dim P_b=1$. Thus we have constructed $2^N$ one-dimensional 
irreps of the random reshuffle algebra using the duality functions. 

An example coming from the opposite end of the spectrum is served by the contact process
defined by the two-site generator of the form
\beast
q = \sigma-I =\left(
\begin{array}{cccc}
-2a&a&a&0\\
b&-a-b&0&a\\
b&0&-a-b&a\\
0&0&0&0
\end{array}
\right).
\eeast
Here $a$ is the recovery rate and $b$ is the infection rate. The contact process is 
self dual, and the duality functions, indexed by subsets of $\Z$, have been used for 
studying various properties (such as ergodic behaviour). However,
the contact process is not expected to be integrable and we are not 
aware of a system of duality functions leading to explicit formulae.
This lack of the integrability seems to mesh well with 
the irreps of the contact process:  Magma experiments for small $N$ 
suggest that its generator algebra has just two irreps in $V^{\otimes N}$, 
the trivial one-dimensional representation and a second representation of dimension $2^N-1$.
We expect a similar picture to be the generic situation for a typical stochastic $\sigma$. 

Our final example illustrates that the presence of small irreps does always lead to useful duality functions. Note that the 
generator algebra of any interacting particle system on $\Z_N$ such that the number of particle
is either conserved of decreases has irreps of dimensions growing polynomially with
system size, corresponding to the subspaces of any given number of particles. 
However, we do not expect all such systems to be integrable. For example
the dimer model [DM] can be mapped to an inhomogeneous symmetric exclusion model
(see Section \ref{s3.5}) where particle numbers are conserved. It is not surprising then that 
MAGMA shows irreps of polynomial size. However, the
only duality functions we managed to find for this model are given by the indicators 
$f_b(\eta) = \ind(\eta =b)$ for $b \in \{0,1\}^{\Z_N}$. The
expectations of $f_b(\eta_t)$'s are just the transition probabilities.
These do obey closed
dual equations (the Kolmogorov equation), but become trivial
in $N\rightarrow \infty$ limit corresponding to 
the infinite model. Indeed, for a system with infinitely many particles the
transition density $\E[f_b(\eta_t)]$ is identically zero for any realisation
$b$ with infinitely many particles. 
We do not know how to characterise the distribution of the infinite dimer model.
In comparison, for [ASEP], where particle numbers are also preserved, there was
an alternative useful set of duality functions, the staircase functions, that did extend 
usefully to (some) infinite systems. 
%
\subsection{Co-ordinate representation of a Hecke algebra} \label{s4.3}
%
The duality functions studied above have two important properties:
firstly, they
 appear as a basis of $T(\Omega)$, for which
the matrices of the two-site generators $\sigma_n$, $n \in \Z$ are block lower-triangular. 
Secondly, the action of the generator on these  functions is modelled
by certain difference operators acting on the spatial arguments. 
This second property is important from the point of view of Markov duality.
More pragmatically, as a consequence of the second property
the expectations of duality functions
$\E_{\eta_0}[f_{x_1,x_2, \ldots, x_n}^{(n)}(\eta_t))]$ satisfy linear evolution equations 
with respect to independent variables
$t\in \R_{+},x_1, \ldots x_n \in \Z$. For example, for the reaction-diffusion
models on $\Z$ it is easy
to guess the solution to these equations in terms of Pfaffians built out of  
$\E_{\eta_0}[f_{x,y}^{(2)}(\eta_t)]$, $x,y \in \Z$. If however we place our particle system
on a circle $\Z_N$ for some $N<\infty$ rather than the integer lattice, the naive Pfaffian ansatz fails. It is still 
unknown what the law of our reaction-diffusion systems 
on a circle is, so it is reasonable to hope that
an algebraic structure of the evolution operator will be useful for determining
such a law. An algebraic interpretation of the 
 duality functions for models in Hecke class is that of  intertwiners between two representations
of Hecke algebra: in the space $T(\Omega)$ of test functions and the space of functions
of several spatial variables.  
This implies that the relevant evolution operators belong to (a closure of) a representation  
of Hecke algebra, which in turn could enable a systematic search for a basis diagonalizing the evolution
operators using the corresponding representation theory. Notice, that 
the role of Markov duality functions as intertwiners 
has already been emphasized by many authors, see for example \cite{giardina2009duality}.

Our aim here is modest: we will just build an explicit co-ordinate representation of Hecke algebras
intertwined with the representation in $T(\Omega)$ we started with. This representation seems
to be different from other known representation of Hecke generators using difference 
operators, e. g.  the polynomial representation used to construct dualities
for multi-species asymmetric exclusion models in \cite{chen2020integrable}. 
The study of its properties and applications to interacting particle systems is a matter for future research.  

Let us start with the most basic definition: let $\pi_1, \pi_2$ be two linear representations of an associative algebra $\mathbb{A}$ in linear spaces $V_1$, $V_2$
correspondingly. A linear map $U:V_1\rightarrow V_2$ is said to intertwine $\pi_1$ with $\pi_2$ if for any $a \in \mathbb{A}$,
\bea\label{intereq}
U \circ \pi_1(a)= \pi_2(a)\circ U.
\eea
In other words, the diagram shown in Figure \ref{intwn} is commutative.
\begin{figure}[h!]\label{intwn}
 \centering
  \includegraphics[width=0.5\textwidth]{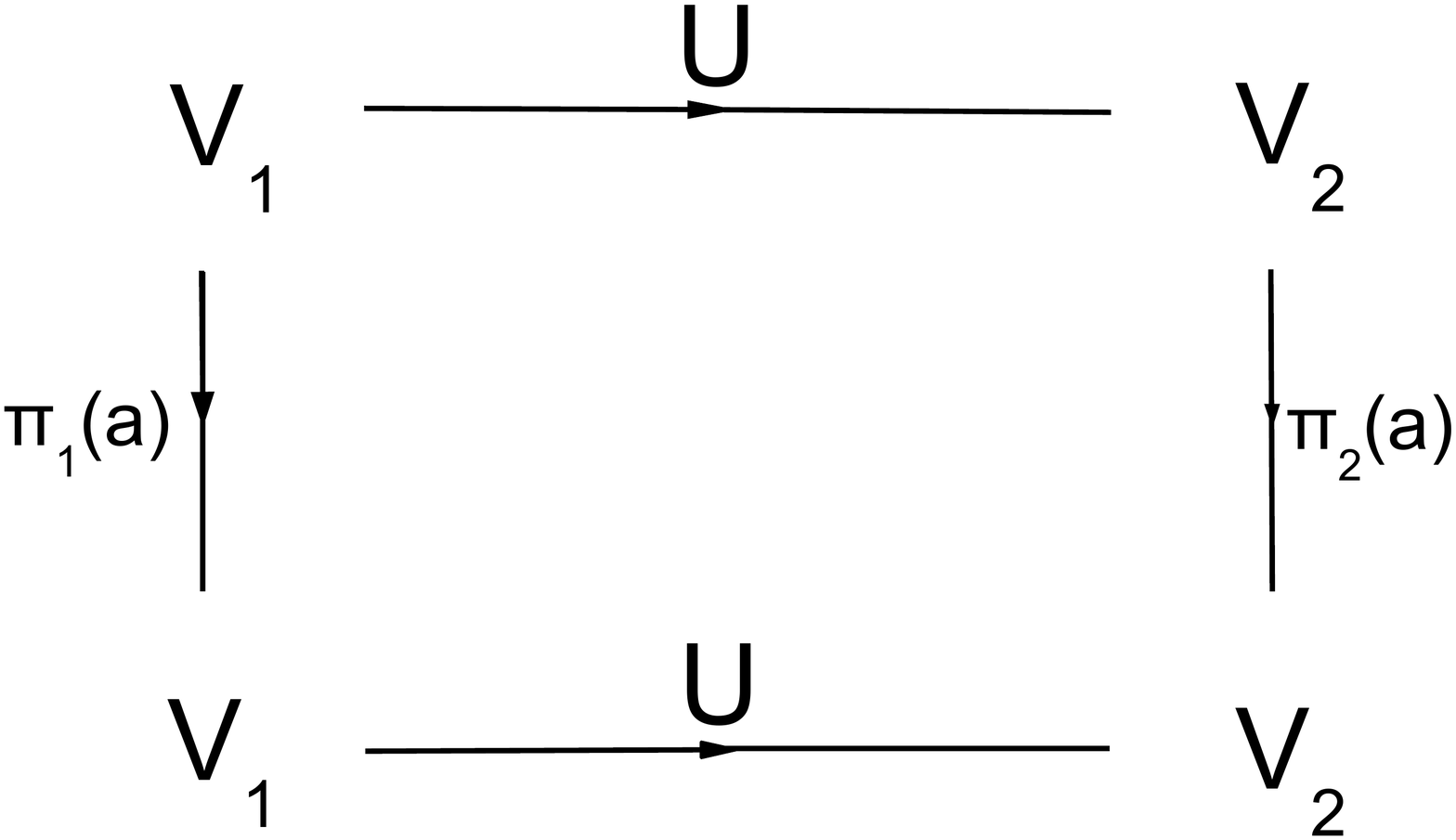}
 \caption{If the diagram above is commutative, the map $U:V_1 \rightarrow V_2$ is called an intertwiner.}
 \end{figure}
Assume that $V_1, V_2$ are $L_2$ spaces of functions on some 
spaces $\Omega_1$, $\Omega_2$ parametrised by co-ordinates
$x$ and $\eta$ correspondingly. Then the relation (\ref{intereq}) takes the form
\bea\label{intereq1}
\pi_2(a) U(\eta, x)=(\pi_1(a))^\dagger U(\eta,x), ~\forall x \in \Omega_1, \eta \in \Omega_2.
\eea
where $U(\eta,x)$ is the kernel of the map $U:V_1\rightarrow V_2$ 
in the chosen basis, the operator $\pi_1(a)$ acts on the $x$ co-ordinates
of $U(\eta, x)$, the operator $\pi_2(a)$ - on the $\eta$ co-ordinates.
The adjoint operator $(\pi_1(a))^\dagger$ is defined with respect to the $L_2$
inner product on $V_1$. 

There is a useful application of (\ref{intereq}) provided the intertwiner $U$ satisfies some
non-degeneracy conditions. Let $\{a_1, a_2, \ldots, a_n\}\subset \text{End}(V_2)$. Suppose
that operators $a$ satisfy a polynomial relation  $R(a_1, a_2, \ldots, a_n)=0$. Assume
that $\{\hat{a}_1, \hat{a}_2, \ldots, \hat{a}_n\}\subset \text{End}(V_1)$ are such that
\beast
a_k U(\eta,x)=\hat{a}_k^\dagger U(\eta,x),~k=1,2,\ldots,n.
\eeast
Finally, suppose that the set $\{U(\eta,\cdot)\}_{\eta \in \Omega_2}$ spans $V_1$.
Then $R(\hat{a}_1, \hat{a}_2, \ldots, \hat{a}_n)=0$. In other words, the set of relations
satisfied by operators $\{a_k\}_{k=1}^n$ is a subset of relations satsfied by the operators
$\{\hat{a}_k\}_{k=1}^n$.

Using the above elementary algebraic preliminaries it is easy
to prove the following statement:
\begin{lemma}\label{lm_int_1d}
Let $L_2(\Z)$ be the space of square summable functions on $\Z$.
Define the following set of linear operators on $L_2(\Z)$:
\bea\label{repcoord1d}
\hat{q}_x=\ind_x\Delta,~x \in \Z,
\eea
where $\Delta=rD+lD^{-1}-1$ is the discrete Laplacian. Then
$\{\hat{q}_x\}_{x \in \Z}$ generate a representation of Temperley-Lieb
algebra,
\bea\label{TL_markov}
\left\{
\begin{array}{ll}
\hat{q}_{i}\hat{q}_j=\hat{q}_j \hat{q}_i,&|i-j|\geq 2,~ i,j\in \Z\\
\hat{q}_{i}\hat{q}_{i\pm 1}\hat{q}_{i}-Q\hat{q}_i
=0,&i\in \Z \\
\hat{q}_i^2=-\hat{q}_i,&  i\in \Z
\end{array}
\right.
,
\eea
where 
$Q=rl$.
\end{lemma} 
A multi-dimensional generalisation is as follows. 
Let $W^{n}=\{z_1< z_2<\ldots<z_n:z_k \in \Z, 1\leq k \leq n \}$.
Let
$L_2(W^n)$ be the space of square-summable functions on $W^{n}$
subject to the following boundary conditions: if $f\in L_2(W^n)$, then
\bea
D_yf(\cdot, \ldots,\cdot, y_i=y,y_{i+1}=y, \cdot, \ldots, \cdot) =0.
\eea
In other words, the function $f\mid_{y_i=y_{i+1}=y}$ does not depend on $y$. Then
\begin{lemma}\label{lm_int_md}
The following operators acting on $L_2(W^n)$, $n\geq 1$, generate a representation
of Hecke algebra $H_\infty$ defined in the Appendix:
\bea\label{repcoordmd}
\hat{q}_x=\sum_{k=1}^n \ind_x^{(k)}\Delta_k,~x \in \Z,
\eea
where $\Delta_k$ is the discrete Laplacian acting on the $k$-th co-ordinate of a function
on $L_2(W^n)$; if $y=(y_1, y_2, \ldots, y_n)\in W^n$, then  $\ind_x^{(k)}(y)=\ind_x(y_k)$.
Namely,
\bea\label{hinf_markov}
\left\{
\begin{array}{ll}
\hat{q}_{i}\hat{q}_j=\hat{q}_j \hat{q}_i,&|i-j|\geq 2,~ i,j\in \Z\\
\hat{q}_{i}\hat{q}_{i+1}\hat{q}_{i}-Q\hat{q}_i
=\hat{q}_{i+1}\hat{q}_i \hat{q}_{i+1}-Q\hat{q}_{i+1},& i\in \Z\\
\hat{q}_i^2=-\hat{q}_i,&  i\in \Z
\end{array}
\right.
,
\eea
where $Q=rl$.
\end{lemma}
The proof of the above lemmas is a straightforward exercise given the following remarks:
\begin{enumerate}
\item Let $f: \Omega \times Z\rightarrow \R$: $(\eta,x)\mapsto f_x(\eta)=(-1)^{\sum_{i\leq x}\eta_i}$.
\footnote{For simplicity we assume for this Section that the 
elements of $\Omega$ have finite Hamming weight, which makes $f$ well defined.}
Then it follows from (\ref{opaction}), (\ref{dislap}) that 
$q_y f_x(\eta)=\hat{q}^\dagger_y f_x(\eta)$, for all $y\in \Z$,
$(\eta,x)\in \Omega\times \Z$ where $\hat{q}_x^\dagger =\ind_x \Delta$.
\item $L_2(\Z)$ is in the span of $\{f_\cdot(\eta)\}_{\eta \in \Omega}$:

let $\eta_y=(\ldots, \eta_{y-3},\eta_{y-2},\eta_{y-1},\eta_y,0,0\ldots)$ and 
$\eta'_y=(\ldots, \eta_{y-3},\eta_{y-2},1-\eta_{y-1},1-\eta_y,0,0\ldots)$.
Then 
\beast
f_{x}(\eta_y)-f_{x}(\eta'_y)=\pm 2 \ind_{y-1}(x), \quad \mbox{for $x \in \Z$.}
\eeast
\item So the algebraic argument given above implies that the operators 
$\{\hat{q}\}_{x \in \Z}$ satisfy the relations of Hecke algebra. The fact that 
each side of the ternary relation is zero can be confirmed using an exact computation.
Finally, the relations of both Hecke and Temperley-Lieb algebras 
are invariant with respect to conjugation
and the statement of Lemma \ref{lm_int_1d} follows.
\item The statement of Lemma \ref{lm_int_md} can be confirmed using a
similar argument. The claim can be verified also by a direct computation based on the fact that the operators
$\hat{q}_x^{(k)}:=\ind_x^{(k)}\Delta_k, x \in \Z, 1\leq k \leq n$ satisfy Temperley-Lieb
relations (\ref{TL_markov}) for coinciding values of the superscript. For different
values of the superscript, the operators commute and moreover 
$\hat{q}_{x_1}^{(k_1)}\hat{q}_{x_2}^{(k_2)}=0$ for $k_1>k_2$, $x_1\leq x_2$. The confirmation
of the last statement requires an application of the boundary conditions.   
\end{enumerate}
\subsection{Baxterisation and R-matrices with non-conserved
particle number} \label{s4.4}
Given a representation of  Hecke algebra, one can generate solutions to the Young-Baxter
equation (YBE) via the so called `Baxterisation' procedure discovered in \cite{jones}. Here we will use Baxterisation simply to write down explicit solutions to YBE which are associated with reaction-diffusion systems. These processes do not
preserve the number of particles. As a result, the corresponding solutions to YBE, the so-called $R$-matrices, 
are quite distinct from the $R$-matrices associated with models in KPZ universality class such as [ASEP]. The purpose
of the present Section is to list the $R$-matrices for reaction-diffusion systems explicitly.

First, let us review the procedure following the conventions of \cite{isaev}, \cite{pyatov}, see the 
Appendix
for more details. 
Recall that the braid group $B_{M+1}$
is generated by elements $\{s_i\}_{i=1}^M$ subject to the relations
\bea\label{isrel}
\begin{array}{l}
s_i s_{i+1} s_{i}=s_{i+1}s_i s_{i+1},~ 1\leq i \leq M-1\\
s_is_j=s_j s_i,~ |i-j|>1,~1\leq i,j \leq M
\end{array}.
\eea 
The type-$A$ Hecke algebra is the quotient of the group algebra of $B_{M+1}$ by the {\em Hecke relation}:
\bea\label{ishecke}
s_i^2=1+(q-q^{-1})s_i, ~ 1\leq i \leq M, q \in \C\setminus 0.
\eea
(For applications to the interacting particle systems we are dealing with it is enough to assume that
$q$ is a positive integer between $0$ and $1$, see the Appendix.)
We have the following claim:
\begin{theorem}(V. Jones, 1990)
Let
\bea\label{ybsol}
R_n(x):=s_n-x s_n^{-1},~1\leq n\leq M.
\eea
where $x \in \C$ is a spectral parameter. Then
\bea\label{isqybe}
R_n(x) R_{n-1}(xy) R_{n}(y)=R_{n-1}(y)R_n(xy) R_{n-1}(x),~2\leq n \leq M,~ x,y \in \C.
\eea
In other words, $R_n(x)$'s solve the Yang-Baxter equation. 
\end{theorem}
The proof of this statement is just a computation based on (\ref{isrel}), which is assisted by the
following explicit formula for the inverse of the generators:
\bea\label{heinv}
s_n^{-1}=-(q-q^{-1})+s_n, ~1\leq n \leq M.
\eea
\\
{\bf Remark.}  Let
$
\frak{R}_n(u,v)=R_n\left(\frac{u}{v}\right),~y \neq 0$.
Then
\[
\frak{R}_n(u,v) \frak{R}_{n-1}(u,z) \frak{R}_{n}(v,z)=
\frak{R}_{n-1}(v,z)\frak{R}_n(u,z) \frak{R}_{n-1}(u,v),~2\leq n \leq M,~ u,v,z \in \C\setminus 0,
\]
a more familiar form of Yang-Baxter equation.\\
\\
Working with particle systems it is more natural to use the 
stochastic generators $\{\sigma_n\}_{n \in \Z}$ satisfying (\ref{hecke_stoch}). The
relation between $\{s_n\}_{n \in \Z}$ and $\{\sigma_n\}_{n \in \Z}$ is given in (\ref{remap}).
The relation between the parameters $q$ and $Q$ entering the quadratic relations 
for the generators $s_i$'s and $q_i$'s correspondingly is given by (\ref{constQ})
Combining (\ref{ybsol}),(\ref{heinv}) with (\ref{remap}) we obtian a solution of Yang-Baxter equation
in terms of idempotent generators:
\bea\label{ybsolfin}
R_n(x)=q-q^{-1}x-(1-x)(q+q^{-1})\sigma_n,~n \in \Z, x \in \C.
\eea
It is possible to extend the above solution to the totally asymmetric case $Q=0$. 
Formally, this corresponds to redefining $R_n\rightarrow q^{-1}R_n$ and taking the limit $q\rightarrow 0$ in (\ref{ybsolfin}), which gives 
 \bea\label{ybedeg}
 R_n(x)=-x-(1-x)\sigma_n, n \in \Z, x \in \C,
 \eea
A direct check based on (\ref{hecke_markov}) with $Q=0$ shows that (\ref{ybedeg})
does indeed
 solve the Yang-Baxter equation.
\subsubsection{$Q=\frac{1}{4}$:  [BVM], [SAVM], [ACSRW] and [SEP]} 
For $Q=\frac{1}{4}$, $q=q^{-1}=1$ and 
\[
R_n(x)=(1-x)(1-2\sigma_n)=(1-x)s_n,
\]
where the second equality uses (\ref{remap}). This is a trivial case of Baxterisation:
if $\{s_n\}_{n \in \Z}$ satisfy the braid relation, then $R_n(x)=f(x)s_n,~n\in \Z, ~x \in \C$
satisfies the YBE for any function $f:\Z\rightarrow \R$.  
\subsubsection{[CSRWB]}
For coalescing-branching random walks, $Q=\theta(1-\theta)$, where $\theta\in [0,\frac{1}{2}]$. First, consider
$\theta>0$. Then $Q>0$ and the formula (\ref{ybsolfin}) applies.  Expressing the answer
in terms of $q$ we find:
\bea\label{expl1}
R(x)=\left(
\begin{array}{cccc}
q^{-1}+x(q-2q^{-1})&-(1-x)q^{-1}&-(1-x)q^{-1}&0\\
-(1-x)(q-q^{-1})&q-q^{-1}&-(1-x)q^{-1}&0\\
-(1-x)(q-q^{-1})&-(1-x)q^{-1}&q-q^{-1}&0\\
0&0&0&-q^{-1}+qx
\end{array}
\right),
\eea
where $x \in C$, $q\neq 0$. 

If $\theta=0$, then $Q=0$ is given by and we need to use (\ref{ybedeg}). 
 Thus we get an $R$-matrix associated with pure branching model:
 \bea\label{expl2}
R(x)=\left(
\begin{array}{cccc}
1-2x&-(1-x)&-(1-x)&0\\
(1-x)&-1&-(1-x)&0\\
(1-x)&-(1-x)&-1&0\\
0&0&0&-1
\end{array}
\right),
\eea
where $x \in \C$.
\subsubsection{[ASRWPI]}
For annihilating random walks with pairwise immigration, $Q=\theta^2$, 
$\theta \in [0, \frac{1}{2}]$. If $\theta>0$,  we can use (\ref{ybsolfin})
to find that
\bea\label{expl3}
R(x)=\left(
\begin{array}{cccc}
-(1-x)+(1+x)\frac{q-q^{-1}}{2}&0&0&(1-x)\left(1-\frac{q+q^{-1}}{2}\right)\\
0&(1+x)\frac{q-q^{-1}}{2}&-(1-x)\frac{q+q^{-1}}{2}&0\\
0&-(1-x)\frac{q+q^{-1}}{2}&(1+x)\frac{q-q^{-1}}{2}&0\\
-(1-x)\left(1+\frac{q+q^{-1}}{2}\right)&0&0&(1-x)+(1+x)\frac{q-q^{-1}}{2}
\end{array}
\right),
\eea
$x \in \C$. 
The $R$-matrix for $\theta=0$ (pure pairwise immigration) can be computed using
(\ref{ybedeg}). The answer is
\bea\label{expl4}
R(x)=\left(
\begin{array}{cccc}
-\frac{1+x}{2}&0&0&-\frac{1-x}{2}\\
0&-\frac{1+x}{2}&-\frac{1-x}{2}&0\\
0&-\frac{1-x}{2}&-\frac{1+x}{2}&0\\
-\frac{1-x}{2}&0&0&-\frac{1+x}{2}
\end{array}
\right),
\eea
$x \in \C$
\section{Linear Algebra Proofs} \label{s5}
We start be confirming the claim in (\ref{rmtemp}) that the only rank one $\sigma$ 
where $\sigma^2 = \sigma$, the deformed braid relation (\ref{braidreln}) holds for $\{\sigma_n\}$, and satisfying
the stochasticity conditions (that the off diagonal entries of 
$\sigma$ are non-negative and that row sums equal $1$) is the particular case of the reshuffle 
model stated. 

We start with a general rank $1$ matrix with the known eigenvector $\sigma = (1,1,1,1)^T(a_1,a_2,a_3,a_4)$.
We require $a_1+a_2+a_3+a_4=1$ and $a_i \geq 0$ for $i=1,..,4$ for stochasticity, 
and this already ensures that $\sigma^2=\sigma$. 
We set $a_4=1 -a_1-a_2-a_3$ and substitute this form into the $8 \times 8$ matrix 
\[
M = \left(\sigma_n \sigma_{n+1} \sigma_n - Q \sigma_n \right) - 
\left(\sigma_{n+1} \sigma_{n} \sigma_{n+1} - Q \sigma_{n+1} \right).
\]
We require $M=0$ for the deformed braid relation to hold. We find $M_{1,1} = -a_1(a_2-a_3)$.
Thus either (i) $a_1=0$ or (ii) $a_2=a_3$. In case (i) we find $M_{5,2} = -a_2^2$, implying that
$a_2=0$, and then $M_{1,8} = a_3(a_3-1)$ showing that $a_3 \in \{0,1\}$; a direct check shows that
the choice $a_3 =1$ does not yield $M=0$ whereas $a_3 =0$ yields a solution $M=0$ (corresponding to
$t=0$ in (\ref{rmtemp})). In the second case (ii) we may now assume $a_1 >0$ as well as $a_2=a_3$. 
Then $M_{7,1} = -a_1Q$ implying that $Q=0$. Then $M_{1,2} = a_1 - (a_1+a_3)^2$. This vanishes
precisely along the described family $a_1 =t^2, \, a_2=a_3=t(1-t)$ for $t \in [0,1]$, and with this choice
we find that $M=0$. 

This exhaustive search gives a flavour of how we establish the second classification Theorem \ref{T2}. The first
classification Theorem \ref{T1} follows the more geometric line of reasoning in Lemmas \ref{temp1} and \ref{temp2}.
\subsection{Proof of Lemma \ref{temp1}.}  \label{s5.1}
Denote $ v= \bp 1\\ 1 \ep$ and $ w= \bp 1\\-1 \ep$, which form an orthogonal basis of $V$. 
From these we  construct an orthonormal basis of $V\otimes V$
consisting of eigenvectors of the inversion $\rho$:
\bea \label{ebasis}
e_1=\frac{1}{2} v\otimes v, \quad e_2=\frac{1}{2}w\otimes w, \quad
e_3=\frac{1}{2\sqrt{2}} (v\otimes w+w\otimes v), \quad
e_4=\frac{1}{2\sqrt{2}} (v\otimes w-w\otimes v).
\eea
The vectors $e_1, e_2, e_3$ span a subspace of $V\times V$ consisting of eigenvectors
of $\rho$ with eigenvalue $1$; $e_4$ spans a one-dimensional eigenspace of $\rho$
corresponding to eigenvalue $-1$.

Recall that $\sigma \rho =\rho \sigma$.
If $u \in V$ is an eigenvector of $\rho$ with eigenvalue
$\lambda$, then $\sigma u$ is also an eigenvector of $\rho$ with the same eigenvalue,
provided $\sigma u \neq 0$. Therefore,
\[
\mbox{$\sigma e_4=\gamma e_4,$ for some $\gamma \in \R$ and 
$ \sigma e_i \in \mbox{Span}(e_1, e_2, e_3)$ for $i=1,2,3.$}
\]
In the basis $\{e_1, e_2, e_3, e_4\}$, the matrix of $\sigma$ becomes, written in block form, 
\begin{equation} \label{sigmaform}
\sigma = 
\begin{pmatrix}
  \sigma_3
  & \rvline & 0_3 \\
\hline
 0_3^T & \rvline &
\gamma
\end{pmatrix}
\end{equation}
where $\sigma_3$ is some $3\times 3$ matrix and $0_3$ is the $3 \times 1$ zero vector.
Since $\sigma^2=\sigma$ we must have $\sigma_3^2=\sigma_3$ and $\gamma^2=\gamma$, the latter
implying that $\gamma \in \{0,1\}$.

The stochastic constraint that row sums equal $1$ implies $\sigma e_1=e_1$. Therefore, in the basis $\{e_1, 
e_2, e_3\}$, 
\begin{equation} \label{sigma3form}
\sigma_3  = 
\begin{pmatrix}
  1
  & \rvline & c^T \\
\hline
0_2 & \rvline &
\sigma_2
\end{pmatrix}
\end{equation}
with $\sigma_2$ is a $2\times 2$ matrix and $c$ is a $2 \times 1$ vector which must satisfy, 
in order that $\sigma_3^2 = \sigma_3$, 
\begin{equation} \label{c-condition}
\sigma_2^2 =\sigma_2, \quad \mbox{and} \quad c^T \sigma_2= 0_2^T
\end{equation}
where $0_2$ denotes the $2\times 1$ zero vector. 

To further analyse the entries in $\sigma_2$, we make use of the 
Pauli basis $\{I,s_1,s_2,s_3,s_4\}$  
for the space of $2\times 2$
real matrices:
\[
I=\bp
1&0\\
0&1
\ep, \quad
s_1=\bp
0&1\\
1&0
\ep, \quad
s_2=\bp
0&1\\
-1&0
\ep, \quad
s_3=\bp
1&0\\
0&-1
\ep.
\]
Recall the anti-commutation relations obeyed by Pauli matrices:
\bea  \label{pcr}
\{s_i,s_j\}:=s_i s_j+s_j s_i=(-1)^{i+1} \, 2\, \delta_{ij} \, I, \quad \mbox{for $1\leq i,j\leq 3.$}
\eea
Let us expand $\sigma_2$ in the Pauli basis:
\bea \label{pauli}
\sigma_2=\alpha_0 I+\sum_{i=1}^3 \alpha_i s_i, \quad \mbox{where $\alpha_j \in \R $ for $0\leq j\leq 3$.}
\eea
In terms of the expansion coefficients, the condition $\sigma_2^2=\sigma_2$ takes the form
\bea
&& \alpha_0^2+\sum_{i=1}^3(-1)^{i+1} \alpha_i^2=\alpha_0 \label{rel1},\\
&& (2\alpha_0-1)\alpha_i=0 \quad \mbox{for $i=1,2,3.$} \label{rel2}
\eea
The simple derivation of this is based on (\ref{pcr}). 
Equation (\ref{rel2}) implies that $\alpha_0=\frac{1}{2}$ or $\alpha_1=\alpha_2=\alpha_3=0$.
For the former choice (\ref{rel1})  reduces to $\sum_{i=1}^3(-1)^{1+i}\alpha_i^2=\frac{1}{4}$, for the latter choice to
$\alpha_0\in \{0,1\}$. To summarise, there are three possible classes of solutions to $\sigma_2^2=\sigma_2$:
\begin{eqnarray*}
&& \mathbf{\{A\}} = \{\alpha_1=\alpha_2=\alpha_3=0,~\alpha_0=1\},\\ 
&& \mathbf{\{B\}} = \{\alpha_1=\alpha_2=\alpha_3=0,~\alpha_0=0\},\\
&& \mathbf{\{C\}} = \{\alpha_1^2-\alpha_2^2+\alpha_3^2=1/4,~\alpha_0=1/2\}.
\end{eqnarray*}
We will now analyse the existence of factorised eigenvectors for each of the cases at hand.

 $\mathbf{\{A\}} = \{\alpha_1=\alpha_2=\alpha_3=0,~\alpha_0=1\}. \;$ 
In this case (\ref{pauli}) gives $\sigma_2 = I$ and then (\ref{c-condition}) implies that $c = (0,0)^T$. 
From (\ref{sigma3form}) and (\ref{sigmaform}) we have that, in the basis $\{e_1,e_2,e_3,e_4\}$,  
\[
\sigma=\left(
\begin{array}{cccc}
1&0&0&0\\
0&1&0&0\\
0&0&1&0\\
0&0&0& \gamma
\end{array}
\right), \quad \mbox{for $\gamma \in \{0,1\}$.}
\]
Therefore $\sigma w\otimes w=w\otimes w$ and 
the existence of a suitable factorised eigenvector is proved.

 $\mathbf{\{B\}} = \{\alpha_1=\alpha_2=\alpha_3=0,~\alpha_0=0\}.\; $
In this case (\ref{pauli}) gives $\sigma_2 = 0$ and then (\ref{c-condition}) gives no constraints on $c$. 
From (\ref{sigma3form}) and (\ref{sigmaform}) we have that, in the basis $\{e_1,e_2,e_3,e_4\}$,   
\bea\label{sigmab}
\sigma=\left(
\begin{array}{cccc}
1&c_1&c_2&0\\
0&0&0&0\\
0&0&0&0\\
0&0&0&\gamma
\end{array}
\right), \quad \mbox{for $c_1, c_2, \in \R,  \gamma \in \{0,1\}.$}
\eea
We now check for the existence of a second eigenvector in factorised form, which without loss of generality
must be of the form  
\begin{equation} \label{factorform}
(w+\frac{\beta}{\sqrt{2}}v) \otimes (w+\frac{\beta}{\sqrt{2}}v) = \beta^2 e_1 + 2 e_2 + 2 \beta e_3
 \quad \mbox{for some $\beta \in \R$.}
\end{equation}
Acting $\sigma$ from (\ref{sigmab}) on the vector $(\beta^2,2,2\beta,0)^T$  shows that the
second factorised eigenvector does not exist. 

It remains to examine which operators of the form (\ref{sigmab}) in the basis 
$\{e_1,e_2,e_3,e_4\}$ satisfy the stochasticity conditions. These conditions are for the matrix
in the standard basis, and to find expression for $\sigma$ in the stochastic basis we need to conjugate
the above matrix (\ref{sigmab}) by the orthogonal transformation
\bea \label{CoB}
O=\bp
\frac{1}{2}&\frac{1}{2}&\frac{1}{\sqrt{2}}&0\\
\frac{1}{2}&-\frac{1}{2}&0&-\frac{1}{\sqrt{2}}\\
\frac{1}{2}&-\frac{1}{2}&0&\frac{1}{\sqrt{2}}\\
\frac{1}{2}&\frac{1}{2}&-\frac{1}{\sqrt{2}}&0
\ep.
\eea
The form of $\sigma$ in the standard basis becomes
\bea 
\label{sigmaas}
\sigma=\bp
\frac{1}{4}+\frac{c_1}{4}+\frac{c_2}{2\sqrt{2}}
&\frac{1}{4}-\frac{c_1}{4}
&\frac{1}{4}-\frac{c_1}{4}
&\frac{1}{4}+\frac{c_1}{4}-\frac{c_2}{2\sqrt{2}}\\
\frac{1}{4}+\frac{c_1}{4}+\frac{c_2}{2\sqrt{2}}
&\frac{1}{4}-\frac{c_1}{4}+\frac{\gamma}{2}
&\frac{1}{4}-\frac{c_1}{4}-\frac{\gamma}{2}
&\frac{1}{4}+\frac{c_1}{4}-\frac{c_2}{2\sqrt{2}}\\
\frac{1}{4}+\frac{c_1}{4}+\frac{c_2}{2\sqrt{2}}
&\frac{1}{4}-\frac{c_1}{4}-\frac{\gamma}{2}
&\frac{1}{4}-\frac{c_1}{4}+\frac{\gamma}{2}
&\frac{1}{4}+\frac{c_1}{4}-\frac{c_2}{2\sqrt{2}}\\
\frac{1}{4}+\frac{c_1}{4}+\frac{c_2}{2\sqrt{2}}
&\frac{1}{4}-\frac{c_1}{4}
&\frac{1}{4}-\frac{c_1}{4}
&\frac{1}{4}+\frac{c_1}{4}-\frac{c_2}{2\sqrt{2}}
\ep.
\eea
The positivity of the non-diagonal elements leads to 
the following inequalities for the parameters of $\sigma$:
\[
1-c_1\geq 0, \quad 1+c_1-\sqrt{2}c_2\geq 0, \quad 1+c_1+\sqrt{2}c_2\geq 0, \quad 1-c_1-2\gamma \geq 0.
\]
We consider the solutions for the two allowed values $\gamma \in \{0,1\}$.

When $\gamma =1$, the sum of the second and the third of the above inequalities gives
$1+c_1\geq 0$,  whereas the fourth inequality is $-1-c_1 \geq 0$. Therefore $c_1=-1$, and 
the second and the third inequalities together give $c_2=0$. These values lead to 
$\sigma$ being the symmetric anti-voter model [SAVM] in (\ref{savm}).

When $\gamma =$ all rows in (\ref{sigmaas}) are identical. We re-parametrise
as $\theta_1 = \frac{1}{4}+\frac{c_1}{4}+\frac{c_2}{2\sqrt{2}}$, 
$ \theta_2 = \frac{1}{4}-\frac{c_1}{4}$ and $\theta_3 = \frac{1}{4}+\frac{c_1}{4}-\frac{c_2}{2\sqrt{2}}
= 1 - \theta_1 - 2 \theta_2$, whereupon the inequalities become $\theta_i \geq 0$ for $i=1,2,3$, 
and we have arrived at the random reshuffle model [RM] in (\ref{rm}). 

 $\mathbf{\{C\}} = \{\alpha_1^2-\alpha_2^2+\alpha_3^2=\frac{1}{4},~\alpha_0=\frac{1}{2}\}. \;$  Note that set of solutions to
\begin{equation} \label{olegsC}
\alpha_1^2-\alpha_2^2+\alpha_3^2=\frac{1}{4}
\end{equation} 
is not empty and can be parameterised as follows:
\[
\alpha_1=\frac{1}{2}\cosh(\phi)\cos(\psi),~\alpha_3=\frac{1}{2}\cosh(\phi)\sin(\psi),~
\alpha_2=\frac{1}{2}\sinh(\phi),~\phi \in \R, \psi \in [0,2\pi).
\]
From (\ref{sigma3form}) and (\ref{sigmaform}) we have that, in the basis $\{e_1,e_2,e_3,e_4\}$,   
\bea\label{sigmac}
\sigma=\bp
1&c_1&c_2&0\\
0&\frac{1}{2}+\alpha_3&\alpha_1+\alpha_2&0\\
0&\alpha_1-\alpha_2&\frac{1}{2}-\alpha_3&0\\
0&0&0&\gamma
\ep,
\eea
where the parameters $\alpha$ satisfy equation (\ref{olegsC}), $\gamma \in \{0,1\}$, and
$(c_1, c_2)\in \R^2$ satisfies
\bea \label{nullc}
(c_1, c_2) \left(\begin{array}{cc} 
\frac{1}{2}+\alpha_3&\alpha_1+\alpha_2\\
\alpha_1-\alpha_2&\frac{1}{2}-\alpha_3
\end{array} \right)=(0,0).
\eea 
The determinant of $\sigma_2$, the matrix for (\ref{nullc}), is zero due to (\ref{olegsC}).
On the other hand, the matrix is not identically zero for any 
$\alpha_1, \alpha_2, \alpha_3 \in \R$. Therefore, the set of solutions to (\ref{nullc})
is a one-dimensional subspace of $\R^2$.

We again look for a factorised eigenvector in the form (\ref{factorform}). 
Acting $\sigma$ from (\ref{sigmac}) on the vector $(\beta^2,2,2\beta,0)^T$
we find that
$ \left(w+\frac{\beta}{\sqrt{2}} v\right) \otimes \left(w+\frac{\beta}{\sqrt{2}} v\right)$
being an eigenvector is equivalent to the following three conditions:
\bea
0 &=&  c_1+\beta c_2,\label{c1c}\\
0 &=& -\left(\frac{1}{2}-\alpha_3 \right)+(\alpha_1+\alpha_2)\beta,\label{c2c}\\
0 & = & (\alpha_1-\alpha_2)-\left(\frac{1}{2}+\alpha_3 \right)\beta.\label{c3c}
\eea
The existence of $\beta$ solving the overdetermined system (\ref{c1c})-(\ref{c3c})
depends on the choice of $\alpha$'s. We break the logic into three subcases C(i), C(ii), C(iii). 

 $\mathbf{\{C1\}} =  \mathbf{\{C\}} \cap
\{\alpha_3 \neq -\frac{1}{2}\}. \;$ In this case, we start by solving (\ref{nullc}) to find that
\[
(c_1,c_2)=K \, (-(\alpha_1-\alpha_2), \frac{1}{2}+\alpha_3) \quad \mbox{for some $K \in \R$.}
\]
Equation (\ref{c1c}) becomes a consequence of (\ref{c3c}). The latter is solved 
by
\beast
\beta=\frac{\alpha_1-\alpha_2}{\left(\frac{1}{2}+\alpha_3 \right)}
\eeast 
Substituting the solution into (\ref{c2c}) we find that it is satisfied 
due to condition $(\ref{olegsC})$. The existence of a factorised eigenvector is established.

When $\alpha_3 = \frac{1}{2}$ then (\ref{olegsC}) implies that $\alpha_1 = \pm \alpha_2$, leading two the last two subcases.

$\mathbf{\{C2\}} =  \mathbf{\{C\}} \cap \{\alpha_3 = -\frac{1}{2}, \; \alpha_1=\alpha_2 = \mu \neq 0\}. \;$  Solving (\ref{nullc}) 
we find
\[
(c_1,c_2)=K \, (1, -2\mu), \quad \mbox{for some $K \in \R$.}
\]
Equations (\ref{c1c})-(\ref{c3c}) become the three equations
\[
K(1-2\mu\beta)=0, \qquad
-1+2\mu \beta=0,\qquad 0=0.
\]
These are solved, for any value of $K$, by taking $\beta=\frac{1}{2\mu}$,
yielding a factorised eigenvector.

 $\mathbf{\{C3\}} =  \mathbf{\{C\}} \cap \{\alpha_3 = -\frac{1}{2}, \; \alpha_1 = -\alpha_2 = \mu\}. \; $ 
In this case there is no second factorised eigenvector, for example since (\ref{c2c}) can have no solution.
It remains to identify the set of $\sigma$ that correspond to the present choice of parameters.
Solving (\ref{nullc}) we find $(c_1,c_2)=K (1,0)$ for $K \in \R$.
Then (\ref{sigma3form}) and (\ref{sigmaform}) yield
 \bea\label{sigmaciii}
\sigma=\bp
1&K&0&0\\
0&0&0&0\\
0&2\mu&1&0\\
0&0&0&\gamma
\ep \quad \mbox{for some $\mu, K \in \R, \gamma \in \{0,1\}$.}
\eea
We conjugate by $(\ref{CoB})$ to find the following 
expression for $\sigma$ in the standard basis:
\[
\sigma=\bp
\frac{3}{4}+\frac{K}{4}+\frac{\mu}{\sqrt{2}}
&\frac{1}{4}-\frac{K}{4}-\frac{\mu}{\sqrt{2}}
&\frac{1}{4}-\frac{K}{4}-\frac{\mu}{\sqrt{2}}
&-\frac{1}{4}+\frac{K}{4}+\frac{\mu}{\sqrt{2}}\\
\frac{1}{4}+\frac{K}{4}
&\frac{1}{4}-\frac{K}{4}+\frac{\gamma}{2}
&\frac{1}{4}-\frac{K}{4}-\frac{\gamma}{2}
&\frac{1}{4}+\frac{K}{4}\\
\frac{1}{4}+\frac{K}{4}
&\frac{1}{4}-\frac{K}{4}-\frac{\gamma}{2}
&\frac{1}{4}-\frac{K}{4}+\frac{\gamma}{2}
&\frac{1}{4}+\frac{K}{4}\\
-\frac{1}{4}+\frac{K}{4}-\frac{\mu}{\sqrt{2}}
&\frac{1}{4}-\frac{K}{4}+\frac{\mu}{\sqrt{2}}
&\frac{1}{4}-\frac{K}{4}+\frac{\mu}{\sqrt{2}}
&\frac{3}{4}+\frac{K}{4}-\frac{\mu}{\sqrt{2}}
\ep  \quad \mbox{for $\mu, K \in \R, \gamma \in \{0,1\}$.}
\]
The stochasticity conditions, that the off diagonal entries are non-negative, are equivalent to 
 that parameters $K,\mu$ and $\gamma$ satisfying six inequalities:
\beast
& \frac14+ \frac{K}{4}\geq 0,\qquad
\frac14 - \frac{K}{4}- \frac{\gamma}{2} \geq 0,\qquad
-\frac{1}{4}+\frac{K}{4}-\frac{\mu}{\sqrt{2}}\geq 0,&\\
& \frac{1}{4}-\frac{K}{4}-\frac{\mu}{\sqrt{2}}\geq 0, \qquad
\frac{1}{4}-\frac{K}{4}+\frac{\mu}{\sqrt{2}}\geq 0,\qquad
-\frac{1}{4}+\frac{K}{4}+\frac{\mu}{\sqrt{2}}\geq 0.&
\eeast
Adding the third and the sixth inequalities we find $-1+K\geq 0$; adding the fourth
and fifth inequalities we find $1-K\geq 0$; therefore $K=1.$
The second inequality reduces to $-\gamma \geq 0$, but since  
$\gamma \in \{0,1\}$ we can conclude that $ \gamma=0.$
The fourth and fifth inequalities reduce to the pair $\mu \geq 0$, $\mu\leq 0$ so that 
$ \mu=0.$ For the choices $K=1, \gamma=0, \mu =0$ all inequalities are satisfied
and the resulting unique $\sigma$ for this subcase is the symmetric version
of the biased voter model
[BVM] from (\ref{bvm}). \qed 
\subsection{Proof of Lemma \ref{temp2}}  \label{s5.2}
%
Our starting point is the most general $4\times 4$ stochastic matrix (in the standard basis):
\bea \label{sm}
\sigma= \bp
1-a_2-a_3-a_4&a_2&a_3&a_4\\
b_1&1-b_1-b_3-b_4&b_3&b_4\\
c_1&c_2&1-c_1-c_2-c_4&c_4\\
d_1&d_2&d_3&1-d_1-d_2-d_3\\
\ep
\eea
where all off-diagonal entries are real non-negative numbers.

As the vectors $v$ and $w$ are linearly independent, we have
$V\otimes V=Span(v\otimes v, w\otimes w, v\otimes w, w \otimes v)$.
Therefore, there exist $\alpha, \beta, \gamma, \delta \in \R$:
\bea \label{sigmavw}
\sigma v\otimes w=\alpha v\otimes v+\beta w\otimes w+\gamma v\otimes w +\delta w\otimes v.
\eea
Consequently, using $\sigma \rho = \rho \sigma$, 
\bea \label{sigmawv}
\sigma w\otimes v=\rho \sigma(v\otimes w)=\alpha v\otimes v+\beta w\otimes w+\gamma w\otimes v +\delta v\otimes w.
\eea
The condition $\sigma^2=\sigma$ expressed in terms of the coefficients $\alpha, \beta, \gamma, \delta$
reads
\bea
\alpha(\gamma+\delta)=0,\label{ord21}\\
\beta(\gamma+\delta)=0,\label{ord22}\\
\gamma^2+\delta^2=\gamma,\label{ord23}\\
(2\gamma-1)\delta=0.\label{ord24}
\eea
It follows from (\ref{ord24}) that $\gamma=\frac{1}{2}$ or $\delta=0$. If $\delta=0$, equation
(\ref{ord23}) gives $\gamma=0$ or $\gamma=1$. In former case, equations (\ref{ord21}, \ref{ord22})
are satisfied for any $\alpha, \beta \in \R$; in the latter case, equations (\ref{ord21}, \ref{ord22}) imply
that $\alpha=\beta=0$.
If $\gamma=1/2$, then equation (\ref{ord23}) gives $\delta=\frac{1}{2}$ or $\delta=-\frac{1}{2}$.
If $\gamma=\delta=\frac{1}{2}$, equations (\ref{ord21}, \ref{ord22}) yield $\alpha=\beta=0$.
If $\gamma=-\delta=\frac{1}{2}$, equations (\ref{ord21}, \ref{ord22}) are satisfied for
any $\alpha, \beta \in \R$. Therefore, we have four cases to consider which we label as A, B, C, D below.

 $\mathbf{\{A\}} =  \{\delta=0, \; \gamma=1, \; \alpha= \beta =0\}. \;$
In this case (\ref{sigmavw}) and (\ref{sigmawv}) show that $\sigma$ acts as the identity
and does not enter our classification which is for non-trivial $\sigma \neq I$. 

 $\mathbf{\{B\}} =  \{\delta=0, ~\gamma=0, ~\alpha, \beta \in \R\}. \; $
The action of $\sigma$ on $V\otimes V$ is fully determined by
$\sigma v\otimes v=v\otimes v$, $\sigma w\otimes w=w\otimes w$,
and, from
 (\ref{sigmavw}) and (\ref{sigmawv}),
  $\sigma v\otimes w=\alpha v\otimes v+\beta w\otimes w$, $\sigma w\otimes v= \alpha v\otimes v+\beta w\otimes w$.
A more symmetric form for the action is
\bea
 \sigma v\otimes v &=&
 v\otimes v,\label{aa1}\\
 \sigma w\otimes w &=& w\otimes w,\label{aa2}\\
 \sigma (v\otimes w-w\otimes v) &=& 0, \label{aa4}\\
  \sigma \frac{(v\otimes w+w\otimes v)}{2} &=& \alpha v\otimes v+\beta w\otimes w. \label{aa3}
\eea
Write $w=(w_1~ w_2)^T$, where $w_1\neq w_2$ so that $w$ is independent of $v$. 
There are two subcases, $w_1=0$ and $w_1\neq 0$. Be redefining the coefficients $\alpha$ and $\beta$ we can  set $w_2=1$ in the former
case and $w_1=1$ in the latter.\\

$\mathbf{\{B1\}} =  \mathbf{\{B\}} \cap \{w_1=0,~w_2=1\}. \; $
In this case,
\beast
v\otimes v=\bp1\\1\\1\\1 \ep, \quad
w\otimes w=\bp 0\\0\\0\\1 \ep, \quad
\frac{(v\otimes w+w\otimes v)}{2}=\bp 0\\\fr\\\fr\\1 \ep, \quad
v\otimes w-w\otimes v=\bp 0\\1\\-1\\0 \ep.
\eeast
For our form (\ref{sm}) the first part (\ref{aa1}) of the action always holds true as rows sums equal $1$. 
Substituting (\ref{sm}) in (\ref{aa2}) we find $a_4=0, b_4=0,c_4=0$, $d_1+d_2+d_3=0$
Due to positivity the last condition implies that $d_1=d_2=d_3=0$. Then, equation (\ref{aa4})
leads to $a_2-a_3=0$, $1-b_1-2b_3=0$, $1-c_1-2c_2=0$. After these identifications the stochastic matrix takes the form
\bea\label{smaa1}
\sigma= \bp
1-2a&a&a&0\\
1-2b&b&b&0\\
1-2c&c&c&0\\
0&0&0&1\\
\ep, \quad \mbox{for $a \geq 0, b,c \in [0,\frac12]$.}
\eea
Finally, substituting (\ref{smaa1}) into (\ref{aa3}) we find $a=b=c=\alpha$, $\alpha+\beta=1$, 
so that the final form of the stochastic matrix is that of the symmetric coalescing random walks with branching
[CSRWB] in (\ref{csrwb}), taking $\theta = \alpha \in [0,\frac12]$. We record the 
the action of $\sigma$ on the tensor products of $v$ and $w$ in Lemma \ref{temp2}.  

 $\mathbf{\{B2\}} =  \mathbf{\{B\}} \cap \{w_1=1,~w_2=\omega \neq 1\}. \; $
In this case,
\beast
v\otimes v=\bp1\\1\\1\\1 \ep, \quad
w\otimes w=\bp 1\\\omega\\\omega\\\omega^2 \ep, \quad
\frac{(v\otimes w+w\otimes v)}{2}=\bp 1\\\frac{(1+\omega)}{2}\\\frac{(1+\omega)}{2}\\\omega \ep, \quad
v\otimes w-w\otimes v= \bp0\\ \omega-1\\1-\omega\\0 \ep.
\eeast
By substituting (\ref{sm}) into conditions (\ref{aa2}, \ref{aa4}), we may reduce the numbers of parameters so that 
\bea\label{smaii}
\sigma= \bp
1+\omega a&-\frac{1+\omega}{2}a&-\frac{1+\omega}{2}a&a\\
\omega b&\frac{1-(1+\omega)b}{2}&\frac{1-(1+\omega)b}{2}&b\\
\omega c&\frac{1-(1+\omega)c}{2}&\frac{1-(1+\omega)c}{2}&c\\
\omega d&-\frac{1+\omega}{2}d&-\frac{1+\omega}{2}d&1+d\\
\ep
\eea
where the parameters $\omega, a, b, c, d$ must then be constrained by the positivity conditions.
The last condition (\ref{aa3}) leads to the following system of four equations:
\bea 
1-\fr (1-\omega)^2 a & = & \alpha+\beta,\label{Ka} \\
\frac{1+\omega}{2}-\frac{1}{2}(1-\omega)^2b & = & \alpha+\beta \omega,\label{Kb} \\
\frac{1+\omega}{2}-\frac{1}{2}(1-\omega)^2c & = & \alpha+\beta \omega,\label{Kc} \\
\omega-\frac{1}{2}(1-\omega)^2d & = & \alpha+\beta \omega^2.\label{Kd}
\eea
We divide the logic into two further subcases, depending on the value of $\omega \neq 1$: namely
the possibilities $\omega \in (-\infty,-1)$, $\omega = -1$, $\omega \in (-1,0)$, $\omega =0 $, 
$\omega \in (0,1)$, and $\omega \in (1,\infty)$. 

$\mathbf{\{B21\}} =  \mathbf{\{B\}} \cap \{w_1=1,~w_2=\omega = - 1\}. \; $
The rate matrix (\ref{smaii}) takes the form:
\bea\label{smaii1}
\sigma= \bp
1-a&0&0&a\\
-b&\frac{1}{2}&\frac{1}{2}&b\\
- c&\frac{1}{2}&\frac{1}{2}&c\\
- d&0&0&1+d\\
\ep.
\eea
The positivity conditions imply that  $b=c=0$ and
the system (\ref{Ka})-(\ref{Kd}) reduces to
\bea
1-2 a=\alpha+\beta,\qquad 0=\alpha-\beta,\qquad
-1-2d=\alpha+\beta.
\eea
Setting $\alpha = \theta$ and solving the last set equations with respect to $a,d,\beta$ 
and substituting the answer into (\ref{smaii1}), we find
\bea\label{smaii1final}
\sigma= \left( \begin{array}{cccc}
\fr+ \theta&0&0&\fr-\theta\\
0&\frac{1}{2}&\frac{1}{2}&0\\
0&\frac{1}{2}&\frac{1}{2}&0\\
\fr+\theta&0&0&\fr-\theta\\
\end{array}\right) \quad \mbox{for $\fr \leq \theta \leq \fr$.}
\eea
This matrix acts on tensor products of $v$ and $w$ as follows:
\[
\sigma v\otimes w=\sigma w\otimes v = -\theta v\otimes v-\theta w\otimes w. 
\]
In Lemma \ref{temp2} we have switched the sign, replacing $w$ by $-w = (-1,1)^T$ in 
order to mesh with previous cases. 
Under particle-hole conjugation the model value of $\theta$ becomes $-\theta$, 
so we record just the parameter range $\theta \in [0,\fr]$ in Lemma \ref{temp2}. 
This case has become annihilating symmetric random walks with pairwise immigration [ASRWPI] in 
(\ref{asrwpi}).

$\mathbf{\{B22\}} =  \mathbf{\{B\}} \cap \{w_1=1,~w_2=\omega = 0\}. \; $
This case 
can be obtained from the stochastic matrix in case $\mathbf{\{B1\}}$ after a 
particle-hole conjugation, so leads to no new cases in the classification. 

$\mathbf{\{B23\}} =  \mathbf{\{B\}} \cap \{w_1=1,~w_2=\omega \in (0,1) \cup (1,\infty)\}. \; $
The positivity condition applied to the first and the
fourth rows of (\ref{smaii}) implies that $a=d=0$. 
Solving (\ref{Ka}) and (\ref{Kd}) with respect to $\alpha, \beta$ we find
$\alpha = \frac{\omega}{1+\omega}$ and $\beta = \frac{1}{1+\omega}$. Then (\ref{Kb}) 
and (\ref{Kc}) yield the values $b=c=\frac{1}{1+\omega}$, and the positivity conditions
in the second and third rows of (\ref{smaii}) hold true. Setting $\theta = \frac{\omega}{1+\omega} \in (0,1)$
we arrive at the voter model [BVM] in (\ref{bvm}), where the symmetric case $\theta = \fr$ is removed 
due to the condition that $\omega \neq 1$. The second factorised eigenvector is then 
$w \otimes w$ where $w = (1, \theta/(1-\theta))$. After scaling, one can replace this vector by
$w = (1-\theta, \theta)$ and the action of $\sigma$ 
on the tensor products of $v$ and $w$ is recorded in Lemma \ref{temp2}.  
The cases $\theta \in \{0,1\}$ coincide with special values of the model [CSRWB] and its 
particle-hole conjugate. 

$\mathbf{\{B24\}} =  \mathbf{\{B\}} \cap \{w_1=1,~w_2=\omega  \in (-\infty,-1)\}. \; $
The positivity condition applied to the second, third ands fourth
rows of (\ref{smaii}) implies that $b=c=d=0$. 
 Solving (\ref{Kb}), (\ref{Kc}) and (\ref{Kd}) with respect to $\alpha, \beta$ we find
$\alpha = \omega/2$ and $\beta = 1/(2 \omega)$. Then solving (\ref{Ka}) we find $a=-1/\omega$.
The matrix (\ref{smaii}) is reduced to 
\bea\label{smaii2final}
\sigma= \bp
0&\frac{1+\omega^{-1}}{2}&\frac{1+\omega^{-1}}{2}&-\omega^{-1}\\
0&\frac{1}{2}&\frac{1}{2}&0\\
0&\frac{1}{2}&\frac{1}{2}&0\\
0&0&0&1
\ep \quad \mbox{for $\omega \in (-\infty,-1)$.} 
\eea
Setting $\theta = -1/\omega$ we reach the annihilating-coalescing symmetric random walks [ACSRW]
in (\ref{acsrw}) for $\theta \in (0,1)$.  We rescale the second factorised eigenvector, to simplify
the action, sending $(1,-\theta^{-1}) \rightarrow (-\theta,1)$. This makes $w=(-\theta,1)$, still ensures
that $\sigma w \otimes w = w \otimes w$,  and 
the action of $\sigma$ on the tensor products of $v$ and $w$ becomes
simpler: $\sigma v\otimes w = \sigma w\otimes v=\fr v\otimes v+\fr w\otimes w$,
as recorded in  Lemma \ref{temp2}.  

$\mathbf{\{B25\}} =  \mathbf{\{B\}} \cap \{w_1=1,~w_2=\omega  \in  (-1,0)\}. \; $
The positivity condition applied to the first, second and third
rows of (\ref{smaii}) implies that $a=b=c=0$. Solving 
(\ref{Ka}), (\ref{Kb}) and (\ref{Kc}) with respect to $\alpha, \beta$ we find
$\alpha = \beta = \fr$. Then solving (\ref{Kd}) we find $d=-1$. The matrix 
(\ref{smaii}) is reduced to 
\bea\label{smaii2copy}
\sigma= \bp
1&0&0&0\\
0&\frac{1}{2}&\frac{1}{2}&0\\
0&\frac{1}{2}&\frac{1}{2}&0\\
-\omega&\frac{1+\omega}{2}&\frac{1+\omega}{2}&0\\
\ep \quad \mbox{for $\omega \in (-1,0)$.}
\eea
Stochastic matrices (\ref{smaii2copy}) and (\ref{smaii2final}) are equivalent 
after particle-hole conjugation and exchanging $\omega \leftrightarrow \omega^{-1}$
so this adds no new case to the classification.

$\mathbf{\{C\}} =   \{\delta=\frac{1}{2},\, \gamma=\frac{1}{2}, \, \alpha= \beta =0\}. \; $
The action of $\sigma$ on $V\otimes V$ is fully determined by
\bea
\sigma v\otimes v &=& v\otimes v,\label{acb1}\\
\sigma w\otimes w&=&w\otimes w,\label{acb2}\\
\sigma v\otimes w&=&\frac{1}{2}( v\otimes w+ w\otimes v),\label{acb3}\\
\sigma w\otimes v&=& \frac{1}{2} (v\otimes w+ w\otimes v).\label{acb4}
\eea
The above equations are invariant with respect to $w\rightarrow aw$, where $a \in \R$ is a constant.
Therefore we can reduce to two subcases for the independent vector $w$, either
$w =(1,\omega)^T$ for $\omega \neq 1$ or $ w = (0,1)^T$.

$\mathbf{\{C1\}} =  \mathbf{\{C\}} \cap \{w_1=1,~w_2 = \omega \neq 1\}. \;$
In this case
\beast
v\otimes v=\bp 1\\1\\1\\1 \ep, \quad
w\otimes w=\bp 1\\\omega\\\omega\\\omega^2 \ep, \quad
v\otimes w=\bp 1\\\omega\\1\\\omega \ep, \quad
w\otimes v=\bp 1\\1\\\omega\\\omega \ep.
\eeast
The vector $v\otimes v$ is automatically an eigenvector of the stochastic matrix (\ref{sm}). The three conditions
(\ref{acb2}, \ref{acb3}, \ref{acb4}) lead to the following restrictions on the coefficients of (\ref{sm}):
\[
\begin{array}{ll}
\left\{ \begin{array}{l}
(\omega-1)(a_2+a_3)+(\omega^2-1)a_4=0,\\
(\omega-1)(a_2+a_4)=0,\\
(\omega-1)(a_3+a_4)=0,
\end{array} \right. 
& 
\left\{ \begin{array}{l}
b_1- \omega b_4=0,\\
b_3+\omega b_4=\frac{1}{2},\\
b_3+b_4=\frac{1}{2},
\end{array} \right. \\
& \\
\left\{ \begin{array}{l}
c_1-\omega c_4=0,\\
c_4+c_2=\frac{1}{2},\\
\omega c_4+c_2=\frac{1}{2}, 
\end{array} \right.
&
\left\{ \begin{array}{l}
(1-\omega^2)d_1+(\omega-\omega^2)(d_2+d_3)=0,\\
(\omega-1)(d_1+d_3)=0,\\
(\omega-1)(d_1+d_2)=0.
\end{array} \right.
\end{array}
\]
As $\omega \neq 1$ and all $a_i$ and $d_i$ are non-negative, the equations imply that
these coefficients are all zero.  The equations for $b_1,b_3,b_4$ and for $c_1,c_2,c_4$
have unique solutions $b_1=b_4=0,~b_3=\frac{1}{2}$ and $c_1=c_4=0,~c_2=\frac{1}{2}$.
Therefore, the rate matrix $\sigma$ is completely determined, and yields the symmetric
exclusion model [SEP] from (\ref{sep}). 

$\mathbf{\{C2\}} =  \mathbf{\{C\}} \cap \{w_1=0,~w_2 =1\}. \; $
In this case,
\beast
v\otimes v=\bp 1\\1\\1\\1 \ep, \quad
w\otimes w=\bp 0\\ 0\\ 0\\ 1 \ep, \quad
v\otimes w=\bp 0\\1\\0\\1 \ep, \quad
w\otimes v=\bp 0\\0\\1\\1 \ep.
\eeast
We proceed similarly to the previous subcase:
condition (\ref{acb1}) is automatic; condition (\ref{acb2}) gives $a_4=b_4=c_4=d_1+d_2+d_3=0$. As $d$'s
are non-negative, we conclude that $d_1=d_2=d_3=0$; condition (\ref{acb3}) gives $a_2=0$, $b_1+b_3=\frac{1}{2}$,
$c_2=\frac{1}{2}$; condition (\ref{acb4}) gives $a_3=0$, $b_3=\frac{1}{2}$, $c_1+c_2=\frac{1}{2}$. Solving
for $b_1, b_3$ and $c_1, c_2$ we find $b_1=c_1=0$, $b_3=c_2=\frac{1}{2}$. Substituting the answers into (\ref{sm})
we again arrive at the symmetric exclusion process [SEP] in (\ref{sep}). Therefore, the current subcase simply gives an alternative
form for the factorised eigenvector $w\otimes w$ as listed in Lemma \ref{temp2}.  

$\mathbf{\{D\}} =   \{\delta=-\frac{1}{2}, \, \gamma=\frac{1}{2}, \, \alpha, \beta \in \R\}. \; $
The action of $\sigma$ on $V\otimes V$ is determined by
\bea
\sigma v\otimes v&=&v\otimes v,\label{ad0}\\
\sigma w\otimes w&=&w\otimes w,\label{ad1}\\
\sigma (v\otimes w-w\otimes v)&=&v\otimes w- w\otimes v,\label{ad2}\\
\sigma \frac{(v\otimes w+ w\otimes v)}{2}&=& \alpha v\otimes v+\beta w\otimes w.\label{ad3}
\eea
We have two subcases to consider: if $w_1=0,$ then $w_2 \neq 0$ and we can rescale 
$w$ so that $w_2=1$. If $w_1 \neq 0$, then we can rescale $w$ in such a way that
$w_1=1$. (The action of $\sigma$ is not invariant with respect
to rescalings of $w$, therefore the final values of the coefficients
$\alpha$ and $\beta$ will depend on this choice of scale.)

$\mathbf{\{D1\}} =  \mathbf{\{D\}} \cap \{w_1=0,~w_2=1\}. \; $
In this case,
\beast
v\otimes v=\bp 1\\1\\1\\1 \ep, \quad
w\otimes w=\bp  0\\ 0\\ 0\\ 1 \ep, \quad
v\otimes w - w \otimes v =\bp  0\\1\\-1\\0 \ep, \quad
\frac{(v\otimes w+ w\otimes v)}{2} = \bp  0\\\fr\\\fr \\1 \ep. 
\eeast
Condition (\ref{ad0}) is automatic.
Condition (\ref{ad1}), that $\sigma w\otimes w=w\otimes w$, applied to (\ref{sm}) 
leads to the equalities $a_4=b_4=c_4=0$ and $d_1+d_2+d_3=0$. The
positivity of $d_i$ now leads us to the the following simplified form for $\sigma$:
\bea\label{smd1a}
\sigma= \bp
1-a_2-a_3&a_2&a_3&0\\
b_1&1-b_1-b_3&b_3&0\\
c_1&c_2&1-c_1-c_2&0\\
0&0&0&1 \ep.
\eea
Similarly, condition (\ref{ad1}), 
complemented by the positivity of $\sigma$, leads to
\bea\label{smd1b}
\sigma= \bp
1-2a&a&a&0\\
0&1&0&0\\
0&0&1&0\\
0&0&0&1\\
\ep.
\eea
Finally, the condition (\ref{ad3}) leads to $\alpha = \beta=a=1/2$.
This is the $\theta = \fr$ case of the stationary coalescence annihilation model
[SCAM] in (\ref{scam}), 
which will reappear for all $\theta$ in the subcase D(ii.b) below.

$\mathbf{\{D2\}} =  \mathbf{\{D\}} \cap \{w_1=1, w_2=\omega\neq 1\}. \; $
In this subcase,
\beast
v\otimes v=\bp 1\\1\\1\\1 \ep, \;
w\otimes w=\bp1\\ \omega\\ \omega\\ \omega^2 \ep, \;
v\otimes w-w\otimes v=\bp0\\ \omega -1 \\1 - \omega\\ 0 \ep, \;
\frac{w\otimes v+v\otimes w}{2}=\bp1\\ \frac{(1+\omega)}{2}\\ \frac{(1+\omega}{2})\\\omega \ep.
\eeast
We examine conditions (\ref{ad0}) - (\ref{ad3}) for the 
matrix (\ref{sm}). Condition (\ref{ad0}) is automatic.
Condition (\ref{ad2}) yields $a_2=a_3$,  
$b_1+2b_3+b_4=0$, 
$c_1+2c_2+c_4=0$, and 
$d_2=d_3$. 
Therefore, by positivity, $b_1=b_3=b_4=0$, $c_1=c_2=c_4=0$ and $\sigma$ takes the
following form:
\bea\label{smd2a}
\sigma= \bp
1-2a_2-a_4&a_2&a_2&a_4\\
0&1&0&0\\
0&0&1&0\\
d_1&d_2&d_2&1-d_1-2d_2\\
\ep.
\eea
The condition (\ref{ad1}) with this $\sigma$ leads to the following
restrictions on the parameters:
\bea\label{sysd}
\begin{array}{l}
(1+\omega) a_4+2a_2=0
\\
(1+\omega) d_1+2\omega d_2=0,
\end{array}
\eea 
where we used that $\omega \neq 1$.

We now examine solutions to (\ref{sysd}), together with the remaining condition
(\ref{ad3}), in various subcases according to the value of $\omega$. 

$\mathbf{\{D21\}} =  \mathbf{\{D\}} \cap \{w_1=1, w_2=\omega = -1\}. \; $
When $\omega = -1$ it follows from (\ref{sysd}) that $d_2=a_2=0$.
The matrix (\ref{smd2a})  reduces to 
\bea\label{smd2b}
\sigma= \bp
1-a_4&0&0&a_4\\
0&1&0&0\\
0&0&1&0\\
d_1&0&0&1-d_1\\
\ep
\eea
Finally, the condition (\ref{ad3}) for this $\sigma$ leads to
$\alpha=\beta=\frac{1}{2}-a_4$, $d_1=1-a_4$. Setting $\theta = a_4 \in [0,1]$
this is the dimer model [DM] in (\ref{dm}). 
The action of $\sigma$ on basic vectors 
$v \otimes w$ and $w \otimes v$ is determined by
\ref{ad2}, \ref{ad3}) and is recorded in Lemma \ref{temp2}. 

$\mathbf{\{D22\}} =  \mathbf{\{D\}} \cap \{w_1=1, w_2=\omega \not \in \{1,-1\}\}. \; $
The solution to (\ref{sysd}) in this case is $a_4=-\frac{2}{1+\omega} a_2$ and $d_1=-\frac{2\omega}{1+\omega} d_2$.
Denoting $a_2=a,d_2=d$, the matrix (\ref{smd2a}) has become
\bea\label{smd2d}
\sigma= \bp
1-\frac{2\omega}{1+\omega}a &a&a&-\frac{2}{1+\omega} a\\
0&1&0&0\\
0&0&1&0\\
-\frac{2\omega}{1+\omega} d&d&d&1-\frac{2}{1+\omega} d\\
\ep.
\eea
Due to the positivity of $\sigma_{12}, \sigma_{42}$ we must have  
$a \geq 0$, $d \geq 0$. If $\omega > -1$ then $\sigma_{14} \geq 0$ implies that $a=0$.
If $\omega>0$ or $\omega <-1$ then $\sigma_{41} \geq 0$ implies that $d=0$.
Therefore when $\omega \in (0,1) \cup (1,\infty)$ the only possible matrix $\sigma$ is the identity. 

When $\omega \in (-\infty,-1)$ we have 
\bea\label{smd2e}
\sigma= \bp
1-\frac{2\omega}{1+\omega}a &a&a&-\frac{2}{1+\omega} a\\
0&1&0&0\\
0&0&1&0\\
0&0&0&1\\
\ep, \quad \mbox{for $a \geq 0$.}
\eea
Resolving the final condition (\ref{ad3}) we find $ \alpha=\frac{\omega}{2}$, $\beta=\frac{1}{2\omega}$, and 
$a=\frac{1+\omega}{2\omega}$. Setting $\theta = (1+\omega)/2 \omega$ we have
$\theta \in (0,\frac12)$ and $\sigma$ becomes 
the stationary coalescence annihilation model [SCAM] in (\ref{scam}). The action of $\sigma$ becomes 
simpler if we replace $w$ by $\hat{w} = \omega^{-1} w = (2 \theta-1,1)^T$. This leaves
(\ref{ad0},\ref{ad1},\ref{ad2}) unchanged but  
(\ref{ad3}) becomes 
$\sigma (v\otimes \hat{w}+\hat{w} \otimes v) = v \otimes v + \hat{w} \otimes \hat{w}$,
and the full action of $\sigma$ is recorded Lemma \ref{temp2}. 

The final case $\omega \in (-1,0]$ is similar: now (\ref{smd2d}) becomes
\bea\label{smd2f}
\sigma= \bp
1 &0&0&0 \\
0&1&0&0\\
0&0&1&0\\
-\frac{2\omega}{1+\omega} d&d&d&1-\frac{2}{1+\omega} d\\
\ep, \quad \mbox{for $d \geq 0$.}
\eea
Resolving the final condition (\ref{ad3}) we find $ \alpha=\beta=\frac{1}{2}$, and 
$d=\frac{1+\omega}{2}$. Setting $\theta = (1+\omega)/2$ we have
$\theta \in (0,\frac12]$ and $\sigma$ becomes the particle-hole conjugate of
the stationary coalescence annihilation model [SCAM] in (\ref{scam}). Note the value $\theta = 1/2$ 
arises, and at the value $\theta =0$ this model is the same as the dimer model [DM] in (\ref{dm}).
This completes the classification in Lemma \ref{temp2}.
\subsection{Proof of Theorem \ref{T2}} \label{s5.3}
It is straightforward to check that the list of models in Theorem \ref{T2}, and their conjugates under
$\rho$, $\tau$ or $\rho \tau$, are all projectors, 
satisfy the braid relation (\ref{braidreln}), and satisfy the stochasticity relation
(\ref{stochasticity}). Moreover the models (although not all their conjugates) satisfy one of the two
Ansatzes (\ref{ansatze}). The logic of the proof therefore is to show that no other solutions are possible. 

The braid relation is also found in the theory of quantum spin systems, and in
quantum inverse scattering, under the guise of R matrices (see Section \ref{s4.4} for
a discussion on the addition of a spectral parameter via Baxterisation and the Yang-Baxter equation).
We can recast Theorem \ref{T2} as a hunt for R matrices with certain properties. 
We give here a brief reminder of the basic definitions and notation regarding R matrices
(further details are in any review of the theory of quantum inverse scattering method or of the theory of 
quantum groups, e.g. \cite{korepin1997quantum}). 
Let $V$ be a linear space.
We are interested in solutions of the  R matrix braid relation 
\begin{equation}
\label{YB}
R_{12} R_{23} R_{12} = R_{23} R_{12} R_{23}
\end{equation}
where $R_{12}= R \otimes I_V$, $R_{23}= I_V\otimes R$, $I_V$ is the identity operator on $V$, and $R \in {\rm End}(V^{\otimes 2})$. An operator $R$ satisfying  (\ref{YB}) is called an {\em R-matrix}.
Usually one demands the R matrix to be an invertible operator, $R\in {\rm Aut}(V^{\otimes 2})$; however this 
restriction is inappropriate in our applications to stochastic processes.  An R matrix which has quadratic 
minimal polynomial 
\begin{equation}
\label{Hecke}
(R - q\, I_{V^{\otimes 2}})(R - p\, I_{V^{\otimes 2}}) =0, \quad \mbox{where $q \neq 0$ and $p \neq q$, }
\end{equation}
is said to be of {\em Hecke type}. Condition (\ref{Hecke}) is called the {\em quadratic relation}.

To link the study of Hecke type R matrices to our stochastic interpretation we set $V= \R^2$ 
and search for $R$ matrices that also satisfy stochasticity conditions (using the standard basis)
\begin{equation} \label{Rstochasticity}
\left\{ \begin{array}{l}
\mbox{(a) the eigenvalues of $R$ are real and distinct: $p < q \in {\Bbb R}$;} \\
\mbox{(b) off-diagonal components of $R$ are nonnegative;} \\
\mbox{(c) the vector $(1,1,1,1)^T$ is an eigenvector with eigenvalue $q$.}
\end{array} \right. 
\end{equation}
The set of such R matrices is mapped, via the map
\begin{equation}
\label{sigma}
\sigma :={ R - p\, I_{V^{\otimes 2}}\over q-p},
\end{equation} 
onto projectors ($\sigma^2=\sigma$) that satisfy the required properties for Theorem \ref{T2}, namely $\sigma$ satisfies the 
stochasticity condition \ref{stochasticity}, and the family $\{\sigma_n\}$ satisfies the 
deformed braid relation
(\ref{braidreln}) with real parameter $Q=-pq/(q-p)^2$. Moreover the map is onto the set of all such $\sigma$ with 
parameter values $Q  \leq 1/4$. 
Therefore, in the rest of the proof we search for an R matrix $R$ satisfying (\ref{YB}) and (\ref{Hecke}), and the stochasticity 
conditions (\ref{Rstochasticity}) and then we record in Theorem \ref{T2} the corresponding projector $\sigma$ given by (\ref{sigma}).

Denote the entries
\begin{equation}
\label{R}
R=
\begin{pmatrix}
a_1&a_2& a_3 & a_4\\
b_1 &b_2&b_3 & b_4 \\
c_1 & c_2 &c_3 & c_4 \\
d_1 & d_2 & d_3 & d_4
\end{pmatrix}.
\end{equation}
We reduce the classification by allowing conjugation with $\rho$, or $\tau$, or both, leading to
\begin{equation} \label{conjugations}
\rho R \rho=
\begin{pmatrix}
a_1&a_3& a_2 & a_4\\
c_1 &c_3&c_2 & c_4 \\
b_1 & b_3 &b_2 & b_4 \\
d_1 & d_3 & d_2 & d_4
\end{pmatrix}, \quad \tau R \tau=
\begin{pmatrix}
d_4&d_3& d_2 & d_1\\
c_4 &c_3&c_2 & c_1 \\
b_4 & b_3 &b_2 & b_1 \\
a_4 & a_3 & a_2 & a_1
\end{pmatrix},
 \quad \rho \tau R \tau \rho=
\begin{pmatrix}
d_4&d_2& d_3 & d_1\\
b_4 &b_2&b_3 & b_1 \\
c_4 & c_2 &c_3 & c_1 \\
a_4 & a_2 & a_3 & a_1
\end{pmatrix}.
\end{equation}
The stochasticity condition (\ref{Rstochasticity}) (b) demands non-negativity of the off-diagonal components of $R$ and 
condition (c) 
fixes the values of the diagonal components in terms of the others (for example $a_1+a_2+a_3+a_4=q$).
Denote 
$$
Y:= R_{12} R_{23} R_{12} - R_{23} R_{12} R_{23}, \qquad H:=(R - q\, I_{V^{\otimes 2}})(R - p\, I_{V^{\otimes 2}}).
$$
These are, respectively,  $8\times 8$ and $4\times 4$ matrices whose vanishing identifies the Hecke type R matrix 
along with the eigenvalues $q$ and $p$. A symbolic calculation of the 64 entries of $Y$ suggests we need
to make some reductions to make the problem easier. The two Ansatzes (\ref{ansatze}) for the form of $\sigma$ are equivalent to the two cases
\begin{equation} \label{IandII}
 \mathbf{I} =  \{d_2 = d_3 =0\},  \qquad \mathbf{II} =  \{b_4=c_4 =0\}. 
\end{equation}
In both of these cases the entries in $H$ and $Y$ start to factorise, and we can push through a complete 
classification of all solutions to $H=Y=0$. We now detail the logic, which is easiest to follow with a symbolic calculator 
for the entries of $Y$ and $H$ at each step. 

We begin with investigation of the Ansatz $\mathbf{I}$ from (\ref{IandII}). Substituting $R$ with $d_2=d_3=0$ into $H$ we find
$H_{4,2}= d_1 a_2$, $H_{4,3}= d_1 a_3$. These two components vanish in exactly one of the following cases:
\begin{equation} \label{1Aand1B}
 \mathbf{\{A\}}=\{d_1=d_2=d_3=0\}, \qquad\mbox{or}\qquad  \mathbf{\{B\}}=\{d_1> 0;  d_2=d_3=a_2=a_3=0\}. 
\end{equation}
In the case $\mathbf{\{A\}}$ several components of the matrix $Y$ factorize. In particular,
\begin{equation}
\label{A1-3}
Y_{4,7}=b_1(a_4 b_1 + b_3(b_4+c_4)), \;\; Y_{6,3}=b_1 c_1(a_2-a_3), \;\;
Y_{7,4}=-c_1(a_4 c_1+c_2(b_4+c_4)).
\end{equation}
Hence consideration of $\mathbf{\{A\}}$ splits into four subcases:
\begin{eqnarray}
\nonumber
&& \mathbf{\{A1\}} = \mathbf{\{A\}} \cap \{b_1>0; c_1> 0; a_3=a_2; a_4=0\}, \quad
\mathbf{\{A2\}}=\mathbf{\{A\}} \cap \{b_1=c_1=0\}, \nonumber  \\
&& \mathbf{\{A3\}} =\mathbf{\{A\}} \cap \{b_1>0; c_1=0; a_4=0\}, \quad
\mathbf{\{A4\}} =\mathbf{\{A\}} \cap \{c_1>0; b_1=0; a_4=0\}.
\nonumber
\end{eqnarray}
Here the
condition $a_4=0$ in cases $\mathbf{\{A1\}}$, $\mathbf{\{A3\}}$ and $\mathbf{\{A4\}}$ follows from non-negativity of all the coefficients in $Y_{4,7}$ and $Y_{7,4}$ and positivity of either $b_1$ or $c_1$.
The case $\mathbf{\{A4\}}$ after left-right conjugation becomes the case $\mathbf{\{A3\}}$, so that it will
produce exactly the left-right conjugates of any R matrices found there, and we therefore do not not need to
analyse it. 

In the case $\mathbf{\{A1\}}$, the vanishing conditions for $Y_{4,7}$ and $Y_{7,4}$ (see (\ref{A1-3})) split further consideration in two directions: 
\[
\mathbf{\{A11\}}=\mathbf{\{A1\}}\cap\{b_3=c_2=0\}, \quad \mathbf{\{A12\}}=\mathbf{\{A1\}}\cap \{b_3+c_2>0; b_4=c_4=0\}.
\]
Here the inequality $b_3+c_2>0$ implies that at least one of the coefficients $b_3$,  $c_2$ is strictly positive. 

In the case $\mathbf{\{A11\}}$ we have $H_{2,3}= a_2 b_1$ implying that $a_2=0$, and then
\begin{equation} \label{finalA11cases}
Y_{2,1}=-b_1(b_1-q)(b_4-q), \qquad Y_{5,1}=c_1(c_1-q)(c_4-q).
\end{equation}
Since $b_1,c_1>0$ the vanishing of $Y_{2,1}$ and $Y_{5,1}$ implies that
$q \in \{b_1,b_4\} \cap \{c_1,c_4\}$ - four cases. In the case $b_1=c_1=q$ we find, 
using $H_{2,1}= -q(b_4+p)$ and $H_{3,1}=-q(c_4+p)$, that $b_4=c_4=-p$. With these choices
$H=Y=0$ and we are left with the R matrix
\[
R=
\begin{pmatrix}
q&0& 0 & 0\\
q &p&0& -p \\
q & 0 &p & -p \\
0 & 0 & 0 & q
\end{pmatrix}
\]
and the final requirement, for stochasticity, is that $p \leq 0 \leq q$. These all correspond, under the map (\ref{sigma}) to 
the biased voter model [BVM] (\ref{bvm2}) with $\theta = q/(q-p) \in (0,1]$. In the case $b_1=c_4=q$ we find, 
using $H_{2,1}= -q(b_4+p)$ and $H_{3,1}=-c_1(c_1+p)$, that $b_4=c_1=-p$. With these choices 
$H=Y=0$ and we are left with the R matrix
\[
R=
\begin{pmatrix}
q&0& 0 & 0\\
q &p&0& -p \\
-p & 0 &p & q \\
0 & 0 & 0 & q
\end{pmatrix}
\]
and the final requirement, for stochasticity, is that $p \leq 0 \leq q$. These all correspond, under the map (\ref{sigma}) to 
the asymmetric voter model [AVM] (\ref{avm}) with $r = q/(q-p) \in (0,1)$. In a similar way, the case $b_4=c_1=q$ leads again to the 
asymmetric voter model with $r \in (0,1)$, and the case $b_4=c_4=q$ to the 
biased voter model with $\theta \in (0,1)$. 

In the case $\mathbf{\{A12\}}$ we have 
$$
Y_{1,4}=-2 a_2^2(a_2-c_2), \qquad Y_{1,6}=a_2^2(b_1-c_1),\qquad Y_{1,7}=2 a_2^2(a_2-b_3).
$$
Hence $\{A12\}$ splits into subcases
$$
\mathbf{\{A121\}}=\mathbf{\{A12\}} \cap\{a_2=0\}, \qquad \mathbf{\{A122\}}=\mathbf{\{A12\}}\cap\{a_2>0; b_3=c_2=a_2; c_1=b_1\}.
$$
In the subcase $\mathbf{\{A121\}}$  the components
$
Y_{4,1}=b_1 b_3(b_1+c_1)$ and $Y_{7,1}=-c_1 c_2 (b_1+c_1)$ cannot vanish simultaneously 
due to the restrictions $b_1, c_1 >0$, $b_3+c_2>0$, so 
this subcase gives no solutions. 
In the subcase $\mathbf{\{A122\}}$ we have
$$
Y_{1,2}=-a_2(a_2-q)(a_2+b_1-q), \qquad H_{1,1}=2 a_2 (2 a_2+b_1+p-q).
$$
Conditions $Y_{1,2}=H_{1,1}=0$ are fulfilled exactly when either $a_2=q,  b_1=-p-q$, or $a_2=-p, b_1=q+p$. 
With either of these two choices $H=Y=0$ and the corresponding two R matrices are
\[
R=
\begin{pmatrix}
-q&q& q & 0\\
-q-p &q+p&q& 0 \\
-q-p & q &q+p & 0 \\
0 & 0 & 0 & q
\end{pmatrix} \qquad 
R=
\begin{pmatrix}
q+2p&-p& -p & 0\\
q+p &0&-p& 0 \\
q+p & -p &0 & 0 \\
0 & 0 & 0 & q
\end{pmatrix}.
\]
These all correspond, under the map (\ref{sigma}) to 
the coalescing symmetric random walks with branching [CSRWB] (\ref{csrwb2}) with first $\theta = q/(q-p)>0$ and second $\theta = -p/(q-p)>0$. The stochasticity conditions require that $\theta \leq \frac12$.

Passing to consideration of the case $\mathbf{\{A2\}}$ we find new factorising components in $Y$:
\begin{equation}
\label{A21}
Y_{2,3}=-b_3(b_3+b_4-q)(c_2+c_4-q), \quad Y_{3,2}=-c_2(b_3+b_4-q)(c_2+c_4-q).
\end{equation}
Notice that if $c_2+b_3>0$ the vanishing of the components $Y_{2,3}$, $Y_{3,2}$ in (\ref{A21}) implies either $b_3+b_4-q=0$ or $c_2+c_4-q=0$. We split the investigation therefore into 
$$
\mathbf{\{A21\}}=\mathbf{\{A2\}}\cap\left\{\!\!\begin{array}{c} c_2+ b_3 >0 \\c_4=q-c_2 \end{array}\!\!\right\}, \quad
\mathbf{\{A22\}}=\mathbf{\{A2\}}\cap\{c_2=b_3=0\}, 
\quad \mathbf{\{A23\}}= \mathbf{\{A2\}}\cap\left\{\!\!\begin{array}{c} c_2+ b_3 >0 \\b_4=q-b_2\end{array}\!\!\right\}.
$$ 
The cases $\mathbf{\{A21\}}$ and $\mathbf{\{A23\}}$ are mapped to each other by left-right conjugation,
so we do not need to analyse $\mathbf{\{A23\}}$.
%
%
In the case $\mathbf{\{A21\}}$ the following components of $Y$ assume factorised form:
$$
Y_{3,7}=b_3 (c_2-q)  (a_2+a_4+b_3-q), \quad Y_{5,6}=-c_2(c_2-q) (a_2+a_4+b_3-q),
$$
and we split the case $\mathbf{\{A21\}}$ into
$$
\mathbf{\{A211\}}=\mathbf{\{A21\}}\cap\{c_2=q\}, \quad \mathbf{\{A212\}}=\mathbf{\{A21\}}\cap\{c_2\neq q; a_4=q-a_2-b_3\}.
$$
For the case $\mathbf{\{A211\}}$ we have $H_{3,3}=q (b_3 + p)$ and $H_{3,4}= q b_4$ 
so that $b_3=-p$ and $b_4=0$. 
We then use $ H_{1,1}=(a_2 + a_3 + a_4) (a_2 + a_3 + a_4 + p - q)$ so that either (i)
$a_2=a_3=a_4=0$, in which case we already have $H=Y=0$ and the R matrix
\[
R=
\begin{pmatrix}
q&0& 0 & 0\\
0&q+p&-p& 0 \\
0& q&0 & 0 \\
0 & 0 & 0 & q
\end{pmatrix}
\]
which, with $p \leq 0 \leq q$ for stochasticity,  
corresponds under the map (\ref{sigma}) to the asymmetric exclusion process [ASEP] (\ref{asep})
with $l = q/(q-p) \in (0,1]$; or (ii)
we have $a_4=q-p-a3-a2$ and we need to use $H_{1,2}=p a_2  + q a_3=0$, whereupon setting
$a_3=- p a_2/q$ we again find $H=Y=0$ and the R matrix
\[
R=
\begin{pmatrix}
p& a_2 & -pa_2/q & (q-a_2)(q-p)/q\\
0&q+p&-p& 0 \\
0& q&0 & 0 \\
0 & 0 & 0 & q
\end{pmatrix}
\]
which, with stochasticity requiring $p \leq 0 \leq q$ and $0 \leq a_2 \leq q$, corresponds under the map (\ref{sigma}) 
to the annihilating-coalesing random walks [ACRW] (\ref{acrw})
with the choices $l = q/(q-p) \in (0,1]$ and $\theta = (q-a_2)/q \in [0,1]$. (The missing values where $l=0$ for the above two
models appear, as explained above, in the conjugate cases  $\rho \sigma \rho$ from $\mathbf{\{A23\}}$.)

Investigation of the case $\mathbf{\{A212\}}$ does not  give any non-zero stochastic 
R matrices. Indeed, vanishing of $H_{2,3} = - b_3(p+b_3+b_4)$ and $H_{3,2} = -c_2(p+b_3+b_4) $ together 
leads to the condition $b_4=-p-b_3$. If then $c_2=0$, so that $b_3>0$, we find $H(2,2)=pq$ forcing $p=0$, which 
makes $b_4=-p-b_3 = -b_3$ contradicting stochasticity. If however $c_2 \neq 0$ we may use
\[
H(2,2) = b_3c_2 + pq, \quad H(1,2)= a_3c_2-a_2(a_3-b_3), \quad H(1,1) = (a_3-b_3+p)(a_3-b_3+q)
\]
to express $b_3$, $a_2$ and $a_3$ in terms of $c_2,p,q$. 
The two solutions (i) $b_3=-pq/c_2; \, a_2=c_2+p; \, a_3 = -q-(pq)/c_2$ and (ii) $b_3=-pq/c_2; \, a_2=c_2+q; \, a_3 = -p-(pq)/c_2$ 
do then yield families of R matrices, that is $H=Y=0$, but the stochasticity conditions eliminate all but the trivial zero matrix;
the final solution (iii) $b_3=a_3=a_2=p=0$ yields a Hecke type R matrix only when $c_2=-q$, and this also fails the stochasticity
condition unless it is trivial. 

We now consider the case $\mathbf{\{A22\}}$. Here we observe following factorisations
$$
Y_{4,4}=q b_4(b_4-q), \qquad Y_{7,7}=-q c_4(c_4-q).
$$
Therefore we split further consideration into the subcases
\begin{eqnarray}
\nonumber
\mathbf{\{A221\}}=\mathbf{\{A22\}}\cap\{b_4=0; c_4=q\}, && \mathbf{\{A222\}}=\mathbf{\{A22\}}\cap\{b_4=c_4=q\},
\\
\nonumber
\mathbf{\{A223\}}=\mathbf{\{A22\}}\cap\{b_4=c_4=0\}, && \mathbf{\{A224\}}=\mathbf{\{A22\}}\cap\{b_4=q; c_4=0\}.
\end{eqnarray}
%
In the case $\mathbf{\{A221\}}$ we have $H_{3,3}= pq$ implying  $p=0$. Then
$$
H_{1,1}=(a_2+a_3+a_4)(a_2+a_3+a_4-q)
$$
and we end up with two possibilities:  the first possibility  is $a_2=a_3=a_4=0$ whereupon
$H=Y=0$ and yielding the R matrix
\[
R=
\begin{pmatrix}
q& 0&0&0\\
0&q&0& 0 \\
0& 0&0 & q \\
0 & 0 & 0 & q
\end{pmatrix}
\]
which corresponds under the map (\ref{sigma}) 
to the first erosion model [EM1] (\ref{em1}).
The other possibility for $H_{1,1}=0$ is to put $a_4=q-a_2-a_3$. In this case $Y_{1,4}= -q(a_2-q)^2$ and $Y_{1,7}= a_3^2q$  which forces $a_2=q$ and $a_3=0$, and we find $H=Y=0$ and obtain the the R matrix
\[
R=
\begin{pmatrix}
0& q&0&0\\
0&q&0& 0 \\
0& 0&0 & q \\
0 & 0 & 0 & q
\end{pmatrix}
\]
which corresponds under the map (\ref{sigma}) 
to the second erosion model [EM2] (\ref{em2}). The case $\mathbf{\{A224\}}$
corresponds to the left right conjugate of case $\mathbf{\{A221\}}$, and so leads only to
the the conjugates $\rho \sigma \rho$ of the two erosion models above.

In the case $\mathbf{\{A222\}}$ we have $H_{2,2}= pq$ implying  $p=0$.
Then 
\[
H_{1,2}=-a_2(a_2+a_3+a_4), \quad H_{1,3}=-a_3(a_2+a_3+a_4), \quad H_{1,4}=(q-a_4)(a_2+a_3+a_4).
\]
For these three entires to vanish we must have $a_2=a_3=0$ and $a_4\in\{0,q\}$. The choice $a_4=q$ 
leads to $H=Y=0$ and we find the R matrix
\[
R=
\begin{pmatrix}
0& 0&0&q\\
0&0&0& q \\
0& 0&0 & q \\
0 & 0 & 0 & q
\end{pmatrix}
\]
which corresponds under the map (\ref{sigma}) 
to the reshuffle model [RM] (\ref{rm}) with parameters $\theta_1=\theta_2=0, \theta_3 =1$. 
The choice $a_4=0$ leads to $H=Y=0$ and we find the R matrix
\[
R=
\begin{pmatrix}
q& 0&0&0\\
0&0&0& q \\
0& 0&0 & q \\
0 & 0 & 0 & q
\end{pmatrix}
\]
which corresponds under the map (\ref{sigma}) 
to the particle-hole conjugation of the coalescing symmetric random walks with branching
[CSRWB] \ref{csrwb2} with the parameter $\theta =0$  (that is a purely branching model).  

The case $\mathbf{\{A223\}}$ produces no non-trivial solutions.
 Indeed we have
$$
H_{1,1}=(a_2+a_3+a_3)(p-q+a_2+a_3+a_4), \qquad
Y_{2,2}=q(a_2+a_3+a_3)(-q+a_2+a_3+a_4)
$$
which vanishing only if either $a_2=a_3=a_4=0$, or if $p=0$ and $a_4=q-a_2-a_3$.
In the first case we find $H=Y=0$ but we get the trivial R matrix $R=q I$ corresponding
to $\sigma=I$. In the second case from the vanishing of
$Y_{1,3}= q^2(a_2-a_3)$ and $Y_{1,4}= -q(a_2^2-q(a_2+a_3-q))$ imply
$a_2=a_3=q$, and the resulting R matrix does not satisfy the stochasticity conditions since $a_4=-a_3=-q \neq 0$.

Consider now the case $\mathbf{\{A3\}}$. Vanishing of the component $H_{3,1}=b_1 c_2$ implies $c_2=0$, whereas vanishing of $Y_{4,7} = b_1b_3(b_4+c_4)$ splits further consideration in two directions:
$$
\mathbf{\{A31\}}=\mathbf{\{A3\}}\cap\{c_2=0; b_3=0\},\quad 
\mathbf{\{A32\}}=\mathbf{\{A3\}}\cap\{c_2=0; b_3>0; b_4=c_4=0\}.
$$
In the case $\mathbf{\{A31\}}$ we have $H_{2,3}=b_1 a_3$ and $ Y_{4,1}=- b_1^2 (a_2+a_3)$
which forces $a_2=a_3=0$. Then $H_{2,1}=-b_1(b_1+b_4+p-q)$ and we set $b_4=q-p-b_1$.
With these substitutions we find
$$
Y_{2,1}=b_1(b_1+p)(b_1-q), \;\; Y_{3,1}=-b_1(c_4(b_1+p-q)-pq), \;\;
H_{3,3}=c_4(c_4+p-q)
$$
Then, vanishing of these three components implies (using $b_1>0$) yields two solutions: (i)
$p=0$ and $c_4=b_1=q$ which leaves $H=Y=0$ and we find the R matrix
\[
R=
\begin{pmatrix}
q& 0&0&0\\
q&0&0& 0 \\
0& 0&0 & q \\
0 & 0 & 0 & q
\end{pmatrix}
\]
which corresponds under the map (\ref{sigma}) 
to the totally asymmetric voter model [AVM] (\ref{avm}), that is with values $l=0, r=1$; or (ii)
$c_4=p=0$ and $b_1=q$ which leaves $H=Y=0$ and we find the R matrix
\[
R=
\begin{pmatrix}
q& 0&0&0\\
q&0&0& 0 \\
0& 0&q& 0 \\
0 & 0 & 0 & q
\end{pmatrix}
\]
which corresponds under the map (\ref{sigma}) to the particle-hole conjugate of
the erosion model [EM1] (\ref{em1}).

In the case $\mathbf{\{A32\}}$ we do not find any new solutions. Indeed, simultaneous
vanishing of 
\[
Y_{1,1} = b_1a_2(a_2+a_3) \quad Y_{2,5}= b_1a_3(b_1+b_3), \quad  Y_{4,1} = b_1^2(b_3-a_2-a_3)
\]
implies $a_2=a_3=b_3=0$ which contradicts the assumption $b_3>0$.

Let us pass to the case $\mathbf{\{B\}}$. We split this into two parts depending on the value of $a_4$: 
\[
\mathbf{\{B1\}} = \{d_1>0;a_4>0; d_2=d_3=a_2=a_3=0\} \quad \mathbf{\{B2\}} = \{d_1>0;a_4=0; d_2=d_3=a_2=a_3=0\}.
\]
After particle-hole conjugation part $\mathbf{\{B2\}}$ becomes $\{d_1=0;a_4>0; d_2=d_3=a_2=a_3=0\}$, which is a 
subset of $\mathbf{\{A\}}$ and already explored, and we need not analyse it.

In the case  $\mathbf{\{B1\}}$ we find the components
\begin{equation} \label{Y1363}
Y_{1,3}= -a_4(b_1 b_3-c_1 c_2 ), \qquad Y_{6,3}=-d_1 (b_4 b_3-c_4 c_2),
\end{equation}
which suggests a further decomposition into the following cases
\begin{equation}
\label{B2}
\mathbf{\{B11\}}=\mathbf{\{B1\}}\cap\{c_2=b_3=0\}, \;\;
\mathbf{\{B12\}}=\mathbf{\{B1\}}\cap\{c_2>0\}, \;\;
\mathbf{\{B13\}}=\mathbf{\{B1\}}\cap\{c_2=0; b_3>0\}.
\end{equation}
The case $\mathbf{\{B13\}}$ is left-right conjugate to the part of $\mathbf{\{B12\}}$ where $b_3=0$
and need not be analysed. First consider the case $\mathbf{\{B11\}}$. The vanishing of
$Y_{3,2} = a_4 b_1^2 + d_1b_4c_4$ and 
$Y_{3,5} = -a_4 c_1^2 - d_1 b_4c_4$ implies that $b_1=c_1=0$. 
The vanishing of
$Y_{6,4}= -a_4b_1c_1-d_1b_4^2$ and 
$Y_{6,7} = a_4b_1c_1+d_1c_4^2$ implies that $b_4=c_4=0$. 
Then the components
$$
Y_{1,7}=a_4((a_4-q)^2+q(a_4-d_1)) \quad\mbox{and}\quad
Y_{2,8}=a_4((d_1-q)^2-q(a_4-d_1))
$$
can only vanish if $a_4=d_1=q$. Finally $H_{1,1} = q^2+pq$ which vanishes only if $p=-q$. 
This leaves $H=Y=0$ and we find the R matrix
\[
R=
\begin{pmatrix}
0& 0&0&q\\
0&q&0& 0 \\
0& 0&q & 0 \\
q & 0 & 0 & 0
\end{pmatrix}
\]
which corresponds under the map (\ref{sigma}) 
to the dimer model [DM] (\ref{dm2}) with values $\theta=1/2$.

Now we investigate the case $\mathbf{\{B12\}}$. From the vanishing of the two entries (\ref{Y1363}) 
and $c_2,a_4,d_1>0$ we may introduce $x = b_3/c_2 \geq 0$ and solve for $c_1= xb_1$ and $c_4=xb_4$.
Then we find the components
$$
Y_{1,6}=-c_2 a_4(x-1)(a_4-c_2(1+x)), \qquad
Y_{8,3}=c_2 d_1(x-1)(d_1-c_2(1+x)),
$$
which vanish exactly in one of the following cases:
$$
\mathbf{\{B121\}}=\mathbf{\{B12\}} \cap\{x=1\}, \quad
\mathbf{\{B122\}}=\mathbf{\{B12\}}\cap\{x\neq 1; d_1=a_4=c_2(1+x)\}.
$$
The case $\mathbf{\{B122\}}$ produces no R matrices: the vanishing of the components
\begin{eqnarray}
\nonumber
Y_{1,1}&=&c_2(x-1)(x+1)^2((b_1-c_2)^2+c_2(b_1-b_4)),
\\
\nonumber
Y_{8,8}&=&-c_2(x-1)(x+1)^2((b_4-c_2)^2-c_2(b_1-b_4))
\end{eqnarray}
is only possible if $b_1=b_4=c_2$. Then the condition $Y_{1,2}=c_2^3(1+x)^3=0$
cannot be satisfied. 

In the case $\mathbf{\{B121\}}$ we have $x=1$ so that $b_3=c_2,\, c_1=b_1$ and $c_4=b_4$. 
The vanishing of $Y_{2,5}=-a_4 b_1^2-2 c_2 b_1 b_4- d_1 b_4^2$ is only possible if 
$b_1=b_4=0$. 
The vanishing of the coefficients
\[
H_{2,2} = c_2(2c_2+p-q), \quad H_{1,1} = a_4(a_4+d_1 +p-q), \quad 
Y_{2,3} = c_2(a_4d_1 - (c_2-q)^2) 
\]
allow us to solve for the entries $c_2,a_4,d_1$ in terms of $p,q$. Indeed we must have
\[
c_2 = (q-p)/2, \quad a_4+d_1 = q-p, \quad a_4d_1 = (p+q)^2/4.
\]
Since $c_2>0$ we must have $p<q$. Also
$(a_4-d_1)^2 = (a_4+d_1)^2 - 4 a_4 d_1 = -4pq$ so that $p \leq 0 < q$. It is convenient
to set $p = -t^2 q$ for $t \geq 0$ which leads to
\[
c_2 = q(1+t^2)/2, \quad a_4 = \frac{q}{2} (1+t^2 \pm 2t), \quad  d_1 = \frac{q}{2} (1+t^2 \mp 2t).
\]
With these substitutions we find $H=Y=0$ and the corresponding R matrix under the map (\ref{sigma}) becomes
\[
\sigma=
\begin{pmatrix}
\frac12 \pm \frac{t}{1+t^2} & 0&0&\frac12 \mp \frac{t}{1+t^2} \\
0&\frac12&\frac12& 0 \\
0& \frac12&\frac12 & 0 \\
\frac12 \pm \frac{t}{1+t^2}  & 0 & 0 & \frac12 \mp \frac{t}{1+t^2} 
\end{pmatrix}.
\]
As $t$ varies over $[0,\infty)$ we get the model [ASRWPI] \ref{asrwpi2} with all parameters $\theta \in [-1/2,1/2]$. 
%
%
%
%
%

Next we investigate the Ansatz (\ref{ansatze}b). We may assume that the matrix 
$R$ (\ref{R}) does not fit form of the Ansatz (\ref{ansatze}a), that is $d_2+d_3>0$ (or else
we are re-searching cases already explored). Noticing that the particle-hole 
conjugation transposes pairs $a_2\leftrightarrow d_3$ and $a_3\leftrightarrow d_2$ we also assume $a_2+a_3>0$.

Substituting $R$ with $b_4=c_4=0$ into $H$ we find: $H_{2,4}= a_4 b_1$, $H_{3,4}=  a_4 c_1$. These two components 
vanish in either of the following cases:
\begin{eqnarray}
\nonumber
\mathbf{\{C\}}&=&\{ d_2+d_3>0; a_2+a_3>0; b_4=c_4=a_4=0\},
\\
\nonumber
\mathbf{\{D\}}&=&\{ d_2+d_3>0; a_2+a_3>0; a_4>0;b_4=c_4=b_1=c_1=0\}.
\end{eqnarray}
In the case $\mathbf{\{C\}}$ analysing the components
$$
Y_{4,7}=-a_3(a_3 d_1+b_3(d_2+d_3)), \qquad Y_{7,4}=a_2(a_2 d_1+c_2(d_2+d_3))
$$	
we find that when they 
vanish then either $a_2>0$ and $d_1=c_2=0$, or $a_3>0$ and $d_1=b_3=0$.
These two cases are related by left-right conjugation
and hence, it is enough investigating one of them, say, the case
$$
\mathbf{\{C'\}}=\mathbf{\{C\}}\cap\{a_2>0;d_1=c_2=0\}.
$$
In this case we have $H_{32}=a_2 c_1 \; \Rightarrow \; c_1=0$, and $Y_{3,4}= a_2^2 b_1\; \Rightarrow\; b_1=0$. Inspecting the components
$$
H_{2,2}=b_3(b_3+p-q), \quad \mbox{and} \quad H_{1,2}=-a_2(a_2+a_3+b_3+p-q)
$$
we conclude that $b_3=0$ (setting $b_3=q-p$ is not possible since $a_2(a_2+a_3)\neq 0$). Vanishing of the components
$$
H_{1,1}=(a_2+a_3)(a_2+a_3+p-q), \qquad H_{4,4}=(d_2+d_3)(d_2+d_3+p-q)
$$
implies $a_3=q-p-a_2$, $d_3=q-p-d_2$. Then, the condition $Y_{1,2}=-a_2(a_2 d_2+ p q)=0$ gives $d_2=-p q/a_2$. Finally, the component
$Y_{5,6}=(q-p)(a_2+p)(a_2-q)$ vanishes if either $a_2=q$, or $a_2=-p$. In  both cases $H=Y=0$ and we find the R matrices
\[
R=
\begin{pmatrix}
p& q&-p&0\\
0&q&0& 0 \\
0& 0&q & 0 \\
0 & -p & q & p
\end{pmatrix}, \qquad 
R=
\begin{pmatrix}
p& -p&q&0\\
0&q&0& 0 \\
0& 0&q & 0 \\
0 & q & -p & -p
\end{pmatrix}.
\]
These are related by left-right conjugation. After applying the stochasticity conditions they 
correspond under the map (\ref{sigma}) to the asymmetric anti-voter model [AAVM] \ref{aavm}
for all parameters $l = 1-r \in [0,1]$. 

The case $\mathbf{\{D\}}$ does not give any solutions. Indeed, we have
$Y_{3,8}= a_4(d_1+d_2+d_3)(c_2-b_3)=0$ implying that $b_3=c_2$. Using this we have
$Y_{8,8}= a_4(d_2+d_3)(d_2-d_3)=0 $ implying, since $d_2+d_3>0$, that $d_3=d_2>0$. Using this we have 
$Y_{1,6}=a_4 d_2(a_3-a_2)=0$ implying $a_3=a_2$. Finally we have
$ Y_{2,5}=a_2^2 d_1+2 a_2 c_2 d_2+a_4 d_2^2 \geq a_4 d_2^2 >0$ showing there are 
no solutions. 

We have completed the exhaustive search. 
\section{Acknowledgements.} The work of AP and  PP was supported by the Russian Foundation of Basic Research within the grant $20-51-12005$. OZ is grateful for the hospitality of the 
Bogoliubov Laboratory of Theoretical Physics of the Joint Institute for Nuclear Research
where part of the research has been carried out. 
\begin{appendix}
\section{Appendix}\label{app7}
We review here  (i) some different presentations of finite dimensional Hecke algebras; (ii) the construction of an infinite dimensional Hecke Algebra $\mathbb{H}_{\infty}$ as a direct limit.

For our purposes it is sufficient to define $\mathbb{H}_n$, a (type-A) Hecke algebra, as a
unital associative algebra
over $\R$ generated by $s_1, s_2, \ldots, s_{n-1}$ subject to the following relations:
\bea\label{hecke_sym}
\left\{
\begin{array}{ll}
s_{i}s_j=s_j s_i,&|i-j|\geq 2,~1\leq i,j\leq n-1\\
s_{i}s_{i+1}s_{i}
=s_{i+1}s_i s_{i+1},&1\leq i\leq n-2\\
s_i^2=1+(q-q^{-1})s_i,& 1\leq i\leq n-1
\end{array}
\right.
\eea
where $q \in \R\setminus \{0\}$ is a real parameter and $1$ is the unital element.
Due to the isomorphism of the Hecke algebras corresponding to $q$, $-q$ and $q^{-1}$, we
can assume that $q \in (0,1]$.
 The second relation
of (\ref{hecke_sym}) is called the braid relation, the third is called the quadratic, or Hecke, relation. 
We see that $\mathbb{H}_n$ is a quotient of Artin's braid group algebra by the quadratic relation.
The case $q=1$ corresponds to the group algebra of the symmetric group $S_{n}$ and therefore
the Hecke algebra $\mathbb{H}_n$ can also be regarded as a deformation of the group algebra
of $S_n$. In a more abstract setting, see \cite{pyatov} for a review, 
$q$ is treated as a formal parameter and 
$\mathbb{H}_n$ is defined as an algebra over the ring $\R[q,q^{-1}]$
of Laurent polynomials in $q$. The
definition we are using in this paper corresponds to the {\em specialisation} of 
a Hecke algebra over $\R[q,q^{-1}]$ obtained by assigning a numerical value to $q$. 

In the paper we also use two alternative sets of generators: define
\bea\label{remap}
\sigma_i = \frac{q-s_i}{q+q^{-1}} \quad \mbox{and} \quad 
q_i = \sigma_i-1=-\frac{q^{-1}+s_i}{q+q^{-1}}, \quad \mbox{for $1\leq i \leq n-1$.} 
\eea
It is straightforward to check that
\bea\label{hecke_stoch}
\left\{
\begin{array}{ll}
\sigma_{i}\sigma_j=\sigma_j \sigma_i,&|i-j|\geq 2,~1\leq i,j\leq n-1\\
\sigma_{i}\sigma_{i+1}\sigma_{i}-Q\sigma_i
=\sigma_{i+1}\sigma_i \sigma_{i+1}-Q\sigma_{i+1},&1\leq i\leq n-2\\
\sigma_i^2=\sigma_i,& 1\leq i\leq n-1
\end{array}
\right. 
\eea
and
\bea\label{hecke_markov}
\left\{
\begin{array}{ll}
q_{i} q_j=q_j q_i,&|i-j|\geq 2,~1\leq i,j\leq n-1\\
q_{i} q_{i+1} q_{i}-Q q_i
= q_{i+1} q_i q_{i+1}-Q q_{i+1},&1\leq i\leq n-2\\
q_i^2=-q_i,& 1\leq i\leq n-1
\end{array}
\right.
\eea
where 
\bea\label{constQ}
Q=\frac{1}{(q+q^{-1})^2},~q=\frac{Q^{-1}-2}{2}-\sqrt{\left(\frac{Q^{-1}-2}{2}\right)-1},
\eea
where we used the restriction $q\in (0,1]$ to invert the relation $Q=(q+q^{-1})^{-2}$.
We will refer to the third
order relation in (\ref{hecke_stoch}, \ref{hecke_markov}) as the deformed braid
relation.

$\mathbb{H}_n$ can be generated by any of the sets
$\{s_i\}_{i=1}^{n-1}$, $\{\sigma_i\}_{i=1}^{n-1}$ or $\{q_i\}_{i=1}^{n-1}$.
We refer to $s_i$'s as Hecke generators, $\sigma_i$'s as stochastic generators,
$q_i$'s as Markov generators.  
The classification lemmas are proved in terms of stochastic generators.  Markov generators are 
natural for building generators of Markov chains as discussed in the paper.

The current paper deals with Markov particle systems on $\Z$. We consider their 
infinitesimal generators in an algebraic framework, and we wish here simply to show that
the expression $\sum_{i \in \Z} q_i$ can be thought of as lying in 
a closure of a representation of an infinite dimensional algebra $\mathbb{H}_\infty$ 
in $\text{End}(\tf)$. Though we do not make use of the construction, we shall define
$\mathbb{H}_\infty$ and construct the corresponding closure.

For any $m,n\in \N$ such that $m\leq n$ let $e_{n,m}:\mathbb{H}_m\rightarrow \mathbb{H}_n$
be the natural embedding of $\mathbb{H}_m$ into $\mathbb{H}_n$ as a sub-algebra generated 
by $\{s_i\}_{i=1}^{m-1}\subset \mathbb{H}_n $. Then $\{e_{n,m}\}_{1\leq m\leq n<\infty}$
is a set of algebra homomorphisms: $e_{n,n}=\text{Id}$, $e_{nm}\circ e_{mk}=e_{nk}$
for all $1\leq k\leq m\leq n$. So $(\mathbb{H}_n,e_{nm})_{1\leq m\leq n<\infty}$ is 
a direct system over the set of natural numbers. 
Therefore one can define $\mathbb{H}_\infty$ as a direct limit:
\bea
\mathbb{H}_\infty:=\varinjlim \mathbb{H}_n=\left(\bigsqcup_{i\in\N}\mathbb{H}_n\right)\bigg/ \thicksim,
\eea
where $\bigsqcup_{i\in\N}\mathbb{H}_n$ is the 
disjoint union and the equivalence relation $\thicksim$ is 
$a_i \sim a_j$, for $a_i \in \mathbb{H}_i$, $a_j \in \mathbb{H}_j$, if there is $k\in \N:$
$e_{kj}(a_j)=e_{ki}(a_i)$. Choosing a standard basis for each $\mathbb{H}_n$
one can describe elements of $\mathbb{H}_\infty$ as finite linear combinations
of words $\{s_{2i_2}s_{3i_3}\ldots s_{n i_n}: 1 \leq i_k\leq k\leq n, n\in \N\}$,
where
\bea
s_{ji}=\left\{\begin{array}{ccc}
1&\text{ if }&i=j,\\
s_{j-1}s_{j-2}\ldots s_i&\text{ if }&j>i.
\end{array}
\right.
\eea
Let $\pi: \mathbb{H}_\infty\rightarrow \text{End}(\tf)$ be a representation of $\mathbb{H}_\infty$
in the space of test functions.
We suppose, as in our concrete application in this paper, that for each $n \in \Z$, $\pi(q_n)$ 'acts' only on the co-ordinates $(\eta_n,\eta_{n+1})$ and also annihilates constant functions.
We equip $\tf$ with the supremum norm,
\bea 
f\in \tf \mapsto ||f||_\infty=\sup_{\eta \in \Omega}|f(\eta)|.
\eea
%
We want to consider an infinite sum $\sum_i q_i$, and for this we say that a sequence 
$(\rho_n)_{n \in \N} \subset \pi(\mathbb{H}_\infty)\subset \text{End}(\tf)$ 
converges pointwise 
if for any $f\in \tf$ the sequence $(\pi(\rho_n) f)_{n\in \N}\subset \tf$ converges
in $(\tf, \|\cdot\|_{\infty})$. Denote the limit by $\rho f$, where $\rho \in \text{End}(\tf)$
is a linear operator.
Finally, let $\overline{\pi(\mathbb{H}_\infty)}$ be the closure of $\pi(\mathbb{H}_\infty)$
in $\text{End}(\tf,\tf)$
obtained by adding all the limiting points.

An example of a limiting point is given by 
\beast
\sum_{k\in \Z}\rho(q_n):=\lim_{M\rightarrow +\infty}\sum_{k=0}^M\rho(q_n)
+\lim_{K\rightarrow - \infty} \sum_{k=K}^{-1}\rho(q_n).
\eeast
Indeed, take any $f\in \tf$. 
For large $|n|$ we have $\pi(q_n)f=0$, which proves the convergence
of $\sum_{k\in \Z}\rho(\ell_n) f$. All Markov chain generators we consider
are elements of $\overline{\pi(\mathbb{H}_\infty)}$ of this type. 
For most of the paper we are dealing with a fixed representation of Hecke algebra.
Accordingly, we replace $\pi(q_n)$ with $q_n$ everywhere
where it cannot lead to a  confusion.\\ 
\end{appendix}
%
\bibliographystyle{abbrv}
\bibliography{Hecke_final}
\end{document}